%% file: gammalim_full.tex
\documentclass[11pt]{article}

\usepackage[utf8]{inputenc}	
\usepackage[T1]{fontenc}
\usepackage[ngerman,english]{babel}	
\usepackage[a4paper,DIV=11]{typearea}
\usepackage{parskip}
\usepackage{amsmath}		
\usepackage{amsfonts}		
\usepackage{amssymb}		
\usepackage{amsthm}		  
\usepackage{mathtools}  
\usepackage{xspace}		
\usepackage{color}      
\usepackage{pst-plot, pstricks}   
\usepackage{subfig}
\usepackage{graphicx}
\usepackage{booktabs}
\usepackage{pifont}
\usepackage{fancyhdr}
\usepackage{verbatim}

\newcommand{\ES}{E_\text{sym}}
\newcommand{\EA}{E_\text{asym}}
\newcommand{\nd}{\noindent}

\newcommand{\C}[1][]{\ensuremath{\mathbb{C}^{#1}}\xspace}	
\newcommand{\R}[1][]{\ensuremath{\mathbb{R}^{#1}}\xspace}	
\newcommand{\N}[1][]{\ensuremath{\mathbb{N}^{#1}}\xspace}	

\newcommand{\beq}{\begin{equation}}
\newcommand{\eeq}{\end{equation}}

\newcommand{\defas}{\mathrel{\mathop{:}\!\!=}}  
\newcommand{\asdef}{\mathrel{=\!\mathop{:}}}    

\newcommand{\with}{\;\middle\vert\;}  
\newcommand{\degr}{\operatorname{deg}}

\newcommand{\tou}{\uparrow}
\newcommand{\tod}{\downarrow}

\newcommand{\wto}{\xrightharpoonup{}}          

\newcommand{\ft}{\mathcal{F}}   

\newcommand{\hm}{\mathcal{H}}
\newcommand{\co}{\mathcal{O}}
\newcommand{\so}{o}
\newcommand{\dds}{\tfrac{d}{ds}}

\def\Xint#1{\mathchoice
{\XXint\displaystyle\textstyle{#1}}%
{\XXint\textstyle\scriptstyle{#1}}%
{\XXint\scriptstyle\scriptscriptstyle{#1}}%
{\XXint\scriptscriptstyle\scriptscriptstyle{#1}}%
\!\int}
\def\XXint#1#2#3{{\setbox0=\hbox{$#1{#2#3}{\int}$}
\vcenter{\hbox{$#2#3$}}\kern-.5\wd0}}

\def\dashint{\Xint-}

\DeclareMathOperator{\sgn}{sgn}

\makeatletter
\newcommand{\interitemtext}[1]{%
\begin{list}{}
{\itemindent=0mm\labelsep=0mm
\labelwidth=0mm\leftmargin=0mm
\addtolength{\leftmargin}{-\@totalleftmargin}}
\item #1
\end{list}}
\makeatother

\theoremstyle{plain}

\theoremstyle{definition}

\theoremstyle{remark}

\theoremstyle{plain}
\newtheorem{thm}{Theorem}
\newtheorem*{thm*}{Theorem}
\newtheorem{prop}{Proposition}
\newtheorem*{prop*}{Proposition}
\newtheorem{lem}{Lemma}
\newtheorem{cor}{Corollary}
\newtheorem{open}{Open problem}
\newtheorem{rem}{Remark}

\theoremstyle{definition}

\theoremstyle{remark}

\title{A reduced model for domain walls in soft ferromagnetic films at the cross-over from symmetric to asymmetric wall types.}
\author{Lukas D\"oring\thanks{Max Planck Institute for Mathematics in the Sciences, Inselstra{\ss}e 22, 04103 Leipzig, Germany (email: Lukas.Doering@mis.mpg.de, Felix.Otto@mis.mpg.de)} \quad Radu Ignat\thanks{Institut de Math\'ematiques de Toulouse, Universit\'e Paul Sabatier, 118 Route de Narbonne, 31062 Toulouse, France (email: Radu.Ignat@math.univ-toulouse.fr)} \quad Felix Otto\footnotemark[1]}

\begin{document}
\maketitle

\input{introduction}
\input{physics}
\input{compactness}

\input{proof}
\appendix
\input{construction}

\section*{Acknowledgements}
R.I. gratefully acknowledges the hospitality of Max Planck Institute (Leipzig) where part of this work was carried out; he also acknowledges partial support by the ANR projects ANR-08-BLAN-0199-01 and ANR-10-JCJC 0106. L.D. acknowledges support of the International Max Planck Research School. F.O. acknowledges the hospitality of Fondation Hadamard, and both L.D. and F.O. acknowledge the hospitality of the mathematics department of the University of Paris-Sud.

\bibliography{references}
\bibliographystyle{abbrv}
\end{document}

%% file: introduction.tex
\begin{abstract}
  We study the Landau-Lifshitz model for the energy of multi-scale transition layers -- called ``domain walls'' -- in soft ferromagnetic films. Domain walls separate domains of constant magnetization vectors $m^\pm_\alpha\in\mathbb{S}^2$ that differ by an angle $2\alpha$. Assuming translation invariance tangential to the wall, our main result is the rigorous derivation of a reduced model for the energy of the optimal transition layer, which in a certain parameter regime confirms the experimental, numerical and physical predictions: The minimal energy splits into a contribution from an asymmetric, divergence-free core which performs a partial rotation in $\mathbb{S}^2$ by an angle $2\theta$, and a contribution from two symmetric, logarithmically decaying tails, each of which completes the rotation from angle $\theta$ to $\alpha$ in $\mathbb{S}^1$. The angle $\theta$ is chosen such that the total energy is minimal. The contribution from the symmetric tails is known explicitly, while the contribution from the asymmetric core is analyzed
in \cite{dioasymwalls12}.

  Our reduced model is the starting point for the analysis of a bifurcation phenomenon from symmetric to asymmetric domain walls. Moreover, it allows for capturing asymmetric domain walls including their extended tails (which were previously inaccessible to brute-force numerical simulation).
\end{abstract}

\vfill

{\small\textbf{Keywords:} $\Gamma$-convergence, concentration-compactness, transition layer, bifurcation, micromagnetics.\\
\textbf{MSC:} 49S05, 49J45, 78A30, 35B32, 35B36\\
\textbf{Submitted to:} Journal of the European Mathematical Society}

\section{Introduction}
\subsection{Model}
We consider the following model: The magnetization is described by a unit-length vector field
$$
  m=(m_1, m_2, m_3) \colon \Omega \to \mathbb{S}^2,$$
 where  the two-dimensional domain $$\Omega = \R \times (-1,1)
$$
corresponds to a cross-section of the sample that is parallel to the $x_1x_3$-plane. The following ``boundary conditions at $x_1=\pm\infty$'' are imposed so that a transition from the angle $-\alpha$ to $\alpha \in (0,\tfrac{\pi}{2}]$ is generated and a
domain wall forms parallel to the $x_2x_3$-plane (see Figure~\ref{fig:samplegeometry}):
\begin{align}\label{eq:bcinfty}
  m(\pm\infty,\cdot)=m^\pm_\alpha := (\cos \alpha, \pm\sin\alpha, 0),
\end{align}
with the convention:
\beq
\label{convent}
  f(\pm\infty,\cdot) = a_\pm \iff \int_{\Omega_+} \lvert f - a_+ \rvert^2 \, dx + \int_{\Omega_-} \lvert f - a_- \rvert^2 \, dx < \infty,
\eeq
where $\Omega_+=\Omega\cap\{x_1\geq0\}$ and $\Omega_-=\Omega\cap\{x_1\leq0\}$. Throughout the paper, we use the variables $x=(x_1,x_3)\in \Omega$ together with the differential operator $\nabla=(\partial_{x_1}, \partial_{x_3})$, and we denote by $m'=(m_1, m_3)$ the projection of $m$ on the $x_1x_3$-plane.

\begin{figure}[hb]
\centering
\begin{pspicture}(0,0)(8,3)
\psframe[fillstyle=solid,fillcolor=gray](1.5,0.5)(6.5,1.5)
\psline(7.3,1.25)(6.3,0.9)
\rput(7.5,1.25){\psscalebox{0.7}{$\Omega$}}
\psframe[linestyle=dashed](0,0)(5,1)
\psline[linestyle=dashed](0,1)(3,2)
\psline[linestyle=dashed](5,1)(8,2)
\psline[linestyle=dashed](5,0)(8,1)
\psline[linestyle=dashed](3,2)(8,2)
\psline[linestyle=dashed](8,1)(8,2)
\psline{->}(7.5,0)(8,0)
\rput[l](8.1,-0.1){\psscalebox{0.7}{$x_1$}}
\psline{->}(7.5,0)(7.8,0.1)
\rput[lb](7.9,0.1){\psscalebox{0.7}{$x_2$}}
\psline{->}(7.5,0)(7.5,0.5)
\rput[b](7.6,0.55){\psscalebox{0.7}{$x_3$}}
\psline[linewidth=1pt]{<-}(2.4,1.66)(2.8,1.8)
\psline[linewidth=1pt]{<-}(3.7,1.62)(3.5,1.82)
\psline[linewidth=1pt]{->}(4.3,1.73)(4.8,1.73)
\psline[linewidth=1pt]{->}(5.3,1.7)(5.8,1.78)
\psline[linewidth=1pt]{->}(6.4,1.66)(6.8,1.8)
\psline[linewidth=1pt]{<-}(1,1.2)(1.4,1.34)
\psline[linewidth=1pt]{<-}(2.3,1.16)(2.1,1.36)
\psline[linewidth=1pt]{->}(2.9,1.27)(3.4,1.27)
\psline[linewidth=1pt]{->}(3.9,1.24)(4.4,1.32)
\psline[linewidth=1pt]{->}(5,1.2)(5.4,1.34)
\end{pspicture}
\caption{The cross-section $\Omega$ in a ferromagnetic sample on a mesoscopic level.}
\label{fig:samplegeometry}
\end{figure}
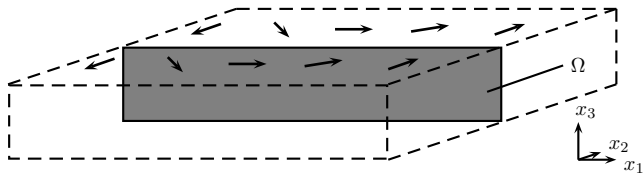

We focus on the following micromagnetic energy functional depending on a small parameter~$\eta$:\footnote{We refer to Section~\ref{sec:physics} for more information on $E_\eta$ and the parameters $\eta$ and $\lambda$.}
\begin{align}\label{eq:mm2deta}
  E_\eta(m) = \int_\Omega \lvert \nabla m \rvert^2 dx + \lambda \ln\tfrac{1}{\eta} \int_{\R[2]} \lvert h(m) \rvert^2 dx + \eta \int_\Omega ( m_1 - \cos \alpha)^2 + m_3^2 \, dx,\quad \eta\in(0,1),
\end{align}
subject to the boundary conditions \eqref{eq:bcinfty}, where $\lambda>0$ is a fixed constant and $h=h(m)\colon\R[2]\to \R[2]$ stands for the unique $L^2$ stray-field restricted to the $x_1x_3$-plane that is generated by the static Maxwell equations:\footnote{Existence and uniqueness of the stray field are a direct consequence of the Riesz representation theorem in the Hilbert space $V=\left\{v\in L_\text{loc}^2(\R[2]) \with \nabla v\in L^2, \, \dashint_{B(0,1)} v\, dx=0\right\}$ endowed with the norm $\|\nabla v\|_{L^2}$: Indeed, by \eqref{eq:bcinfty}, the functional $v\mapsto \int_\Omega \big(m'-(\cos \alpha, 0)\big)\cdot \nabla v\, dx$ is linear continuous on $V$ so that there exists a unique solution $h=-\nabla u$ with $u\in V$ of \eqref{eq:maxwell_2D} written in the weak form $\int_{\R[2]} \nabla u\cdot \nabla v\, dx=\int_\Omega m'\cdot \nabla v\, dx$ for every $v\in C^\infty_c(\R[2])$.}
\begin{align}\label{eq:maxwell_2D}
  \left\{\begin{aligned}\nabla\cdot(h+m'\mathbf{1}_\Omega) &= 0 \quad\text{in }\mathcal{D}'(\R[2]),\\
  \nabla \times h &= 0 \quad\text{in }\mathcal{D}'(\R[2]). \end{aligned}\right.
\end{align}
The first term of \eqref{eq:mm2deta} 
is called the ``exchange energy'', favoring a constant magnetization. The second term (called ``stray-field energy'') 
can be written as the $\dot{H}^{-1}(\R[2])$-norm of the $2D$ divergence of $m'$ (where $m$ is always extended by $0$ outside of $\Omega$):
$$\int_{\R[2]} \lvert h(m) \rvert^2 dx=\|\nabla \cdot (m' \mathbf{1}_\Omega)\|_{\dot{H}^{-1}(\R[2])}^2:=\sup \left\{\int_\Omega m'\cdot \nabla v\, dx \with v\in C^\infty_c(\R[2]), \,  \|\nabla v\|_{L^2(\R[2])}\leq 1\right\}.$$ The last term in \eqref{eq:mm2deta} (a combination of material anisotropy and external magnetic field) forces the magnetization to favor the ``easy axis'' 
$m^\pm_\alpha$ and serves as confining mechanism for the tails of the transition layer. We refer to Section~\ref{sec:physics} for more physical details about this model.

We are interested in the asymptotic behavior of minimizers $m_\eta$ of $E_\eta$ with the boundary condition \eqref{eq:bcinfty} as $\eta\tod 0$.
The main feature of this variational principle is
the non-convex constraint on the magnetization ($|m_\eta|=1$) and the non-local structure of the energy (due to the stray field $h(m_\eta)$). The competition between the three terms of the energy together with the boundary constraint \eqref{eq:bcinfty} induces an optimal transition layer that exhibits two length scales (cf. Figure~\ref{asym}):

\begin{itemize}
  \item an asymmetric core of size $\big(|x_1|\lesssim 1\big)$ (up to a logarithmic scale in $\eta$) where the magnetization $m_\eta$ is asymptotically divergence-free (so, generating 
no stray field) and hence the leading order term in $E_\eta$ is
given by the exchange energy; in this region, $m_\eta$ describes a transition path on $\mathbb{S}^2$ between the two directions $m^\pm_\theta$ determined by some angle $\theta$. 
  \item two symmetric tails of size $\big(1\lesssim |x_1|\lesssim \frac 1 \eta \big)$ (up to a logarithmic scale in $\eta$) where $m_\eta$ asymptotically behaves as a symmetric N\'eel wall:
    a one-dimensional (i.e., $m_\eta=m_\eta(x_1)$) rotation on $\mathbb{S}^1:=\mathbb{S}^1\times \{0\}\subset \mathbb{S}^2$ between the angles $\theta$ and $\alpha$ (on the left and right sides of the core). Here, the 
formation of the wall profile is driven by the stray-field energy that induces a logarithmic decay of $m_{1, \eta}$ on these two tails.
\end{itemize}

The constant $\lambda>0$ and the wall angle $\alpha$ play a crucial role in the behavior of a minimizer $m_\eta$. In fact, for either $\alpha \ll 1$, or $\alpha\in(0,\tfrac{\pi}{2}]$ arbitrary but $\lambda$ small, a minimizer is expected to be asymptotically symmetric (i.e., $m_\eta=m_\eta(x_1)$) as $\eta \tod 0$. However, for sufficiently large $\lambda$, 
there exists a critical wall angle $\alpha^*$ where a bifurcation occurs:
It becomes favorable to nucleate an asymmetric domain wall in the core of the transition layer.

In \cite[Section 3.6.4 (E)]{hubertschaefer98}, Hubert and Sch\"afer state:
\begin{quote}
  ``The magnetization of an asymmetric N\'eel wall points in the same direction at both surfaces, which is [\ldots] favourable for an applied field along this direction. This property is also the reason why the wall can gain some energy by splitting off an extended tail, reducing the core energy in the field generated by the tail. [\ldots] The tail part of the wall profile increases in relative importance with an applied field, so that less of the vortex structure becomes visible with decreasing wall angle. At a critical value of the applied field the asymmetric disappears in favour of a symmetric N\'eel wall structure.''
\end{quote}

To justify this physical prediction, we will establish the asymptotic behavior of $\{E_\eta\}_{\eta \downarrow 0}$ through the method of $\Gamma$-convergence. The limiting reduced model does then show that the minimal energy splits into the separate contributions from the symmetric and asymmetric regions of the transition layer. This makes it possible to infer information on the size of the regions and the conjectured bifurcation from symmetric to asymmetric walls. For details, we refer to Section~\ref{outlook}.

\bigskip

\subsection{Results}\label{sec:results}
Let $\alpha\in (0, \frac \pi 2]$ and $\eta\in(0,1)$. Observe that for $m \colon \Omega \to \mathbb{S}^2$, finite energy $E_\eta(m)<\infty$ is equivalent to $m\in\dot{H}^1(\Omega,\mathbb{S}^2)$ and
$m'(\pm\infty,\cdot)\stackrel{\eqref{convent}}{=}(\cos \alpha, 0)$ (which in particular implies $\lvert m_2 \rvert(\pm\infty,\cdot)=\sin \alpha$, see Lemma \ref{lem:modml2}). In the following we focus on the set of magnetizations of wall angle $\alpha\in(0,\frac{\pi}{2}]$ with
a transition imposed by \eqref{eq:bcinfty}:
\begin{align}\label{eq:defxalpha}
  X^\alpha &:= \left\{ m \in \dot{H}^1(\Omega,\mathbb{S}^2) \with m(\pm\infty,\cdot)=m^\pm_\alpha \right\}.
\end{align}
Our main result consists in proving $\Gamma$-convergence of $\{ E_\eta\}_{\eta\downarrow 0}$, defined on $X^\alpha\subset \dot{H}^1(\Omega,\mathbb{S}^2)$, in the weak $\dot{H}^1$-topology 
to the $\Gamma$-limit functional
\begin{align}\label{eq:defe0}
  E_0(m) =
    \int_\Omega \lvert \nabla m \rvert^2 dx + 2\pi\,\lambda\,\bigl( \cos\theta_m - \cos\alpha \bigr)^2,
\end{align}
which is defined on a space $X_0\subset\dot{H}^1(\Omega,\mathbb{S}^2)$:

In order to give the definitions of $X_0$ (see \eqref{eq:defx0}) and the angle $\theta_m$ associated to $m\in X_0$ (see \eqref{eq:defthetam}), we need some preliminary remarks. First, due to the logarithmic penalization of the stray field in \eqref{eq:mm2deta} as $\eta \tod 0$, limiting configurations of a family $\{m_\eta\}_{\eta\downarrow 0}$ of uniformly bounded energy $E_\eta(m_\eta)\leq C$ (e.g., minimizers of $E_\eta$) are stray-field free. Second, note that for any $m\in\dot{H}^1(\Omega,\mathbb{S}^2)$ with $\nabla \cdot (m' \mathbf{1}_\Omega) = 0$ in $\mathcal{D}'(\R[2])$ (i.e., $\nabla \cdot m' =0$ in $\Omega$ and $m_3=0$ on $\partial\Omega$) there is a unique constant angle $\theta_m\in[0,\pi]$ such that
\begin{align}\label{eq:defthetam}
  \bar{m}_1(x_1) := \dashint_{-1}^{1} \! m_1(x_1,x_3) \, dx_3 = \cos\theta_m \quad \text{for all }x_1\in\R.
\end{align}
Observe that such vector fields have the property $m'(\pm\infty,\cdot)=(\cos \theta_m, 0)$ in the sense of \eqref{convent} (see \eqref{poinc1} and \eqref{poinc3} below) and moreover, $\lvert m_2\rvert(\pm\infty,\cdot)=\sin {\theta_m}$ (see Lemma \ref{lem:modml2} if $\theta_m\in (0, \pi)$, and Remark \ref{rem_zero} below if $\theta_m\in \{0, \pi\}$). We define $X_0$ as the non-empty (see Appendix) set of such configurations $m$ that additionally
change sign as $\lvert x_1 \rvert \to \infty$, namely $m_2(\pm\infty,\cdot)=\pm \sin {\theta_m}$ in the sense of \eqref{convent}: 
\begin{align}\label{eq:defx0}
  X_0 &:= \left\{ m\in \dot{H}^1(\Omega, \mathbb{S}^2) \with \,\nabla\cdot m' =0 \text{ in }\Omega, \,m_3=0\text{ on }\partial\Omega,\,m(\pm\infty,\cdot)=m^\pm_{\theta_m} \right\}.
\end{align}
Note, however, that due to vanishing control of the anisotropy energy as $\eta\tod 0$, a limiting configuration $m$ in general satisfies \eqref{eq:bcinfty} for an angle $\theta_m$ that differs from $\alpha$. 

\bigskip

\begin{rem}
\label{rem_zero}
Observe that if $\theta_m\in \{0,\pi\}$ for $m\in \dot{H}^1(\Omega,\mathbb{S}^2)$ with $\nabla\cdot (m'\mathbf{1}_\Omega) =0$ in $\mathcal{D}'(\R[2])$ -- in particular if $m\in X_0$ --, we have $m\in \{\pm {\bf e}_1\}$: Indeed, since $\lvert\bar{m}_1\rvert\equiv 1$ in $\R$ and $\lvert m \rvert=1$ in $\Omega$, we deduce $\lvert m_1 \rvert \equiv 1$ and $m_2\equiv m_3\equiv 0$ in $\Omega$.
\end{rem}

\bigskip

We further remark that the first term in the $\Gamma$-limit energy \eqref{eq:defe0} accounts for the exchange energy of the asymmetric core of a transition layer $m_\eta$ as $\eta\tod 0$, while
the second term in $E_0$ accounts for the contribution coming from stray field/anisotropy energy through extended (symmetric) tails of the wall configurations at positive $\eta$.

Our $\Gamma$-convergence result is established in three steps. We start with compactness results. The main difficulty comes from the boundary conditions \eqref{eq:bcinfty}, which are 
in general not carried over by
weak limits of magnetization configurations with uniformly bounded exchange energy. However, since the energy $E_\eta$ is invariant under translations in $x_1$-direction, and due to the constraint \eqref{eq:bcinfty} in $X^\alpha$, a change of sign in $m_2$ can be preserved as $\eta\tod 0$ by a suitable translation in $x_1$.

\medskip
\begin{prop}[Compactness]\label{prop:compactness1}
  Let $\alpha\in (0, \frac \pi 2]$. The following convergence results hold \textbf{up to a subsequence and translations in the $x_1$-variable}:
\begin{enumerate}
  \item Let $\{m_\eta\}_{\eta\downarrow 0}\subset X^\alpha$ with uniformly bounded energy, i.e., 
  $\sup_{\eta \downarrow 0} E_\eta(m_\eta)<\infty$. Then $m_\eta \wto m$ weakly in $\dot{H}^1(\Omega)$ for some $m \in X_0$.
\item Let $\{m_k\}_{k\uparrow \infty}\subset X^\alpha$ with uniformly bounded energy $E_\eta$ for $\eta\in(0,1)$ fixed, i.e., $\sup_k E_\eta(m_k)<\infty$. Then $m_k\wto m$ weakly in $\dot{H}^1(\Omega)$ for some $m \in X^\alpha$. Moreover, the corresponding stray fields $\{h(m_k)\}_{k\uparrow\infty}$ converge weakly in $L^2(\R[2])$, i.e., $h(m_k)\wto h(m)$ in $L^2(\R[2])$.
\item Let $\{m_k\}_{k\uparrow\infty}\subset X_0$ with uniformly bounded exchange energy, i.e., $\sup_k \int_\Omega |\nabla m_k|^2\, dx<\infty$, such that the angles $\theta_k := \theta_{m_k}$ associated to $m_k$ in \eqref{eq:defthetam} satisfy $\theta_k\in [0, \pi]$. Then $\theta_k \to \theta$ for some angle $\theta \in [0,\pi]$ and $m_k \wto m$ weakly in $\dot{H}^1(\Omega)$ for some $m\in X_0$ with $\theta_m=\theta$ (i.e., $m\in X_0\cap X^\theta$).
\end{enumerate}
\end{prop}

The main ingredient in Proposition \ref{prop:compactness1} is the following concentration-compactness type lemma related to the change of sign at $\pm \infty$:

\medskip

\begin{lem}\label{lem:sgnchg}
Let $u_k \colon \R \to \R$, $k\in\N$, be continuous and satisfy the following conditions:
\begin{gather}
  \limsup_{k\tou\infty} \int_{\R} \lvert \dds u_k(s) \rvert^2 ds < \infty,\label{eq:nubdd}\\
  \limsup_{s\tod -\infty} u_k(s)<0 \text{ and } \liminf_{s\tou\infty} u_k(s) > 0 \, \, \textrm{for every } k\in \N, \label{eq:sgnuinfty}
\end{gather}
where we denote by $\dds u_k$ the distributional derivative of the function $u_k$.

Then for each $k\in \N$, there exists a zero $z_k$ of $u_k$ and a limit $u \in \dot{H}^1(\R)$ such that $u(0)=0$,
\begin{align*}
  u_k(\cdot + z_k) \to u \quad \text{locally uniformly in $\R$ and weakly in $\dot{H}^1(\R)$ for a subsequence}
\end{align*} 
and
\begin{align}\label{eq:signu}
  \limsup_{s\tod-\infty} u(s) \leq 0 \quad \text{as well as} \quad \liminf_{s\tou\infty} u(s) \geq 0.
\end{align}
\end{lem}

The second step consists in proving the following lower bound:

\medskip
\begin{thm}[Lower bound]\label{thm:gammalimlb}
  Let $\alpha\in(0,\frac{\pi}{2}]$. For $m \in X_0$ and any family $\{m_\eta\}_{\eta\tod 0} \subset X^\alpha$ with $m_\eta \wto m$ in $\dot{H}^1(\Omega)$ as $\eta\tod 0$, the following lower bound holds:
  \begin{align}\label{eq:lowerbd}
    \liminf_{\eta\tod 0} E_\eta(m_\eta) \geq E_0(m).
  \end{align}
\end{thm}

The last step consists in constructing recovery sequences for limiting configurations:

\medskip
\begin{thm}[Upper bound]\label{thm:gammlimub}
  For $\alpha\in(0,\frac{\pi}{2}]$ and every $m \in X_0$ there exists a family $\{m_\eta\}_{\eta\tod 0} \subset X^\alpha$ with $m_\eta \to m$ strongly in $\dot{H}^1(\Omega)$ and
  \begin{align}\label{eq:upperbd}
    \limsup_{\eta\tod 0} E_\eta(m_\eta) \leq E_0(m).
  \end{align}
\end{thm}

As a consequence, one deduces the asymptotic behavior of the minimal energy $E_\eta$ over the space $X^\alpha$ as $\eta\downarrow 0$.

\medskip
\begin{cor}  
\label{cor1}
For $\alpha\in(0,\frac{\pi}{2}]$ and $\theta\in[0,\pi]$ we define
$$\EA(\theta)=\min_{\substack{m \in X_0\\\theta_m=\theta}} \int_\Omega \lvert \nabla m \rvert^2 dx$$
and
$$\ES(\alpha-\theta)=2\pi\, \bigl( \cos\theta-\cos\alpha\bigr)^2.$$
Then it holds
\begin{align}
\label{eq:reducedmodel}
 \lim_{\eta\downarrow 0} \min_{m_\eta\in X^\alpha} E_\eta(m_\eta)=\min_{m\in X_0} E_0(m) = \min_{\theta \in [0, \pi]} \Bigl( \EA(\theta) + \lambda\,\ES(\alpha-\theta)\Bigr).
\end{align}
In fact, the optimal angle $\theta$ is attained in $[0,\tfrac{\pi}{2}]$.
Moreover, every minimizing sequence $\{m_\eta\}_{\eta\downarrow 0}\subset X^\alpha$  of $\{E_\eta\}_{\eta\downarrow 0}$ in the sense of $E_\eta(m_\eta)\to \min_{X_0} E_0$ is relatively compact in the strong $\dot{H}^1(\Omega)$-topology, up to translations in $x_1$, having as accumulation points in $X_0$ minimizers of $E_0$.
\end{cor}

One benefit of \eqref{eq:reducedmodel} is splitting the problem of determining the optimal transition layer into two more feasible ones: First, the energy of asymmetric walls (i.e. walls of small width) has to be determined (at the expense of an additional constraint on $\nabla\cdot m'$). Afterwards, a one-dimensional minimization procedure is sufficient to determine the structure of the wall profile. Direct numerical simulation of \eqref{eq:mm2deta} has been a difficult endeavor (see \cite{miltatlabrune94} and also \cite[Section 3.6.4 (E)]{hubertschaefer98}).

\subsection{Outlook}
\label{outlook}
In the following we briefly discuss an application of our reduced model to the cross-over from symmetric to asymmetric N\'eel wall and point out further interesting (topological) questions and open problems associated with the energy of asymmetric domain walls.

{\bf Bifurcation}. The previous result represents the starting point
in the analysis of the bifurcation phenomenon (from symmetric to asymmetric walls) in terms of the wall angle $\alpha$ (see also \cite{dioasymwalls12}). We will prove that there is a supercritical (pitchfork) bifurcation 
(cf. Figure \ref{bifurc}): This means that for small angles $\alpha\ll 1$, the optimal transition layer $m_\eta$ of $E_\eta$ is asymptotically symmetric (the symmetric N\'eel wall); beyond a critical angle $\alpha^*$, the symmetric wall is no longer stable, whereas the asymmetric wall is.
To 
understand the type of the bifurcation, by \eqref{eq:reducedmodel}, we need to compute the asymptotic expansion of the asymmetric energy up to order $\theta^4$ as $\theta\to 0$ (since the symmetric part of the energy is quartic for small angles $\theta, \alpha\ll 1$, i.e., 
$\ES(\alpha-\theta) \lesssim \alpha^4$).\footnote{Observe that for given $\alpha \in (0,\frac{\pi}{2}]$ the optimal wall angle $\theta_\alpha=\operatorname{argmin}\left(\EA(\theta)+\lambda\ES(\alpha-\theta)\right)\in[0,\frac{\pi}{2}]$ satisfies the estimate $\theta_\alpha \lesssim \alpha^2$. Indeed, by comparison with $\theta=0$ we have $\EA(\theta_\alpha) +2\pi \lambda (\cos\theta_\alpha - \cos\alpha)^2 \leq 2\pi \lambda (1-\cos\alpha)^2$. {Omitting $E_\text{asym}(\theta_\alpha)$ we first obtain $\theta_\alpha\to 0$ as $\alpha\tod0$}, so that by \eqref{expans} one deduces that $2\theta_\alpha^2\lessapprox  \lambda (1-\cos\alpha)^2$ for small $\alpha>0$. From here, the desired estimate follows.}
In fact, we show (see \cite{dioasymwalls12}):
\beq \label{expans}
  \EA(\theta) = 4\pi \theta^2 + \tfrac{304}{105} \pi \theta^4 + \so(\theta^4) \quad \textrm{as}\quad \theta\downarrow 0.
\eeq
This allows us to heuristically determine a critical angle $\alpha^*$ at which the symmetric N\'eel wall loses stability and an asymmetric core is generated. Moreover, a new path of stable critical points with increasing inner wall angle $\theta$ branches off of $\theta=0$ (see Figure \ref{bifurc}).
Indeed, for small $\alpha$, combining with \eqref{expans}, the RHS of \eqref{eq:reducedmodel} as function of $\theta \in [0, \alpha]$ has the unique critical point $\theta=0$ if $\alpha\leq \alpha^*$ where the bifurcation angle $\alpha^*$ is given by $$\alpha^*=\arccos\left(1-\tfrac 2 \lambda\right)+o(1),\quad \textrm{as}\quad \alpha \to 0.$$ (Observe that $\alpha^*\in[0,\tfrac{\pi}{2}]$ provided
$\lambda\geq 2$; therefore, the bifurcation appears only if $\lambda$ is large enough.)
For $\alpha>\alpha^*$, the symmetric wall becomes unstable under 
symmetry-breaking perturbations and the optimal splitting angle $\theta$ becomes positive; hence, the asymmetric wall becomes favored by the system. Moreover, the second variation of the RHS of 
\eqref{eq:reducedmodel} along the branch of positive splitting angles is positive so that the bifurcation from symmetric to asymmetric wall is supercritical. 

\begin{figure}
\centering
\begin{pspicture}(-1,-0.5)(7,3.5)
\psline{->}(-0.5,0)(6,0)
\psline{->}(0,-0.5)(0,3)
\psline[linewidth=2pt](0,0)(3,0)
\psline[linewidth=2pt,linestyle=dashed](3,0)(5.5,0)
\parametricplot[plotstyle=line,linewidth=2pt,plotpoints=51]{0.3}{0.375}{t 10 mul 2 3.1415 mul 44.78 mul 1 t 3.1415 2 mul div 360 mul cos neg add mul 4 3.1415 mul neg add 12 3.1415 mul 3.1415 44.78 mul 4 t 3.1415 2 mul div 360 mul cos neg add 3 div mul add div sqrt 15 mul}
\psline(1,1)(1.5,0.1)
\psline(1,1)(3,1.2)
\rput*(1,1){\psscalebox{0.8}{stable}}
\psline(5,1)(4.5,0.1)
\rput*(5,1){\psscalebox{0.8}{unstable}}
\rput[b](0,3.1){\psscalebox{0.8}{$\theta$}}
\rput[l](6.1,0){\psscalebox{0.8}{$\alpha$}}
\rput[t](3,-0.1){\psscalebox{0.8}{$\alpha^*$}}
\rput[tr](-0.1,-0.1){\psscalebox{0.8}{$0$}}
\end{pspicture}
\caption{Bifurcation diagram for the angle $\theta$ of the asymmetric core, depending on the global wall angle $\alpha$.}
\label{bifurc}
\end{figure}
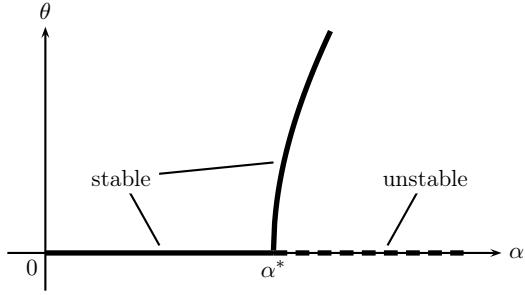

{\bf Topological degree and vortex singularity}. We now discuss topological properties of stray-field free magnetization configurations: In fact, if $m\in X_0$ satisfies \eqref{eq:bcinfty} for some angle $\theta \in (0, \frac \pi 2]$, denoting the ``extended'' boundary of $\Omega$ 
 \beq
\label{not_bdry}
\overline{Bdry}:=\partial\Omega\cup\bigg(\{\pm\infty\}\times[-1,1]\bigg) \cong \mathbb{S}^1,\eeq
then one can define the following winding number of $m$ on $\overline{Bdry}$: due to $m_3=0$ on $\partial\Omega$ as well as $m_3(\pm\infty,\cdot)=0$ 
(so, $(m_1, m_2)\colon\overline{Bdry}\to \mathbb{S}^1$), one obtains (by the homeomorphism \eqref{not_bdry}) a map $\tilde{m} \in H^\frac{1}{2}(\mathbb{S}^1, \mathbb{S}^1)$ to which a topological degree can be associated (see, e.g., \cite{brezisnirenberg95}). In particular, in the case of smooth $\tilde{m}\colon\mathbb{S}^1\to \mathbb{S}^1$, the topological degree (also called winding number) of $\tilde{m}$ is defined as follows: $$\degr(\tilde{m}):=\frac{1}{2\pi} \int_{\mathbb{S}^1} \det(\tilde{m}, \partial_\theta \tilde{m})\, d{\cal H}^1$$
where $\partial_\theta \tilde{m}$ is the angular derivative of $\tilde{m}$.

We will show the following relation between the winding number of $m\in X_0$ on $\overline{Bdry}$ and topological singularities of $(m_1, m_3)$ inside $\Omega$: the non-vanishing topological degree of $(m_1, m_2)\colon\overline{Bdry}\to \mathbb{S}^1$ generates vortex singularities of $(m_1, m_3)$ as illustrated in Figure \ref{asym}. By vortex singularity of $v:=(m_1, m_3)$, we understand a zero of $v$ carrying a non-zero topological degree. In general, this is implied by the existence of a smooth cycle (i.e., closed curve) $\gamma\subset \Omega$ such that
$|v|>0$ on $\gamma$ and $\degr(\frac{v}{|v|}, \gamma)\neq 0$; the vector field $v$ then vanishes in the domain bounded by $\gamma$.

\medskip 

\begin{lem}
\label{lem-topo}
  Let $m \in X_0$ (i.e. $m\in \dot{H}^1(\Omega,\mathbb{S}^2)$ with $\nabla \cdot (m' \mathbf{1}_\Omega) = 0$ in $\mathcal{D}'(\R[2])$) such that \eqref{eq:bcinfty} holds for some angle $\theta \in (0, \frac \pi 2]$. Suppose that $(m_1, m_2):\overline{Bdry}\to \mathbb{S}^1$ has a non-zero winding number on $\overline{Bdry}$. Then there exists a vortex singularity of $(m_1,m_3)$ in $\Omega$ carrying
  a non-zero topological degree.
\end{lem}

\bigskip

Motivated by Lemma \ref{lem-topo}, let us introduce the set 
$${\cal L}^\theta=\{m\in X_0\cap X^\theta\,:\, \deg m=1\}$$ for a fixed angle $\theta\in (0, \tfrac{\pi}{2}]$. First of all, 
we have that ${\cal L}^\theta\neq \emptyset$ (see Appendix).\footnote{Naturally, one can address a similar question by imposing an arbitrary winding number $n$. For the case $n=0$, we
analyze this problem in \cite{dioasymwalls12} which is typical for asymmetric N\'eel walls; in particular, for small angles $\theta$, we construct an element $m\in X_0\cap X^\theta$ with $\deg m=0$ and asymptotically minimal energy. {Moreover, given any $m\in X_0\cap X^\theta$ with finite energy, one can use a reflection and rescaling argument to define a finite-energy magnetization on $\Omega$ with degree $0$ (see Remark \ref{rem5} (iii) ).}}
 Since $X_0=\cup_{\theta \in [0, \pi]} \big(X_0\cap X^\theta\big)$, the relation ${\cal L}^\theta\neq \emptyset$ obviously implies that $X_0\cap X^\theta\neq \emptyset$ for every $\theta \in (0, \pi)$ which is essential in our reduced model given by the $\Gamma$-convergence program.
 A natural question concerns the closure (in the weak $\dot{H}^1(\Omega)$-topology) of the set ${\cal L}^\theta$. This is important in order to define the (limit) asymmetric Bloch wall by minimizing the exchange energy on ${\cal L}^\theta$.\footnote{This question is related to the theory of Ginzburg-Landau minimizers with prescribed degree (see e.g. Berlyand and Mironescu \cite{MirBer}).}

\medskip
\begin{open}
Is the following infimum $$\inf_{m\in {\cal L}^\theta} \int_\Omega |\nabla m|^2\, dx$$ attained for every angle $\theta\in (0, \frac \pi 2]$?
\end{open}

\subsection{Structure of the paper}
This paper is organized as follows: In Section~\ref{sec:physics}, we explain the relation of \eqref{eq:mm2deta} to the full Landau-Lifshitz energy, as well as the physical background of our analysis.

In Section~\ref{sec:comp}, we prove the compactness results in Lemma~\ref{lem:sgnchg} and Proposition~\ref{prop:compactness1}, which in particular yield existence of minimizers of $E_\eta$, $\EA(\theta)$ and $E_0$.

Section~\ref{sec:proof} contains the proofs of the lower and upper bound (Theorems \ref{thm:gammalimlb} and \ref{thm:gammlimub}) of our $\Gamma$-convergence result and also, the proof of Corollary \ref{cor1}.

In the Appendix, finally, we show that the set $X_0\cap X^\theta$ is non-empty for any given angle $\theta \in (0,\tfrac{\pi}{2}]$. To this end, we construct an admissible configuration in $\EA(\theta)$ with non-zero topological degree on the boundary of $\Omega$ (i.e., of asymmetric Bloch-wall type). Moreover, we prove Lemma~\ref{lem-topo}. 

%% file: physics.tex
\section{Physical background}\label{sec:physics}
In this section, we denote by $\nabla=(\partial_{x_1},\partial_{x_2},\partial_{x_3})$ the full gradient of functions depending on $x=(x_1,x_2,x_3)$. Recall that the prime $'$ denotes the projection on the $x_1x_3$-plane, i.e. $\nabla' = (\partial_{x_1},\partial_{x_3})$, $x'=(x_1,x_3)$.

{\bf Micromagnetics}. Let $\omega\subset \R[3]$ represent a ferromagnetic sample 
whose magnetization is described by the unit-length vector-field $m \colon \omega \to \mathbb{S}^2$.
Assume that the sample exhibits a uniaxial anisotropy with $\mathbf{e}_2=(0,1,0)$ as ``easy axis'', i.e. favored direction of $m$.
 The well-accepted micromagnetic model (see e.g. \cite{dkmo04,hubertschaefer98}) states that in its ground state the magnetization minimizes the Landau-Lifshitz energy:
\begin{align}\label{eq:mmfull}
  E^{3D}(m) = d^2 \int_\omega \lvert \nabla m \rvert^2 dx + \int_{\R[3]} \lvert h(m) \rvert^2 dx + Q\int_\omega \! m_1^2+m_3^2 \, dx - 2 \int_\omega \! h_\text{ext} \cdot m \, dx.
\end{align}
Here, the exchange length $d$ is a material parameter that determines the strength of the exchange interaction of quantum mechanical origin, relative to the strength of the stray field $h=h(m)$. The stray field is the gradient field $h=-\nabla u$ that is (uniquely) generated by the distributional divergence $\nabla \cdot (m\mathbf{1}_\omega)$ via Maxwell's equation
\begin{align} \label{eq:maxwell}
  \nabla\cdot(h+m\mathbf{1}_\omega) &= 0 \quad\text{in }\mathcal{D}'(\R[3]).
\end{align}
The non-dimensional quality factor $Q>0$ is a material constant that measures the relative strength of the energy 
contribution coming from misalignment of $m$ with $\mathbf{e}_2$. \footnote{A typical, experimentally accessible, soft ferromagnetic material is Permalloy, for which $d\approx 5\text{nm}$ and $Q=2.5\cdot 10^{-4}$.}
 The last term, called Zeeman energy, favors alignment of $m$ with an external magnetic field $h_\text{ext} \colon \omega \to \R[3]$.

\bigskip

{\bf Derivation of our model}.
We assume the magnetic sample to be a thin film, infinitely extended in the $(x_1x_2)$-plane, i.e. $\omega = \R[2]\times (-t,t)$, where two magnetic domains of almost constant magnetization $m\approx m^\pm_\alpha$ have formed for $\pm x_1\gg t$.
Physically, such a configuration is stabilized by the combination of uniaxial anisotropy 
and suitably chosen external field $h_\text{ext}=Q\cos\alpha\,\mathbf{e}_1$. Moreover, we assume that $m$ and hence, the stray field $h=(h_1, 0, h_3)$ are independent of the $x_2$-variable so that \eqref{eq:mmfull} formally reduces to integrating the energy density (per unit length in $x_2$-direction):
\begin{align*}
  E^{2D}(m) = d^2 \int_{\omega'} \lvert \nabla' m \rvert^2 dx' + \int_{\R[2]} \lvert h' \rvert^2 dx' + Q \int_{\omega'} \! (m_1-\cos \alpha)^2 + m_3^2 \, dx'
  \end{align*}
  where $\omega'=\R\times (-t, t)$ and $h'=h'(m)=-\nabla' u$ satisfies \eqref{eq:maxwell_2D} driven by the $2D$ divergence of $m' \mathbf{1}_{\omega'}$. Recall that the prime $'$ here denotes a projection onto the coordinate directions $(x_1,x_3)$ transversal to the wall plane. 
After non-dimensionalization of length with the film thickness $t$, i.e., setting $\tilde x'=\tfrac{x'}{t}$, $\tilde \omega'=\tfrac{\omega'}{t}$, $\tilde m(\tilde x')=m(x')$, $\tilde u(\tilde x')=\tfrac{u(x')}{t}$,
the above specific energy (per unit length in $x_2$) is given by
\begin{align}\label{energy2d}
\tilde E^{2D}( \tilde m) = d^2\int_{\tilde \omega'} | \tilde \nabla' \tilde m|^2 d\tilde x' +t^2 
\int_{\R[2]} |\tilde{\nabla}' \tilde u |^2 d\tilde x' + Q t^2  \int_{\tilde \omega'} \! (\tilde m_1 - \cos\alpha )^2 + \tilde m^2_3 \, d\tilde x',
\end{align}
where the differential operator $\tilde \nabla'$ refers to the variables $\tilde x'=(\tilde x_1, \tilde x_3)$ and $\tilde u\colon\R[2]\to \R$ is the $2D$ stray-field potential given by $$\tilde \Delta' \tilde u=\tilde \nabla'\cdot (\tilde m \mathbf{1}_{\tilde \omega'})\quad \textrm{in} \quad \mathcal{D}'(\R[2]).$$
Throughout the section, {\it we omit $\, \tilde{}\, $ and $\, '\, $}.

\medskip

\nd {\bf Symmetric walls}. In the regime of very thin films (i.e. for a sufficiently small ratio of film thickness $t$ to exchange length $d$, see below for the precise regime),
the symmetric N\'eel wall $m$ is the favorable transition layer: It is characterized by a reflection symmetry w.r.t. the midplane $x_3=0$, see 3) below. In fact, to leading order in $\frac{t}{d}$, it is independent of the thickness variable $x_3$, i.e. $m=m(x_1)$, and in-plane, i.e. $m_3=0$.
The symmetric N\'eel wall is a two length-scale object with a core of size $w_{core}=\co(\tfrac{d^2}{t})$ and two logarithmically decaying tails $w_{core}\lesssim |x_1|\lesssim w_{tail}=\co(\tfrac{t}{Q})$ (see e.g. Melcher \cite{melcher03,melcher04}). 
It is invariant w.r.t. all the symmetries of the variational problem (besides translation invariance):
\begin{enumerate}
\item[1)] $x_1\to -x_1$, $x_3\to -x_3$, $m_2\to -m_2$;
\item[2)] $x_1\to -x_1$, $m_3\to -m_3$, $m_2\to -m_2$;
\item[3)] $x_3\to -x_3$, $m_3\to -m_3$; 
\item[4)] $Id$.
\end{enumerate}
The specific energy of a N\'eel wall of angle $\alpha=\tfrac{\pi}{2}$ is given by
$$E^{2D}(\textrm{symmetric N\'eel wall})=\co({t^2} \frac{1}{\ln \frac{w_{tail}}{w_{core}}})=\co(t^2 \frac{1}{\ln \frac{t^2}{d^2 Q}})$$
(see e.g. \cite{ottocrossover02,dkmo04}). 
For a symmetric N\'eel wall of angle $\alpha<\tfrac{\pi}{2}$, 
the energy is asymptotically quartic in $\alpha$ as it is proportional to
$(1-\cos\alpha)^2$ (see e.g. \cite{ignat09}).

\medskip

\nd {\bf Asymmetric walls}. 
For thicker films, the optimal transition layer has an asymmetric core, where the symmetry 3) is broken (see e.g. \cite{hubert69,hubert70}). The main feature of this asymmetric core is that it is approximately stray-field free. Hence to leading order, the asymmetric core is given by a smooth transition layer $m$ that satisfies \eqref{eq:bcinfty} and
\beq
\label{mag_asym}
m \colon \omega \to \mathbb{S}^2, \, \nabla \cdot m'=0 \textrm{ in } \, \omega \quad \textrm{and} \quad m_3=0 \textrm{ on } \, \partial \omega.\eeq
Observe that $(m_1, m_2)\colon\partial \omega \to \mathbb{S}^1$ since $m_3$ vanishes on $\partial \omega$, so that one can define a topological degree of $(m_1, m_2)$ on $\partial \omega$ (where $\partial \omega$ is the closed ``infinite'' curve $\big(\R\times \{\pm 1\}\big)\cup \big(\{\pm \infty\}\times [-1,1]\big)$). 
The physical experiments, numerics and constructions predict two types of asymmetric walls, differing in
their symmetries and the degree of $(m_1, m_2)$ on
$\partial \omega$: 
\begin{enumerate}
\item For small wall angles $\alpha$,
the system prefers the so-called asymmetric N\'eel wall. Its main features are the conservation of symmetries 1) and 4) and a
vanishing degree of $(m_1, m_2)$ on $\partial \omega$ (see Figure~\ref{asym}). Due to symmetry 1), the $m_2$ component of an asymmetric N\'eel wall
vanishes on a curve that is symmetric with respect to the center of the wall (by $x\to -x$). Moreover, the phase of $(m_1,m_2)$ is not monotone at the surface
$|x_3|=1$.

\item For large wall angles $\alpha$,
  the system prefers the so-called asymmetric Bloch wall. These walls only have the trivial symmetry 4).
  Another difference is the non-vanishing topological degree on $\partial \omega$ (i.e., $\degr\big((m_1, m_2), \partial \omega\big)=\pm 1$). Therefore, a vortex is nucleated in the wall core, and the curve of zeros of $m_2$ is no longer symmetric with respect the center of the wall (see Figure~\ref{asym}). Moreover, the phase of $(m_1,m_2)$ is expected to be monotone at the surface
$|x_3|=1$.
\end{enumerate}
The asymmetric wall has a single length scale $w_{core}\sim t$ and the specific energy comes from the exchange energy (see e.g. \cite{ottocrossover02,dkmo04}). It is of the order
$$E^{2D}(\textrm{asymmetric wall})=\co(d^2).$$ 
For small wall angles, the energy of the optimal asymmetric wall is asymptotically quadratic in $\alpha$ (see \cite{dkmo04}).
\begin{figure}[ht] %
  \begin{minipage}{0.48\linewidth}
    \centering
    \includegraphics[height=4cm]{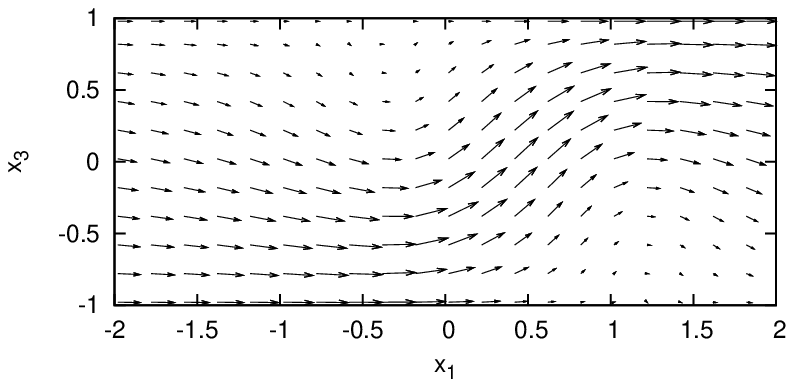}
    \rput(3.2,2.8){
      \pscircle*[linecolor=white](0,0){0.31cm}
      \pscircle(0,0){0.3cm}
      \rput{45}(0,0){\psline(-0.3,0)(0.3,0)}
      \rput{-45}(0,0){\psline(-0.3,0)(0.3,0)}
    }
    \rput(-2.0,2.8){
      \pscircle*[linecolor=white](0,0){0.31cm}
      \pscircle(0,0){0.3cm}
      \pscircle*(0,0){0.05cm}
    }
  \end{minipage}
  \;
  \begin{minipage}{0.48\linewidth}
    \centering
    \includegraphics[height=4cm]{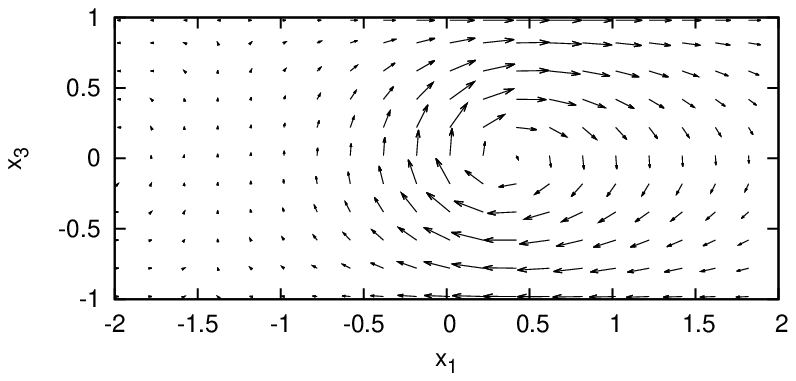} 
    \rput(3.2,2.85){
      \pscircle*[linecolor=white](0,0){0.31cm}
      \pscircle(0,0){0.3cm}
      \rput{45}(0,0){\psline(-0.3,0)(0.3,0)}
      \rput{-45}(0,0){\psline(-0.3,0)(0.3,0)}
    }
    \rput(1.35,2.85){\psscalebox{0.7}{
      \pscircle*[linecolor=white](0,0){0.31cm}
      \pscircle(0,0){0.3cm}
      \rput{45}(0,0){\psline(-0.3,0)(0.3,0)}
      \rput{-45}(0,0){\psline(-0.3,0)(0.3,0)}
    }}
    \rput(-2.0,2.85){
      \pscircle*[linecolor=white](0,0){0.31cm}
      \pscircle(0,0){0.3cm}
      \pscircle*(0,0){0.05cm}
    }
  \end{minipage}
  \caption{Asymmetric N\'eel wall (on the left) and asymmetric Bloch wall (on the right). Numerics.\protect\footnotemark}
  \label{asym}
  \end{figure}

  \footnotetext{The magnetization was obtained by numerically solving the Euler-Lagrange equation corresponding to $E_\text{asym}(\theta)$. To this end, a Newton method with suitable initial data was employed.}

\medskip

\nd {\bf Regime}. We focus on the challenging regime of soft materials of thickness $t$ close to the exchange length $d$ (up to a logarithm),
where we expect the cross-over in the energy scaling of symmetric walls and asymmetric walls (see \cite{ottocrossover02}):
$$Q\ll 1 \quad \textrm{and} \quad \ln \tfrac{1}{Q}\sim (\tfrac{t}{d})^2.$$
Rescaling the energy \eqref{energy2d} by $d^2$ and setting
$$\eta:=Q\tfrac{t^2}{d^2}\ll 1\quad \textrm{and} \quad \lambda:=\tfrac{t^2}{d^2 \ln \frac{1}{\eta}}>0,$$
then $\lambda=\co(1)$ is a tuning parameter in the system, and the rescaled energy, which is to be minimized, takes the form of energy $E_\eta$
given in \eqref{eq:mm2deta}
under the constraint
\begin{align*} &m\colon\Omega=\R\times (-1,1)\to \mathbb{S}^2, \, \, m(\pm \infty, \cdot)=m_\alpha^\pm,\\
& h=-\nabla u\colon\R[2]\to \R[2],  \, \,\nabla\cdot (h + m' \mathbf{1}_{\Omega})=0 \, \, \textrm{in} \, \, \mathcal{D}'(\R[2]).
\end{align*}

Observe that the parameter $\lambda$ measures the film thickness $t$ relative to the film thickness $d\ln^\frac{1}{2}\tfrac{1}{Q}$ characteristic to the cross-over. The limit $\eta \tod 0$ corresponds to a limit of vanishing strength of anisotropy, while at the same time the relative film thickness $\frac{t}{d}$ increases in order to remain in the critical regime of the cross-over.

\medskip

\nd {\bf Other microstructures in micromagnetics}. In other asymptotic regimes, different pattern formation is observed. Let us briefly mention three other microstructures that were recently studied: the concertina pattern, the cross-tie wall and a zigzag pattern.

{\it Concertina pattern}. In a series of papers (\cite{os10,sswmo12} among others) the formation and hysteresis of the concertina pattern in thin, sufficiently elongated ferromagnetic samples were studied. While in this case the transition layers between domains of constant magnetization are symmetric N\'eel walls, the program carried out for the concertina (a mixture of theoretical and numerical analysis, and comparison to experiments) serves as motivation for our work on the energy of domain walls in moderately thin films. Moreover, we hope that our analysis of the wall energy is helpful for studying a different route to the formation of the concertina pattern in not too elongated samples as proposed in \cite{vdbv82}, see also \cite{deflo12}.

{\it Cross-tie wall}. An interesting transition layer observed in physical experiments is the cross-tie wall (see \cite[Section 3.6.4]{hubertschaefer98}). It was rigorously studied in a reduced $2D$ model (by assuming vertical invariance of the magnetization) where a forcing term amounts to strong planar anisotropy that dominates the stray-field energy (see \cite{ARS02, RS01, RS03}). For small wall angles $\theta\in (0, \frac \pi 4]$, the optimal transition layer
is given by the symmetric N\'eel wall; for larger angles $\theta>\frac \pi 4$, the domain wall has a two-dimensional profile consisting in a mixture of vortices and N\'eel walls. The energetic cost of a transition in this $2D$ model is proportional to $\sin \theta-\theta \cos \theta$, so it is cubic in $\theta$ as $\theta\to 0$. This is due to the scaling of the stray-field energy (because of the thickness invariance assumption), which makes this reduced $2D$ model seem  
artificial. In the physics literature, it is known that for the full 3D model and large wall angles
the cross-tie wall may also be favored over the asymmetric Bloch wall. We hope that our more realistic wall-energy density confirms and helps to quantify this issue.

{\it A zigzag pattern}. In thick films, zigzag walls also occur. This pattern has been studied by Moser \cite{Moser09} in a $3D$ model with a
uniaxial anisotropy in an external magnetic field perpendicular to the ``easy axis'' (rather similar to our model). In fact, zigzag walls are to be expected there; however, this question is still open since 
the upper bound given for the limiting wall energy through
a zigzag construction does not match
the lower bound. Recently, in a 
reduced $2D$ model, Ignat and Moser \cite{IgMo} succeeded to rigorously prove the optimality of the zigzag pattern (for small wall angles). This was due to the improvement of the lower bound based on an entropy method (coming from scalar conservation laws). Remarkably, the function $\sin \theta - \theta \cos \theta$ plays an important role for the limiting energy density in that context as well as for the cross-tie wall.

%% file: compactness.tex
\section{Compactness and existence of minimizers}\label{sec:comp}
In this section we prove compactness results for sequences $\{m_k\}_{k\uparrow\infty}$ of magnetizations of bounded exchange energy. As an application we will derive existence of minimizers of $E_\eta$ (for some fixed $\eta\in(0,1)$) and $\EA(\theta)$ subject to a prescribed wall angle

$\theta\in (0, \pi)$,
and show that the optimal angle in $E_0$ is attained (cf. \eqref{eq:reducedmodel}).

All these statements are rather straightforward up to one point: The condition of sign-change
\begin{equation*}
\pm m_2(\pm \infty, \cdot) \geq 0
\end{equation*}
can in general not be recovered in the limit as shown in Figure \ref{fig:signprob}.

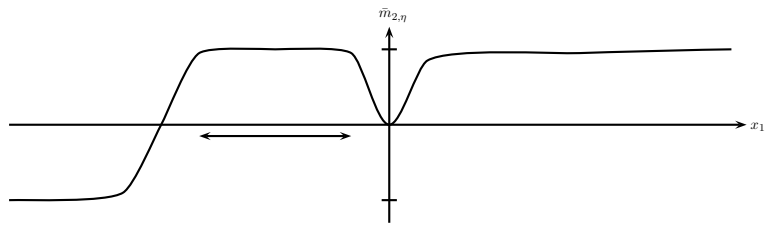
\begin{figure}[hbtp]
\centering
\begin{pspicture*}(0.45,0.5)(10.5,3.6)
\psline{->}(0.5,2)(10.2,2)
\psline{->}(5.5,0.7)(5.5,3.3)
\psline{-}(5.4,1)(5.6,1)
\psline{-}(5.4,3)(5.6,3)
\rput(5.55,3.45){\psscalebox{0.5}{$\bar{m}_{2,\eta}$}}
\rput(10.35,1.975){\psscalebox{0.5}{$x_1$}}
\psline{<->}(3,1.85)(5,1.85)
\psecurve[curvature=0.5 .1 0]{-}(0,1)(0.5,1)(2,1.1)(2.5,2)(3,2.95)(4,3)(5,2.95)(5.5,2)(6,2.85)(8,2.95)(10,3)(11.5,2.99)
\end{pspicture*}
\caption{The $x_3$-average $\bar{m}_{2,\eta}$ of the $m_2$-component. The arrow $\longleftrightarrow$ denotes that the length of the corresponding interval grows to $+\infty$ as $\eta \tod 0$. Then the limit $\bar{m}_{2}$ (as $\eta\tod 0$) has the same sign at $+\infty$ and $-\infty$.
}
\label{fig:signprob}
\end{figure}

However, we will show that one can always choose zeros $x_{1,\eta}$ of $\bar{m}_{2,\eta}$ in such a way that $m_\eta(\cdot+x_{1,\eta},\cdot)$ has the correct change of sign in the limit $\eta \tod 0$.

In the sequel we denote by $C>0$ a universal, generic constant, whose value may change from line to line, unless otherwise stated.

\subsection{Compactness}\label{sec::compactness}
We start by proving the $1D$ concentration-compactness result stated in Lemma~\ref{lem:sgnchg}.

\begin{proof}[Proof of Lemma \ref{lem:sgnchg}:] Due to \eqref{eq:sgnuinfty}, the set $Z_k := \{z \in \R \,\big|\, u_k(z)=0\}$ of zeros of $u_k$ is non-empty, and up to a translation in $x_1$-direction we may assume $u_k(0)=0$ for all $k\in\N$. 

  \textbf{Step 1:} \textit{For every sequence $\{z_k \in Z_k\}_{k\uparrow\infty}$ there exist a subsequence $\Lambda \subset \N$ and a limit $u \colon \R \to \R$ such that $u_k(\cdot+z_k)\to u$ locally uniformly for $k\tou\infty, k\in\Lambda$. Moreover, we have the bound 
  \begin{align*}
    \int_{\R} \lvert \dds u \rvert^2 ds \leq \liminf_{\substack{k\tou\infty\\k\in\Lambda}} \int_{\R} \lvert \dds u_k \rvert^2 ds < \infty.
  \end{align*}}

Indeed, by Cauchy-Schwarz's inequality, we have for $t\neq \tilde t$ that
  \begin{align*}
    \frac{\lvert u_k(t) - u_k(\tilde t) \rvert^2}{\lvert t- \tilde t \rvert}= \frac{\bigl(\int_{\tilde t}^t \dds u_k\, ds\bigr)^2}{\lvert t- \tilde t \rvert} \leq \int_{\R} \lvert \dds u_k \rvert^2 ds;
  \end{align*}
  thus, by \eqref{eq:nubdd}, we deduce that $\{u_k(\cdot+z_k)\}_{k\uparrow \infty}$ is uniformly H\"older continuous with exponent $\frac{1}{2}$. In particular, since $u_k(z_k)=0$, we also have that $\{u_k(\cdot+z_k)\}_{k\uparrow \infty}$ are locally uniformly bounded. Hence, the Arzel\`a-Ascoli compactness theorem yields uniform convergence on each compact interval $[-n, n]$, $n\in \N$, up to a subsequence. By a diagonal argument, one finds a subsequence $\Lambda\subset \N$ and a continuous limit $u\colon\R\to \R$ such that
 $$\textrm{$u_k(\cdot+z_k)\to u$ locally uniformly for $k\tou\infty$, $k\in\Lambda$.}$$ 
Moreover, the $L^2(\R)$-estimate on $\dds u$ follows from weak convergence in $L^2$ of $\dds u_k$ and weak lower-semicontinuity of the $L^2$ norm.
  
  \textbf{Step 2:} {\it Inductive construction of zeros.} \, 
Assume by contradiction that for every sequence $\{z_k \in Z_k\}_{k\uparrow\infty}$, no accumulation point $u$ (w.r.t. to locally uniform convergence) of the sequence $\{u_k(\cdot+z_k)\}_{k\uparrow\infty}$ satisfies \eqref{eq:signu}. We will show by an iterative construction that one can select a subsequence of $\{u_k\}_{k\uparrow\infty}$ such that each term $u_k$ has asymptotically infinitely many zeros (i.e., $\#Z_k \to \infty$ as $k\tou \infty$) with large distances in-between.

More precisely, we prove that for every $l\in\N$ there exist a limit $u^l\in\dot{H}^1(\R)$ and subsequences $\Lambda_l \subset \Lambda_{l-1}\subset \ldots \subset \Lambda_1 \subset \N$, such that for all $k\in\Lambda_l$ there exists an additional zero $z_k^l \in Z_k$ of $u_k$ with the properties:
\begin{gather*}
  \min_{1\leq i\neq j\leq l} \lvert z^i_k-z^j_k \rvert \to \infty \quad \text{and} \quad u_k(\cdot+z^l_k) \to u^l \text{ locally uniformly, as }k\tou \infty,\, k\in\Lambda_l.
\end{gather*}

In Step 3, we finally show that this construction implies that $u\equiv 0$ is one of the accumulation points of $\{u_k(\cdot+z_k)\}_{k\uparrow\infty}$ for $z_k\in Z_k$ a diagonal sequence of these $z_k^l$, i.e., \eqref{eq:signu} is satisfied, in contradiction to our assumption.

At level $l=1$, we choose the zero $z_k^1=0$ of $u_k$ for every $k\in \N$. Then by Step 1, there exists a subsequence $\Lambda_1 \subset \N$ and a limit $u^1\in \dot{H}^1(\R)$ such that 
$$\textrm{$u_k(\cdot+z_k^1)\to u^1$ locally uniformly for $k \tou \infty$, $k\in\Lambda_1$}.$$ By assumption, $u^1$ does not satisfy \eqref{eq:signu}. Hence, there exists $\varepsilon_1>0$ such that for every $s>0$ we can find $s_1>s$ such that
  \begin{align*}
    u^1(s_1)\leq -\varepsilon_1 <0 \quad \text{or}\quad u^1(-s_1) \geq \varepsilon_1 > 0. 
  \end{align*}
By uniform convergence, we also deduce that for every $s>0$ there exists an index $k_s\in \Lambda_1$ such that
  \begin{align*}
   \sup_{[-s_1,s_1]} \lvert u_{k}(\cdot+z_{k}^1) - u^1 \rvert \leq \tfrac{\varepsilon_1}{2} \text{ for } k \geq k_s, \, k\in \Lambda_1,
  \end{align*}
 which in particular implies that
  \beq
  \label{asum1}
    u_{k}(s_1+z_{k}^1)< 0 \quad \text{or} \quad u_{k}(-s_1+z_{k}^1)>0, \, \, \text{ for } k \geq k_s, \, k\in \Lambda_1.
  \eeq

  At level $l=2$, we proceed as follows: By the construction at level $l=1$, for every $s:=n\in \N$ we choose as above $s_1\geq n$ and $k:=k_n\in\Lambda_1$ (here, $\{k_n\}_{n\tou\infty}$
 is to be chosen increasing). We also know that $u_k$ satisfies \eqref{eq:signu} which implies by \eqref{asum1} that $u_k$ changes sign
 at the left of $-s_1+z_{k}^1$ or at the right of $s_1+z_{k}^1$. Choose $z_k^2\in Z_k$ as this new zero of $u_k$. Since $z^1_{k_n}=0$, we have
  \begin{align*}
    \lvert z_{k_n}^1 - z_{k_n}^2 \rvert \to \infty \, \, \textrm{as} \, \, n\tou \infty.
  \end{align*}
 Let $\tilde \Lambda_{2} = \left\{k_{n} \with n\in\N\right\} \subset \Lambda_1$ be the sequence of these indices.
 By Step 1, there exist a subsequence $\Lambda_2 \subset \tilde \Lambda_2$ and a limit $u^2\in \dot{H}^1(\R)$ such that 
 $$\textrm{$u_k(\cdot+z_k^2)\to u^2$ locally uniformly for $k \tou \infty$, $k\in\Lambda_2$}.$$
 
We now show the general construction, i.e. how one obtains the $(l+1)^\text{th}$ set of zeros from the construction after the $l^\text{th}$ step. Indeed, suppose the functions $u^1, \ldots, u^l$, the sequences $\Lambda_l\subset \ldots \subset \Lambda_1 \subset \N$ and the zeros $z_k^1,\ldots,z_k^l$ of $u_k$ for every $k\in \Lambda_l$ have already been constructed. We now construct $u^{l+1}$, $\Lambda_{l+1}$ and $z_k^{l+1}$ for $k \in \Lambda_{l+1}$:
By assumption, none of the limits $u^j$, $1\leq j \leq l$, satisfies \eqref{eq:signu}. Hence, there exists $\varepsilon_l>0$ such that for every $s >0$ we can find $s_1, \dots, s_l\geq s$ with the property:
  \begin{align*}
    \Bigl( u^j(s_j)\leq -\varepsilon_l <0 \quad \text{or}\quad u^j(-s_j) \geq \varepsilon_l > 0 \Bigr)\, \, \text{ for every } \, 1\leq j \leq l.  
  \end{align*}
By uniform convergence, we also deduce that for every $s>0$ there exists an index $k_s\in \Lambda_l$ such that for every
$1\leq j \leq l$ and every $k\geq k_s$ with $k\in \Lambda_l$:
  \begin{align*}
    \sup_{[-s_j,s_j]} \lvert u_{k}(\cdot+z_{k}^j) - u^j \rvert \leq \tfrac{\varepsilon_l}{2}
    \qquad \text{and} \qquad
    \min_{1\leq i\neq j \leq l} \lvert z_{k}^i - z_{k}^j \rvert \geq 4\max_{1\leq j\leq l} s_j.
  \end{align*}
  In particular, for every $s:=n\in \N$ we choose as above $s_1, \dots, s_l\geq n$ and $k:=k_n\in\Lambda_l$ (again, $\{k_n\}_{n\tou\infty}$
 is to be chosen increasing). Then we deduce that for all $1\leq j \leq l$ and $k\in\Lambda_l$:
  \begin{align*}
    u_{k}(s_j+z_{k}^j)< 0 \quad \text{or} \quad u_{k}(-s_j+z_{k}^j)>0,
  \end{align*}
  and the $l$ intervals $\{I_j:=[z_{k}^j-s_j, z_{k}^j+s_j]\}_{1\leq j\leq l}$ are disjoint.

Since $u_{k}$ satisfies \eqref{eq:signu}, 
there exists a new zero $z_k^{l+1} \in Z_{k}\setminus \bigcup_{j=1}^l I_j$ of $u_k$. Indeed, let us assume (after a rearrangement) that these intervals are ordered $I_1< I_2 < \dots < I_l$. If there is no zero to the left of $I_1$ 
(i.e., on $(-\infty, z_{k}^1-s_1]$) and in-between these $l$ intervals (i.e., on $\bigcup_{j=1}^{l-1} [z_{k}^j+s_j, z_{k}^{j+1}-s_{j+1}]$), then $u_{k}$ must have a negative sign at the right endpoint of each interval $I_j$ (i.e., $u_k(z_{k}^j+s_j)<0$) with ${1\leq j\leq l}$. In particular, there must be a zero of $u_k$ at the right of $I_l$, that we call $z_k^{l+1}$.

Set $\tilde \Lambda_{l+1} = \left\{k_{n} \with n\in\N \right\} \subset \Lambda_l$.
Then
  \begin{align*}
    \min_{1\leq j \leq l} \lvert z_{k}^j - z_k^{l+1} \rvert \to \infty \, \, \text{ as } k\tou \infty,\, k\in \tilde \Lambda_{l+1}.
  \end{align*}
  Finally, by Step 1, there exist $\Lambda_{l+1}\subset \tilde \Lambda_{l+1}$ and $u^{l+1}$ such that
  \begin{align*}
    u_k(\cdot+z_{k}^{l+1}) \to u^{l+1} \quad \text{locally uniformly in $\R$ as }k \tou \infty,\,k\in\Lambda_{l+1},
  \end{align*}
  which finishes the construction at the level $l+1$.

  \textbf{Step 3:} {\it Construction of vanishing diagonal sequence.}  We prove that the assumption in Step 2 (i.e. the assumption that no accumulation point of a sequence of translates of $\{u_k\}_{k\tou\infty}$ satisfies \eqref{eq:signu}) leads to a contradiction:
  
  Consider the construction done in Step 2. The sequence $\{u^l\}_{l\tou\infty}$ is uniformly bounded in $\dot{H}^1(\R)$. Hence, as in Step 1, there is a subsequence $\Lambda \subset \N$ and a function $u$ such that $u^l \to u$ locally uniformly for $l\tou\infty$, $l\in \Lambda$. In the following, we prove that $u\equiv 0$ on $\R$ (in particular \eqref{eq:signu} is satisfied). Indeed,
 we first observe that $0=u^l(0)\to u(0)$ as $l\tou\infty$, $l\in \Lambda$; thus, $u(0)=0$. 
 Let now $a>0$ and we want to prove that $u(a)=0$. For that, let $l\in \Lambda$ and $k\in \Lambda_l$. Then for $1\leq j \leq l$,
  \begin{align*}
    \frac{\lvert u_k(a+z_k^j) \rvert^2}{a} = \frac{\lvert u_k(a+z_k^j)-u_k(z_k^j)\rvert^2}{a} \leq \int_{z_k^j}^{a+z_k^j} \lvert \dds u_k \rvert^2 ds.
  \end{align*}

  For $k=k(a)\in\Lambda_l$ sufficiently large, the intervals $\{[z_k^j,a+z_k^j]\}_{1\leq j \leq l}$ are disjoint and we have
  \begin{align*}
     \sum_{1\leq j \leq l} \frac{\lvert u_k(a+z_k^j) \rvert^2}{a} \leq \sum_{1 \leq j\leq l} \int_{z_k^j}^{a+z_k^j} \lvert \dds u_k \rvert^2 ds \leq \int_{\R} \lvert \dds u_k \rvert^2 ds.
  \end{align*}

  Letting $k \tou \infty$, $k\in\Lambda_l$, it follows
  \begin{align*}
    \sum_{1\leq j \leq l} \frac{\lvert u^j(a) \rvert^2}{a} \leq \limsup_{k\tou \infty} \int_{\R} \lvert \dds u_k \rvert^2 ds < \infty.
  \end{align*}

  We may now let $l \tou \infty$, $l\in \Lambda$, and deduce that $u^l(a) \to 0$. In particular, this shows $u(a)=0$. The same argument adapts to the case $a<0$, so that one concludes $u\equiv 0$ in $\R$.

  Therefore, taking a diagonal sequence of the functions constructed in Step 2, one can then find a family $\{u_{k_l}(\cdot+z_{k_l}^l)\}_{l\in \Lambda}$ converging (locally uniformly) to the limit function $u\equiv 0$ that satisfies \eqref{eq:signu} in contradiction to our assumption.
\end{proof}

The following lemma reduces the problem of finding admissible limits (i.e., satisfying the limit condition \eqref{eq:bcinfty}) for a sequence of vector fields $\{m_k \colon \Omega \to \mathbb{S}^2\}_{k\tou\infty}$ to shifting the $x_3$-average $\bar{m}_{2,k}$ of the second component ${m}_{2,k}$:
\medskip
\begin{lem}\label{lem:modml2}
  Let $m\in \dot{H}^1(\Omega,\mathbb{S}^2)$ satisfy the limit condition $m'(\pm\infty,\cdot)=(\cos \theta, 0)$ in the $(m_1m_3)$-components in the sense of \eqref{convent} for some angle $\theta \in (0,\pi)$. Then
  \begin{align*}
    \int_\Omega \bigl\lvert \lvert m_2 \rvert - \sin\theta \bigr\rvert^2 dx < \infty.
  \end{align*}
  If additionally the $x_3$-average $\bar{m}_2$ of $m_2$ satisfies \eqref{eq:signu} (i.e. $\bar{m}_2$ changes sign), then we have 
  \begin{align*}
    m(\pm \infty,\cdot) = m^\pm_\theta.
  \end{align*}
\end{lem}
\medskip
\begin{rem}
\label{rem:excep}
\begin{enumerate}
  \item Note that the assumption $\theta \not\in\{0,\pi\}$ is crucial: If we consider $m\colon \Omega\to\mathbb{S}^2$ given by $m_3\equiv 0$,
  $$m_1(x)=\begin{cases} \cos (\frac \pi 3 x_1) &\textrm{ if } |x_1|\leq 1,\\ 
    1-\tfrac{\lvert x_1 \rvert}{1+x_1^2} &\textrm{ if } |x_1|> 1, \end{cases}\quad m_2(x) =\begin{cases} \sin (\frac \pi 3 x_1) &\textrm{ if } |x_1|\leq 1,\\ \sgn(x_1) \sqrt{1-m_1^2(x)} &\textrm{ if } |x_1|> 1, \end{cases}$$ then $m'(\pm\infty,\cdot)=(1, 0)$ (in the sense of \eqref{convent}), and $m\in\dot{H}^1(\Omega,\mathbb{S}^2)$ since
$$\int_\Omega \lvert \nabla m \rvert^2 dx=2\int_{\R} \frac{|\dds m_1|^2}{1-m_1^2}\, ds\leq C+4\int_{1}^\infty \! \frac{|\dds m_1|^2}{1-m_1}\, ds < \infty,$$ but $\int_\Omega m_2^2 \, dx = \infty$, so that \eqref{convent} fails for $m_2$. 

  \item Under the hypothesis of Lemma \ref{lem:modml2}, in the case $\theta\in \{0, \pi\}$, by Remark~\ref{rem_zero} one may still conclude that $m\in \{\pm {\bf e}_1\}$ provided that $\nabla\cdot (m'\mathbf{1}_\Omega)=0$ in ${\cal D}'(\R[2])$ (i.e. $m\in X_0$).
\end{enumerate}
\end{rem}
\begin{proof}[Proof of Lemma~\ref{lem:modml2}]
  By $m_1^2+m_2^2+m_3^2 = 1= \sin^2\theta+\cos^2\theta$ and the triangle inequality we have:
  \begin{align*}
    \MoveEqLeft \int_\Omega \lvert m_2^2 - \sin^2\theta \rvert^2 dx \leq 2 \int_\Omega \! \lvert m^2_1 - \cos^2\theta \rvert^2 + m_3^4 \, dx\\
    &\leq 8 \int_\Omega \! \lvert m_1 - \cos\theta \rvert^2 + m_3^2 \, dx < \infty,
  \end{align*}
  where we used $$|m_1^2 - \cos^2\theta| = \big|m_1 + \cos\theta\big| \, \big|m_1-\cos\theta \bigr|\leq 2 \big\lvert m_1-\cos\theta \bigr\rvert,$$ $m'(\pm\infty,\cdot)=(\cos \theta, 0)$ and $\lvert m_3 \rvert\leq 1$. Since $\lvert m_2^2 - \sin^2\theta\rvert = \bigl\lvert \lvert m_2 \rvert - \sin\theta \bigr\rvert \, \bigl \lvert \lvert m_2 \rvert + \sin\theta \bigr\rvert \geq \sin\theta \, \bigl\lvert \lvert m_2 \rvert - \sin\theta \bigr\rvert$ and $\theta \in (0,\pi)$, it follows that
  \begin{align*}
    \int_\Omega \bigl\lvert \lvert m_2\rvert - \sin\theta \bigr\rvert^2 dx \leq \tfrac{1}{\sin^2\theta} \int_\Omega \lvert m_2^2 - \sin^2\theta \rvert^2 dx < \infty.
  \end{align*}
This proves the first part of the lemma. To establish the second part, we note that due to $\bigl\lvert \lvert m_2 \rvert - \lvert \bar{m}_2 \rvert \bigr\rvert \leq\lvert m_2 - \bar{m}_2 \rvert$ we have
  \begin{align}\nonumber
    \int_{\R} \bigl\lvert \lvert \bar{m}_2 \rvert - \sin\theta \bigr\rvert^2 dx_1&=\frac 1 2  \int_{\Omega} \bigl\lvert \lvert \bar{m}_2 \rvert - \sin\theta \bigr\rvert^2 dx\\
& \label{eq:absm2bl2}     \leq \int_\Omega \lvert m_2 - \bar{m}_2 \rvert^2 + \bigl\lvert \lvert m_2 \rvert - \sin\theta \bigr\rvert^2 dx < \infty,
  \end{align}
  where we used the Poincar\'e-Wirtinger inequality
  \begin{align}\label{eq:pwirt}
    \int_{\R}\int_{-1}^1 \lvert m_2 - \bar{m}_2 \rvert^2\, dx_3 \, dx_1 \leq C \int_{\Omega} |\partial_{x_3} m_2|^2\, dx.
  \end{align}
Since $\||\bar{m}_2 |\|_{\dot{H}^1(\R)}=\|\bar{m}_2 \|_{\dot{H}^1(\R)}\leq \frac 1{\sqrt{2}} \|{m}_2 \|_{\dot{H}^1(\Omega)}<\infty$, we deduce with help of \eqref{eq:absm2bl2} that
$|\bar{m}_2 |-\sin \theta\in {H}^1(\R)$; in particular,
$\lvert \bar{m}_2(s)\rvert \to \sin\theta>0$ as $\lvert s \rvert \to \infty$.

  Under the additional assumption $\liminf_{s\tou\infty} \bar{m}_2(s)\geq 0$ and $\limsup_{s\tod-\infty} \bar{m}_2(s)\leq 0$, we deduce from $\lvert \bar{m}_2(s)\rvert \to \sin\theta>0$ as $\lvert s \rvert \to \infty$ that $\lvert \bar{m}_2(s) \rvert = \bar{m}_2(s)$ and $\lvert \bar{m}_2(-s) \rvert = -\bar{m}_2(-s)$ if $s$ is sufficiently large, so that \eqref{eq:absm2bl2} translates into
  \begin{align*}
    \int_{\R_-} \lvert \bar{m}_2+\sin\theta \rvert^2 dx_1 + \int_{\R_+} \lvert \bar{m}_2-\sin\theta \rvert^2 dx_1 < \infty.
  \end{align*}
  Together with \eqref{eq:pwirt}, this finally yields
  \begin{align*}
    \MoveEqLeft\int_{\Omega_-} \lvert m_2 +\sin\theta \rvert^2 dx + \int_{\Omega_+} \lvert m_2 - \sin\theta \rvert^2 dx\\
    &\leq2 \int_{\Omega} \lvert m_2 - \bar{m}_2 \rvert^2 dx + 4\int_{\R_-} \lvert \bar{m}_2 + \sin\theta \rvert^2 dx_1 + 4\int_{\R_+} \lvert \bar{m}_2 - \sin\theta \rvert^2 dx_1 < \infty.
  \end{align*}
\end{proof}

We now prove Proposition~\ref{prop:compactness1}. In fact, we shall prove it in form of the following proposition that treats all the cases at once: (i) corresponds to $\eta_k \tod 0$, (ii) corresponds to $\eta_k \equiv \eta\in(0,1)$, and (iii) corresponds to $\eta_k \equiv 0$.

\medskip
\begin{prop}\label{prop:compactness2}
  Suppose that the sequences $\{\theta_k\}_{k\uparrow\infty} \subset (0,\pi)$, $\{\eta_k\}_{k\uparrow\infty}\subset[0,1)$ satisfy
$ \theta_k \to \theta$ and $\eta_k\to\eta$ as $k\to \infty$
  \begin{align}\label{eq:convtheta}
    \text{ with } \theta\in(0,\pi) \text{ whenever }\eta\in(0,1).
  \end{align}
  Suppose further that the sequence $\{m_k\}_{k\uparrow\infty} \subset \dot{H}^1(\Omega,\mathbb{S}^2)$ satisfies
  \begin{align}\label{eq:spaceformk}
    m_k \in \left\{
    \begin{aligned}
      &X^{\theta_k}, &\text{for }\eta_k\in(0,1),\\
      &X_0 \cap X^{\theta_k}, 
      &\text{for }\eta_k = 0,
    \end{aligned}
    \right.
  \end{align}
  and
  \begin{align}\label{eq:energybd}
    \left\{
    \begin{aligned}
      &E_{\eta_k}(m_k), &\text{for }\eta_k>0,\\
      &E_0(m_k), &\text{for }\eta_k = 0,
    \end{aligned}
    \right\}
    \text{ is bounded as } k\to \infty.
  \end{align}

  Then there exist zeros $x_{1,k}$ of $\bar{m}_{2,k}$ such that after passage to a subsequence, there exists $m\in\dot{H}^1(\Omega,\mathbb{S}^2)$ such that
  \begin{gather}
    m_k(\cdot+x_{1,k},\cdot) \wto m \quad \text{weakly in }\dot{H}^1(\Omega)\text{ and weak-$*$ in }L^\infty(\Omega),\label{eq:convmk}\\
    h(m_k) \left\{\begin{aligned} &\wto h(m), &\text{for }\eta\in(0,1),\\&\to 0, &\text{for }\eta=0,\end{aligned}\right\} \text{ in }L^2(\Omega),\notag
    \intertext{and}
    m\in \left\{\begin{aligned}
      &X^\theta, &\text{for }\eta\in(0,1),\\
      &X_0 \cap X^{\tilde{\theta}}, &\text{for }\eta = 0,
    \end{aligned}\right.\notag
  \end{gather}
  with $\tilde{\theta}\in [0, \pi]$.\footnote{One might have that $\tilde{\theta}\neq \theta$, see Remark \ref{rem_doi} (ii).}
\end{prop}
\begin{proof}[Proof of Proposition~\ref{prop:compactness2}] 
We divide the proof in several steps:

\textbf{Step 1:} {\it Compactness of translates of averages $\{\bar{m}_{2,k}\}$}. According to \eqref{eq:energybd}, we have that
  \begin{gather*}
    \int_{\R} \lvert \tfrac{d}{dx_1} \bar{m}_{2,k} \rvert^2 dx_1    \quad\text{is bounded for $k\uparrow\infty$,}
  \end{gather*}
  i.e. \eqref{eq:nubdd} for $\bar{m}_{2,k}$. From \eqref{eq:spaceformk} we obtain
  \begin{gather*}
    \int_{\R_-} \lvert \bar{m}_{2,k} + \sin\theta_k \rvert^2 dx_1 + \int_{\R_+} \lvert \bar{m}_{2,k} - \sin\theta_k \rvert^2 dx_1<\infty \quad\text{for each $k\in\N$}.
  \end{gather*}
  Since $\theta_k \in (0,\pi)$, this implies in particular \eqref{eq:sgnuinfty} for $\bar{m}_{2,k}$. Hence by Lemma~\ref{lem:sgnchg}, there exist zeros $x_{1,k}$ of $\bar{m}_{2,k}$ and $u\in\dot{H}^1(\R)\cap C(\R)$ s.t. for a subsequence
  \begin{align*}
    \bar{m}_{2,k}(\cdot+x_{1,k},\cdot) \wto u \quad \text{weakly in }\dot{H}^1(\R) \text{ and locally uniformly},
  \end{align*}
  with $u$ satisfying \eqref{eq:signu}.

\medskip

 \textbf{Step 2:} {\it Convergence of $\{m_k(\cdot+x_{1,k},\cdot)\}$}. Because of \eqref{eq:energybd}, by standard weak-compactness results, there exists $m\in\dot{H}^1(\Omega)\cap L^\infty(\Omega)$ s.t. for a subsequence
  \begin{align*}
    m_k(\cdot+x_{1,k},\cdot)\wto m \quad \text{weakly in $\dot{H}^1(\Omega)$ and weak-$*$ in $L^\infty(\Omega)$}.
  \end{align*}
By Rellich's compactness result,
  \begin{align*}
    m_k(\cdot+x_{1,k},\cdot) \to m \quad \text{in }L^2_\text{loc}(\Omega) \text{ and a.e.},
  \end{align*}
  so that in particular $m_k \in \dot{H}^1(\Omega,\mathbb{S}^2)$ yields $m\in \dot{H}^1(\Omega,\mathbb{S}^2)$. We thus may identify $u$ as $\bar{m}_2$, i.e.
  $$\bar{m}_2\equiv u \quad \textrm{in} \quad \R,$$
  so that $\bar{m}_2$ satisfies \eqref{eq:signu}.
  
  To simplify notation, we identify $m_k$ with its translate $m_k(\cdot+x_{1,k},\cdot)$ in the sequel of the proof.

 \medskip

 \textbf{Step 3:} {\it If $\eta\in(0,1)$, we show that $m\in X^\theta$ and compactness of $\{h(m_k)\}$.} Indeed, in this case, \eqref{eq:energybd} yields in particular
  \begin{align*}
   \bigg\{ \int_\Omega \lvert m_{1,k} - \cos\theta_k \rvert^2 + m_{3,k}^2 \, dx\bigg\}_k \quad\text{is bounded as $k\tou\infty$},
  \end{align*}
  so that Fatou's lemma and Step 2 lead to:
  \begin{align*}
    \int_\Omega \lvert m_1-\cos\theta \rvert^2 + m_3^2\, dx < \infty,
  \end{align*}
  that is, $m'(\pm\infty,\cdot) = (\cos\theta,0)$. Since by assumption \eqref{eq:convtheta}, $\theta\in(0,\pi)$, and since $\bar{m}_2$ satisfies \eqref{eq:signu}, Lemma~\ref{lem:modml2} yields $m(\pm\infty,\cdot)=(\cos\theta,\pm\sin\theta,0)$. Hence, we indeed have $m\in X^\theta$. For proving compactness of stray fields, we note that \eqref{eq:energybd} yields in particular
  \footnote{Note that $h(m(\cdot+z_{1},\cdot))\equiv h(m)(\cdot+z_{1},\cdot)$ by uniqueness of $L^2$ stray-fields in \eqref{eq:maxwell_2D} associated to configurations satisfying \eqref{eq:bcinfty}. }
  \begin{align*}
   \bigg\{ \int_{\R[2]} \lvert h(m_k) \rvert^2\bigg\}_k \quad\text{is bounded as $k\tou\infty$}.
  \end{align*}
  Hence there exists $h\in L^2(\Omega)$ s.t. for a subsequence
  \begin{align*}
    h(m_k) \wto h \quad\text{weakly in }L^2(\R[2]).
  \end{align*}
  Passing to the limit in the distributional formulation $\nabla \cdot \bigl( h(m_k) + m_k' \mathbf{1}_\Omega \bigr) = 0$, $\nabla\times h(m_k)=0$ to obtain $\nabla\cdot\bigl( h+m'\mathbf{1}_\Omega\bigr)=0$, $\nabla\times h =0$ in $\mathcal{D}'(\R[2])$, and using uniqueness of the stray-field of $m$ with \eqref{eq:bcinfty}, we learn that $h=h(m)$.

 \medskip

 \textbf{Step 4:} {\it If $\eta=0$, we show that $h(m_k) \to 0$ in $L^2(\R[2])$ and $m\in X_0\cap X^{\tilde{\theta}}$ for some $\tilde
 {\theta}\in [0, \pi]$.} Indeed, in this case, \eqref{eq:energybd} yields in particular (recall that $m_k\in X_0$ yields $h(m_k)\equiv 0$):
  \begin{align*}
    \int_{\R[2]} \lvert h(m_k) \rvert^2 dx \to 0,
  \end{align*}
  so that passing to the limit in the distributional formulation $\nabla \cdot \bigl( h(m_k) + m_k' \mathbf{1}_\Omega \bigr) = 0$ we learn that $\nabla \cdot (m'\mathbf{1}_\Omega) = 0$ in $\mathcal{D}'(\R[2])$. Since $m\in\dot{H}^1(\Omega)\cap L^\infty(\Omega)$, this yields
  \begin{align}\label{eq:mdivfree}
    \nabla \cdot m' = 0 \text{ in }\Omega, \, m_3=0 \text{ on }\partial\Omega.
  \end{align}
On the other hand, $\nabla\cdot (m'\mathbf{1}_\Omega)=0$ in $\mathcal{D}'(\R[2])$ implies $\frac{d}{dx_1} \bar{m}_1=0$ on $\R$, so that there exists $\tilde{\theta}\in[0,\pi]$ with
  \begin{align}\label{eq:valm1b}
    \bar{m}_1 = \cos\tilde{\theta} \quad\text{on }\R.
  \end{align}
  We note that in general, $\tilde{\theta} \neq \theta = \lim_{\theta\uparrow\infty} \theta_k$. By the Poincar\'e-Wirtinger inequality in $x_3$ we obtain from \eqref{eq:valm1b}:
  \begin{align}\label{poinc1}
    \int_\Omega \lvert m_1 - \cos\tilde{\theta} \rvert^2 dx \leq C \int_\Omega \lvert \partial_{x_3} m_1 \rvert^2 dx <\infty.
  \end{align}
  By the Poincar\'e inequality in $x_3$, we obtain from \eqref{eq:mdivfree}:
  \begin{align}\label{poinc3}
    \int_\Omega \! m_3^2 \, dx \leq C\int_\Omega \lvert \partial_{x_3} m_3 \rvert^2 dx < \infty.
  \end{align}
  Hence we have $m'(\pm\infty,\cdot)=(\cos\tilde{\theta},0)$. To conclude, we distinguish two cases:
  
  \medskip
  
\nd {\it Case 1:  $\tilde{\theta} \in (0,\pi)$}. In this case, we may conclude by Lemma~\ref{lem:modml2} that $m(\pm\infty,\cdot)=(\cos\tilde{\theta},\pm\sin\tilde{\theta},0)$ as in case of $\eta\in(0,1)$. We thus obtain $m\in X_0\cap X^{\tilde{\theta}}$.

\medskip

\nd {\it Case 2: $\tilde{\theta}\in\{0,\pi\}$}. In this case, we apply Remark~\ref{rem_zero} to conclude from \eqref{eq:valm1b} that $m$ is one of the constant functions $\pm \mathbf{e}_1$ and thus trivially lies in $X_0\cap X^{\tilde{\theta}}$.
\end{proof}

\bigskip

\begin{rem}
\label{rem_doi}
\begin{enumerate}

\item The assumption $\theta\in(0,\pi)$ whenever $\eta\in(0,1)$ in \eqref{eq:convtheta} is due to Remark~\ref{rem:excep}, since in general the condition $m_2(\pm\infty,\cdot)=\pm\sin\theta$ fails if $\theta\in\{0,\pi\}$. However, if $\theta\in\{0,\pi\}$, one gets a weaker statement concerning the behavior of $\bar{m}_2$ at $\pm\infty$:

  {\bf Claim}. {\it Suppose that the sequences $\{\theta_k\}_{k\uparrow\infty}\subset(0,\pi)$, $\{\eta_k\}_{k\uparrow\infty}\subset(0,1)$ satisfy 
  \begin{align*} 
  \theta_k \to \theta \text{ and }\eta_k\to\eta \text{ with }\theta\in\{0,\pi\}\text{ and }\eta\in(0,1).
  \end{align*}
  Consider a sequence $\{m_k\}_{k\uparrow\infty}\subset \dot{H}^1(\Omega,\mathbb{S}^2)$ for which $m_k\in X^{\theta_k}$ and $\{E_{\eta_k}(m_k)\}$ is bounded. Then there exists $m\in\dot{H}^1(\Omega,\mathbb{S}^2)$ such that after passage to a subsequence:\footnote{No translation in $x_1$-direction is required here.}
  \begin{align*}
    m_k\wto m &\quad \text{weakly in $\dot{H}^1(\Omega)$ and weak-$*$ in $L^\infty(\Omega)$},\\
    h(m_k) \wto h(m) &\quad \text{weakly in }L^2(\R[2]),\\
    E_\eta(m) < \infty,&\\
    \bar{m}_2(x_1) \to 0 &\quad \text{as } \lvert x_1 \rvert \to \infty.
  \end{align*}
  }

  Indeed, we can essentially proceed as in the proof of Steps 1,2 and 3 in Proposition~\ref{prop:compactness2}. However, note that there is no need to apply Lemma~\ref{lem:sgnchg}; moreover, the application of Lemma~\ref{lem:modml2} (at Step 3) is no longer possible. Instead, note that $\bar{m}_1-\cos\theta,\, \bar{m}_3\in H^1(\R)$ yields $\lim_{\lvert x_1 \rvert\uparrow\infty} \bar{m}_1(x_1)=\cos\theta\in\{\pm 1\}$ and $\lim_{\lvert x_1 \rvert \uparrow\infty} \bar{m}_3(x_1)=0$. Therefore,
  \begin{align*}
    1=\limsup_{\lvert x_1\rvert\uparrow \infty} \dashint_{-1}^1 \lvert m(x) \rvert \, dx_3 \geq \limsup_{\lvert x_1 \rvert \uparrow \infty} \Bigl\lvert \bigl(\bar{m}_1, \bar{m}_2, \bar{m}_3\bigr)(x_1)\Bigr\rvert=\Bigl\lvert\bigl(1, \limsup_{\lvert x_1 \rvert \uparrow \infty} \lvert \bar{m}_2(x_1)\rvert, 0\bigr)\Bigr\rvert,
  \end{align*}
  i.e., $\lim_{\lvert x_1 \rvert\uparrow \infty} \bar{m}_2(x_1)=0$. \qed

\item Note that in the case $\eta=0$, the angle $\tilde{\theta}=\theta_m$ associated to the limiting configuration $m$ via \eqref{eq:defthetam} in general does not coincide with the limit $\theta$ of the sequence $\theta_k$. In particular, in the situation of Proposition~\ref{prop:compactness1} (i) for $\theta_k\equiv \alpha = \theta$, the limit angle $\tilde{\theta}=\theta_m$ describes the amount of asymmetric rotation in the wall core. Hence, the possibility of having $\tilde{\theta}\neq\theta$ is directly related to observing a non-trivial behavior of the reduced model \eqref{eq:reducedmodel}.

  However, there are also cases in which $\theta=\lim_k \theta_k$ coincides with the limit angle $\tilde{\theta}$, as can be seen in the statement of Proposition \ref{prop:compactness1} (iii).

\end{enumerate}
\end{rem}

\bigskip

\begin{proof}[Proof of Proposition~\ref{prop:compactness1}:]
  Statements (i) and (ii) are an immediate consequence of Proposition~\ref{prop:compactness2} by letting $\theta_k\equiv\alpha$.

  Statement (iii) follows from Remark~\ref{rem_zero}, if there exists a constant subsequence $\theta_k\in\{0,\pi\}$. Otherwise, we find a convergent subsequence $\{\theta_k\}_{k\uparrow\infty}\subset(0,\pi)$ to which we apply Proposition~\ref{prop:compactness2} with $\eta_k\equiv0$. In this latter case, not relabeling the subsequence, it remains to prove that the limit $\theta \defas \lim_{k\uparrow \infty} \theta_k$ satisfies $\theta = \theta_m$, i.e., $m\in X^\theta$.  Indeed, exploiting \eqref{eq:defthetam} and \eqref{eq:convmk}, one obtains
  \begin{align*}
    \cos\theta \leftarrow \cos\theta_k \equiv \bar{m}_{1,k}(\cdot+x_{1,k}) \to \bar{m}_1 \equiv \cos\theta_m \quad \text{as }k\uparrow\infty.
  \end{align*}
  Since $\theta,\,\theta_m\in[0,\pi]$, this yields $\theta=\theta_m$.
\end{proof}

\subsection{Existence of minimizers}
\label{sec:min}
Due to the compactness statements in Proposition~\ref{prop:compactness1}, one obtains existence of minimizers for $E_\eta$, $\EA(\theta)$ and $E_0$.
\medskip
\begin{thm}\label{thm:existence}\quad
\begin{itemize}
\item For fixed parameters $\eta\in(0,1)$ and $\alpha \in (0,\tfrac{\pi}{2}]$, there exists a minimizer of $E_\eta$ over the set $X^\alpha$.
\item For $\theta \in [0, {\pi}]$ fixed, there exists a minimizer of $\EA(\theta)$ over the (non-empty, cf. Appendix) set $m\in X_0$ with $\theta_m=\theta$.
\item The $\Gamma$-limit energy $E_0$ admits a minimizer over $X_0$. The optimal angle $\theta$ in the minimization problem \eqref{eq:reducedmodel} is attained.
\end{itemize}
\end{thm}
\begin{proof}
  Observe that the functionals $E_\eta$ and $\{m\mapsto \int_\Omega \lvert\nabla m \rvert^2 dx\}$ are lower-semicontinuous with respect to the weak convergence obtained in Proposition~\ref{prop:compactness1}. Hence, the first two statements in Theorem~\ref{thm:existence} follow immediately by the direct method in the calculus of variations, i.e. by applying the compactness results in Proposition~\ref{prop:compactness1} to minimizing sequences.
  
  For the third statement, we need an auxiliary lemma that we prove using the existence of minimizers of $E_\text{asym}(\theta)$ we have just shown:
\medskip
\begin{lem}\label{lem:lsceasym}
The map $\theta\in [0, \pi] \mapsto \EA(\theta)\in \R_+$ is lower semicontinuous.
\end{lem}
\begin{proof}[Proof of Lemma~\ref{lem:lsceasym}]
  This immediately follows from Proposition~\ref{prop:compactness1} (iii) by considering for each sequence $\{\theta_k\in [0, \pi]\}_{k\uparrow\infty}$, a sequence $\{m_k\in X_0\}_{k\uparrow\infty}$ of minimizers of $\EA(\theta_k)$ for each $k$.
\end{proof}

  Now, the third statement in Theorem~\ref{thm:existence} again follows by the direct method in the calculus of variations, since $E_0$ is just a continuous perturbation of $\EA$.
\end{proof}

%% file: proof.tex
\section{Proof of $\Gamma$-convergence}\label{sec:proof}
\subsection{Lower bound. Proof of Theorem \ref{thm:gammalimlb}}
To establish the lower bound \eqref{eq:lowerbd}, one has to estimate the exchange term in $E_\eta(m_\eta)$ as well as stray-field and anisotropy energy from below as $\eta\downarrow 0$. If $m$ is the limit of $m_\eta$ (in the weak $\dot{H}^1$-topology), then the exchange term will be estimated as $\eta\downarrow 0$ by $\int_\Omega \lvert \nabla m \rvert^2 dx$, while the stray-field and anisotropy energy will be estimated by $\lambda \ES(\alpha-\theta_m)$,
where $\theta_m$ is associated to $m\in X_0$ via \eqref{eq:defthetam}.

Let $C>0$ always denote a universal, generic constant.

W.l.o.g. we may assume $E_\eta(m_\eta) \leq C_0 < \infty$ for some $C_0>0$ and $m_\eta\to m$ in $L^2_\text{loc}(\Omega)$ and a.e. in $\Omega$
as $\eta\downarrow 0$.

\textbf{Step 1:} \textit{Exchange energy.} We first address estimating the exchange energy from below. Since $m_\eta \wto m$ in $\dot{H}^1(\Omega)$ as $\eta\tod 0$, we obviously have
\begin{align}\label{eq:exchangeestimate}
\int_\Omega \lvert \nabla m \rvert^2 dx \leq \liminf_{\eta \tod 0} \int_\Omega \lvert \nabla m_\eta \rvert^2 dx,
\end{align}
by weak lower-semicontinuity of the $L^2$ norm of $\{\nabla m_\eta\}_{\eta\downarrow 0}$.

\textbf{Step 2:} \textit{Choice of test function.} Now it remains to estimate both stray-field and anisotropy energy in $E_\eta(m_\eta)$ from below by $2\pi \, \lambda \, \bigl( \cos\theta_m - \cos\alpha \bigr)^2$.

Here the idea is to approximate the limit
\begin{align*}
\cos\theta_m - \cos\alpha = \dashint_{-1}^1 \dashint_{-1}^1 \! \bigl( m_1- \cos\alpha \bigr) \, dx_3 \, dx_1
\end{align*}
by $\dashint_{-1}^1 \dashint_{-1}^1 \! \bigl( m_{1,\eta} - \cos\alpha \bigr) \,dx_3 \, dx_1$ and to define a suitable test function $\zeta \colon \R[2] \to \R$, that captures the profile of the tails of a N\'eel wall (i.e. when $|x_1|\geq 1$) and has the property that
\begin{align*}
\dashint_{-1}^1 \dashint_{-1}^1 \! \bigl( m_{1,\eta} -\cos\alpha \bigr) \, dx_3 \, dx_1 = \dashint_{-1}^1 \dashint_{-1}^1 \! \bigl( m'_\eta - (\cos\alpha,0) \bigr) \cdot \nabla \zeta \, dx_3 \, dx_1.
\end{align*}
In this way, the stray-field energy will control $2\pi\,\lambda \,\bigl(\cos\theta_m-\cos\alpha\bigr)^2$. Note that the argument here is similar to the one used in \cite{ottocrossover02}.

\medskip
\begin{lem}\label{lem:zeta}
Let $a >1$ and $\zeta_0 \colon \R \to \R$ be the odd piecewise affine function defined by
\begin{align}\label{eq:innerzeta}
\zeta_0(x_1) \defas
\begin{cases}
x_1,                     & 0 < x_1 < 1,\\
1,                       & 1 \leq x_1 < a,\\
-\frac{1}{a}\,(x_1-2a),  & a \leq x_1 < 2a,\\
0,                       & 2a \leq x_1
\end{cases}
\end{align}
(see Figure \ref{fig:zeta}).
Let $\zeta \colon \R[2] \to \R$ be given by $\zeta(x_1,x_3) = \zeta_0(x_1)$ on $\Omega$ and harmonically extended to $\R[2]$ away from $\Omega$, i.e. $\zeta$ satisfies
\begin{gather}\label{eq:harmextzeta}
\left\{
\begin{aligned}
\Delta \zeta &= 0          & &\text{on } \R[2]\setminus \bar \Omega,\\
\zeta(\cdot, \pm 1) &= \zeta_{0} & &\text{on } \R.
\end{aligned}
\right.
\end{gather}
Then we have
\begin{align}
\label{new4}
\int_{\R[2]} \lvert \nabla \zeta \rvert^2 dx \leq \frac{8}{\pi} \, \ln a + C
\end{align}
for some constant $C = \co(1)$ as $a\tou\infty$.
\end{lem}
\medskip
\begin{rem} Problem
\eqref{eq:harmextzeta} can be solved explicitly via Fourier transform in the $x_1$-variable: 
\begin{align}\label{eq:defft}
  \bigl(\ft_{x_1} f\bigr)(k_1)\defas \tfrac{1}{\sqrt{2\pi}} \int_{\R} f(x_1) e^{-ik_1x_1} \,dx_1, \quad k_1\in \R, 
\end{align}
where $f \colon \R \to \R$.
In fact, \eqref{eq:harmextzeta} becomes a second-order ODE for $\ft_{x_1}(\zeta)$ in the $x_3$-variable. Imposing that $\zeta\in \dot{H}^1(\R[2]\setminus \Omega)$, we deduce:
\begin{align}\label{eq:outerzeta}
\ft_{x_1}\!\left(\zeta\right)(k_1,x_3) = \ft_{x_1}\!\left(\zeta_0\right)(x_1) \, e^{- \lvert k_1 \rvert \,(\lvert x_3 \rvert - 1)}, \quad k_1\in \R, \, |x_3|>1.
\end{align}
\end{rem}

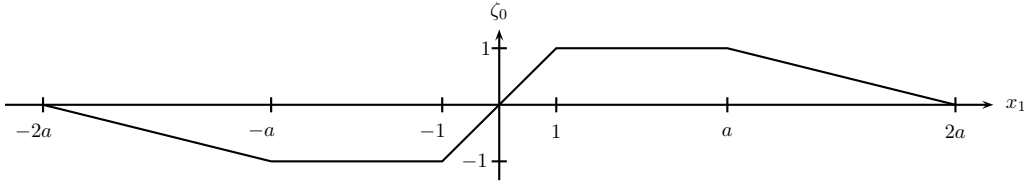
\begin{figure}[htbp]
\centering
\begin{pspicture*}(-0.5,0)(12.9,2.5)
\psline{->}(-0.5,1)(12.5,1)
\psline{->}(6,0)(6,2)
\psline{-}(0,0.9)(0,1.1)
\psline{-}(3,0.9)(3,1.1)
\psline{-}(5.25,0.9)(5.25,1.1)
\psline{-}(5.9,1.75)(6.1,1.75)
\psline{-}(5.9,0.25)(6.1,0.25)
\psline{-}(6.75,0.9)(6.75,1.1)
\psline{-}(9,0.9)(9,1.1)
\psline{-}(12,0.9)(12,1.1)
\rput(-0.125,0.65){\psscalebox{0.7}{$-2a$}}
\rput(2.875,0.65){\psscalebox{0.7}{$-a$}}
\rput(5.125,0.65){\psscalebox{0.7}{$-1$}}
\rput(5.675,0.25){\psscalebox{0.7}{$-1$}}
\rput(5.825,1.75){\psscalebox{0.7}{$1$}}
\rput(6.75,0.65){\psscalebox{0.7}{$1$}}
\rput(9,0.65){\psscalebox{0.7}{$a$}}
\rput(12,0.65){\psscalebox{0.7}{$2a$}}
\rput(6,2.25){\psscalebox{0.7}{$\zeta_0$}}
\rput(12.8,1){\psscalebox{0.7}{$x_1$}}
\psline{-}(0,1)(3,0.25)(5.25,0.25)(6.75,1.75)(9,1.75)(12,1)
\end{pspicture*}
\caption{Test function $\zeta_0$}
\label{fig:zeta}
\end{figure}

\begin{proof}[Proof of Lemma \ref{lem:zeta}]
First we show
\begin{align}\label{eq:decompdirenzeta}
\int_{\R[2]} \lvert \nabla \zeta \rvert^2 dx = 2\,\left( \int_{\R} \left\lvert \bigl\lvert \tfrac{d}{dx_1} \bigr\rvert^\frac{1}{2} \zeta_0 \right\rvert^2 dx_1 + \int_{\R} \bigl\lvert \tfrac{d}{dx_1} \, \zeta_0 \bigr\rvert^2 dx_1 \right),
\end{align}
where we define 
\begin{align*}
  \int_{\R} \bigl\lvert \lvert \tfrac{d}{dx_1} \rvert^\frac{1}{2} f \bigr\rvert^2 dx_1 \defas \int_{\R} \lvert k_1 \rvert \, \lvert \ft_{x_1} f \rvert^2 dk_1 \in [0,\infty], \quad f \in L^2(\R).
\end{align*}

Indeed, for the contribution from $\Omega$ we simply have
\begin{align*}
\int_{\Omega} \lvert \nabla \zeta \rvert^2 dx \stackrel{\zeta=\zeta_0\, \textrm{in}\, \Omega}{=} \int_{\Omega} \lvert \nabla \zeta_0 \rvert^2 dx
 \stackrel{\zeta_0=\zeta_0(x_1)}{=}2 \int_{\R} \bigl\lvert \tfrac{d}{dx_1} \zeta_0 \bigr\rvert^2 dx_1.
\end{align*}
Moreover, the contribution from $\R[2] \setminus \Omega$ can be computed using \eqref{eq:outerzeta}:
\begin{align*}
\begin{split}
\MoveEqLeft[3] \int_{\R[2]\setminus\Omega} \lvert \nabla \zeta \rvert^2 dx = 2 \int_{\R \times (1,\infty)} \lvert \nabla \zeta \rvert^2 dx\\
= &2 \int_1^\infty \int_{\R} \lvert \ft_{x_1} \! (\nabla \zeta) (k_1,x_3) \rvert^2 dk_1 \, dx_3\\
= &2 \int_1^\infty \int_{\R} \bigl\lvert k_1 \,\ft_{x_1} \! (\zeta) (k_1,x_3) \bigr\rvert^2 + \bigl\lvert \partial_{x_3} \ft_{x_1} \! (\zeta)(k_1,x_3) \bigr\rvert^2 dk_1 \, dx_3\\
= &2 \int_1^\infty \int_{\R} \bigl\lvert \lvert k_1 \rvert \,\ft_{x_1} \! (\zeta) (k_1,x_3) \bigr\rvert^2 + \bigl\lvert \lvert k_1 \rvert \ft_{x_1} \! (\zeta)(k_1,x_3) \bigr\rvert^2 dk_1 \, dx_3\\
\stackrel{\eqref{eq:outerzeta}}{=} &2 \int_{\R} \lvert k_1 \rvert \, \lvert \ft_{x_1} \! (\zeta_0)\rvert^2 \; \bigg(\underbrace{\int_1^\infty 2 \,\lvert k_1 \rvert \, e^{-2\,\lvert k_1 \rvert \, (x_3 - 1)} dx_3}_{=1}\bigg) \; dk_1\\
= &2 \int_{\R} \Bigl\lvert \bigl\lvert \tfrac{d}{dx_1} \bigr\rvert^\frac{1}{2} \zeta_0 \Bigr\rvert^2 dx_1.
\end{split}
\end{align*}

Therefore \eqref{eq:decompdirenzeta} is established.

To prove \eqref{new4}, one first observes that $\int_{\R} \lvert \tfrac{d}{dx_1} \zeta_0 \rvert^2 dx_1$ remains bounded as $a\tou \infty$, so that the
leading-order contribution to  \eqref{eq:decompdirenzeta} is given by the homogeneous $\dot{H}^{\frac{1}{2}}$ norm of $\zeta_0$.
Recall that the $\dot{H}^\frac{1}{2}$ norm can be expressed as
\begin{align}\label{h1trac}
\int_{\R} \Bigl\lvert \bigl\lvert \tfrac{d}{dx_1} \bigr\rvert^\frac{1}{2} \zeta_0 \Bigr\rvert^2 dx_1 = \min \left\{ \int_{\R\times\R_+} \lvert \nabla \bar{\zeta} \rvert^2 dx \with \bar{\zeta}\in \dot{H}^1(\R\times\R_+),\, \bar{\zeta}(\cdot,0) = \zeta_0 \right\}.
\end{align}
Therefore, to estimate \eqref{h1trac}, we choose an admissible function $\bar \zeta$:
\begin{align*}
\bar{\zeta}(x) = \bar{\zeta}(r,\theta) \defas \zeta_0(r) \, \varphi(\theta), \quad x\in \R\times\R_+,
\end{align*}
where $(r,\theta)$ denote the polar coordinates of $x \in \R\times\R_+$ and $\varphi \colon [0, \pi] \to [-1,1]$ is given by
\begin{align*}
\varphi(\theta) =
1-\frac{2}{\pi} \theta, \quad 0\leq \theta \leq \pi.\end{align*}
Observe that indeed $\bar{\zeta}(\cdot,0) = \zeta_0$ in $\R$ (since $\zeta_0$ is odd and $\varphi(0)=-\varphi(\pi)=1$). Therefore, we may estimate
\begin{align*}
\int_{\R} \Bigl\lvert \bigl\lvert \tfrac{d}{dx_1} \bigr\rvert^\frac{1}{2} \zeta_0 \Bigr\rvert^2 dx_1 
&\stackrel{\eqref{h1trac}}{\leq}  \int_{\R\times\R_+} \lvert \nabla \bar{\zeta} \rvert^2 dx =  \int_0^{\pi}\int_0^\infty \!\Bigl( \lvert \tfrac{\partial}{\partial r}\bar{\zeta} \rvert^2 + \lvert \tfrac{1}{r} \tfrac{\partial}{\partial \theta} \bar{\zeta} \rvert^2 \Bigr) rdr\,d\theta\\[2ex]
&= \int_0^{\pi} \! \varphi^2 \, d\theta \underbrace{\int_0^\infty \! \lvert \tfrac{d}{dr} \zeta_0 \rvert^2 rdr}_{= \co(1)}  +\underbrace{\int_0^{\pi} \! \lvert \tfrac{d}{d\theta} \varphi \rvert^2 d\theta}_{=\frac{4}{\pi}} \underbrace{\int_0^\infty \! \zeta_0^2 \tfrac{dr}{r}}_{= \ln a + \co(1)}\\
&= \frac{4}{\pi} \ln a + \co(1),
\end{align*}
which yields the asserted scaling.
\end{proof}

\textbf{Step 3:} \textit{Stray-field and anisotropy energy.} With the test function constructed in Step~2 we can establish the relation between $\lambda\,\ES(\alpha-\theta_m)$ and stray-field/anisotropy energy. First we use the definition of $\zeta$ to rewrite
\begin{align*}
\MoveEqLeft \dashint_{-1}^1 \dashint_{-1}^1 \! \bigl( m_{1,\eta} - \cos\alpha \bigr) \, dx_1 \, dx_3\\
&= \tfrac{1}{4} \; \int_{-1}^1 \int_{-1}^1 \! \bigl( m_{1,\eta} -\cos\alpha \bigr) \, \underbrace{\partial_{x_1} \zeta}_{= 1} \, dx_1 \, dx_3\\
&= \tfrac{1}{4} \; \int_\Omega \! \bigl( m_{1,\eta} -\cos\alpha \bigr) \, \partial_{x_1} \zeta \, dx + \tfrac{1}{4a} \int_{\left((-2a,-a) \cup (a,2a)\right) \times (-1,1)} \bigl(m_{1,\eta} - \cos\alpha \bigr) dx\\
&\stackrel{\zeta=\zeta(x_1)}{=} \tfrac{1}{4} \; \int_\Omega \! m'_\eta \cdot \nabla\zeta \, dx + \tfrac{1}{4a} \int_{\left((-2a,-a) \cup (a,2a)\right) \times (-1,1)} \bigl(m_{1,\eta} - \cos\alpha \bigr) dx\\ 
&\stackrel{\eqref{eq:maxwell_2D}}{=} - \tfrac{1}{4} \, \int_{\R[2]} \! h(m_\eta) \cdot \nabla \zeta \, dx + \tfrac{1}{4a} \int_{\left((-2a,-a) \cup (a,2a)\right) \times (-1,1)} \bigl(m_{1,\eta} - \cos\alpha \bigr) dx \\
&\leq  \tfrac{1}{4} \left( \int_{\R[2]} \lvert h(m_\eta) \rvert^2 dx \right)^{\frac{1}{2}} \, \left( \int_{\R[2]} \lvert \nabla \zeta \rvert^2 dx \right)^\frac{1}{2} + \tfrac{1}{4a} \left( 4a \right)^\frac{1}{2} \left( \int_\Omega \! \bigl( m_{1,\eta} - \cos\alpha \bigr)^2 dx \right)^\frac{1}{2}\\
&\stackrel{\text{Lemma } \ref{lem:zeta}}{\leq} \left( \left(\tfrac{1}{2\pi} \, \ln a + \co(1) \right)\, \int_{\R[2]} \lvert h(m_\eta) \rvert^2 dx \right)^{\frac{1}{2}} + \left( \tfrac{1}{4a} \int_\Omega \! \bigl(m_{1,\eta} - \cos\alpha \bigr)^2 dx \right)^\frac{1}{2},
\end{align*}
as $a\tou \infty$.
If we now apply
\begin{align*}
\sqrt{\alpha} + \sqrt{\beta} \leq \sqrt{\alpha(1+\delta) + \delta^{-2} \beta},
\end{align*}
which holds\footnote{Use $\sqrt{\alpha\beta} \leq \frac{\delta}{2} \alpha + \frac{1}{2\delta} \beta$ and $\frac{1}{\delta} \leq \frac{1}{\delta^2} - 1$.} for $0 < \delta \leq \frac{1}{2}$, $\alpha,\beta\geq0$, to
\begin{align*}
\alpha &= \left( \tfrac{1}{2\pi} \, \ln a + \co(1) \right) \, \int_{\R[2]} \lvert h(m_\eta) \rvert^2 dx,\\
\beta &= \tfrac{1}{4a} \int_\Omega \! \bigl( m_{1,\eta} - \cos\alpha \bigr)^2 dx,
\end{align*}
we find
\begin{align*}
\left( \dashint_{-1}^1 \dashint_{-1}^1 \bigl( m_{1,\eta} - \cos\alpha \bigr) dx \right)^2 \leq \; &\left(\tfrac{1}{2\pi} \ln a +\co(1) \right)(1+\delta) \int_{\R[2]} \lvert h(m_\eta) \rvert^2 dx\\
&+\tfrac{1}{4\delta^2} \, \tfrac{1}{a} \int_\Omega \bigl(m_{1,\eta} - \cos\alpha\bigr)^2 dx.
\end{align*}
Now, for $\delta\in (0, \frac 1 2]$ fixed, choose $a=a(\eta)$ such that
\begin{align*}
\tfrac{1}{4\delta^2 a} = \lambda^{-1} \tfrac{1+\delta}{2\pi} \, \eta.
\end{align*}
This implies $a\tou \infty$ as $\eta\tod 0$ and $\ln a=\ln\tfrac{1}{\eta} + \co(1)$ as $\eta\tod 0$.

Note that $\lambda \ln\tfrac{1}{\eta} \int_{\R[2]} \lvert h(m_\eta) \rvert^2 dx \leq E_\eta(m_\eta)$ is uniformly bounded and thus $\int_{\R[2]} \lvert h(m_\eta) \rvert^2 dx \to 0$ as $\eta\tod 0$. Together with $\cos\theta_m = \dashint_{-1}^1 \! \bar{m}_1 \, dx_1 \leftarrow \dashint_{-1}^1 \! \bar{m}_{1,\eta} \, dx_1$, we obtain:
\begin{align*}
\bigl( \cos\theta_m - &\cos\alpha \bigr)^2 = \left( \liminf_{\eta \tod 0} \dashint_{-1}^1 \dashint_{-1}^1 \bigl(m_{1,\eta} - \cos\alpha \bigr) dx_1\,dx_3 \right)^2\\
&\leq \liminf_{\eta \tod 0} \left( \tfrac{1}{2\pi} (1+\delta) \ln a \int_{\R[2]} \lvert h(m_\eta) \rvert^2 dx + \tfrac{1}{4\delta^2a} \int_\Omega \bigl(m_{1,\eta} - \cos\alpha \bigr)^2 dx \right)\\
&= (1+\delta) \, \liminf_{\eta \tod 0} \left( \tfrac{1}{2\pi} \ln\tfrac{1}{\eta} \int_{\R[2]} \lvert h(m_\eta) \rvert^2 dx + \lambda^{-1} \tfrac{1}{2\pi}\, \eta \int_\Omega \bigl( m_{1,\eta} - \cos\alpha \bigr)^2 dx \right).
\end{align*}
Letting now $\delta \tod 0$, it follows: 
\begin{align}\label{eq:sfestimate}
2\pi \lambda \bigl(\cos\theta_m - \cos\alpha\bigr)^2 \leq \liminf_{\eta \tod 0} \left( \lambda \ln\tfrac{1}{\eta} \int_{\R[2]} \lvert h(m_\eta) \rvert^2 dx + \eta \int_\Omega \bigl( m_{1,\eta}- \cos\alpha \bigr)^2 dx \right).
\end{align}

\textbf{Step 4:} \textit{Conclusion.} By combining \eqref{eq:exchangeestimate} and \eqref{eq:sfestimate} one sees
\begin{align*}
E_0(m) &\leq \liminf_{\eta \tod 0} \int_\Omega \lvert \nabla m_\eta \rvert^2 dx + \liminf_{\eta \tod 0} \left( \lambda \ln\tfrac{1}{\eta} \int_{\R[2]} \lvert h(m_\eta) \rvert^2 dx + \eta \int_\Omega \bigl( m_{1,\eta} -\cos\alpha \bigr)^2 dx \right)\\
&\leq \liminf_{\eta \tod 0} E_\eta(m_\eta),
\end{align*}
i.e. the lower bound \eqref{eq:lowerbd} is proven. \qed

\subsection{Upper bound}
For each $m\in X_0$ we construct a recovery sequence $\{m_\eta\}_{\eta\tod 0} \subset X^\alpha$ such that $m_\eta \to m$ in $\dot{H}^1(\Omega)$ and \eqref{eq:upperbd} holds. 
For that, in the general case $\theta_m \not\in \{0,\alpha,\pi\}$, the basic guideline will be a decomposition of $\Omega$ into several parts (as shown in Figure \ref{fig:domaindec}):
\begin{figure}[htbp]
\centering
\begin{pspicture*}(-0.5,-1)(12.5,3.5)
\psline{-}(0,0.15)(0,1.85)
\psline{-}(3.55,0.15)(3.55,1.85)
\psline{-}(4.95,0.15)(4.95,1.85)
\psline{-}(7.05,0.15)(7.05,1.85)
\psline{-}(8.45,0.15)(8.45,1.85)
\psline{-}(12,0.15)(12,1.85)
\rput(-0.1,0.6){\psscalebox{0.65}{\psframebox*[framearc=0.3]{$-\tfrac{a+1}{\eta}$}}}
\rput(3.45,0.6){\psscalebox{0.65}{\psframebox*[framearc=0.3]{$-(a+1)$}}}
\rput(4.95,0.6){\psscalebox{0.65}{\psframebox*[framearc=0.3]{$-a$}}}
\rput(7.05,0.6){\psscalebox{0.65}{\psframebox*[framearc=0.3]{$a$}}}
\rput(8.45,0.6){\psscalebox{0.65}{\psframebox*[framearc=0.3]{$a + 1$}}}
\rput(12,0.6){\psscalebox{0.65}{\psframebox*[framearc=0.3]{$\tfrac{a+1}{\eta}$}}}
\psline{-}(-0.25,1.85)(12.25,1.85)
\psline{-}(-0.25,0.15)(12.25,0.15)
\rput(5.675,0.15){\psframebox*[framearc=0.3]{$-1$}}
\rput(5.825,1.85){\psframebox*[framearc=0.3]{$1$}}
\psline{->}(-0.5,1)(12.5,1)
\psline{->}(6,-0.25)(6,2.25)
\psline{-}(0,0.9)(0,1.1)
\psline{-}(3.55,0.9)(3.55,1.1)
\psline{-}(5.9,1.85)(6.1,1.85)
\psline{-}(5.9,0.15)(6.1,0.15)
\psline{-}(8.45,0.9)(8.45,1.1)
\psline{-}(12,0.9)(12,1.1)
\psline{|-|}(0,-0.5)(3.54,-0.5)
\rput*(1.775,-0.5){\psscalebox{0.8}{$\Omega_T$}}
\psline{|-|}(3.56,-0.5)(4.94,-0.5)
\rput*(4.275,-0.5){\psscalebox{0.8}{$\Omega_I$}}
\psline{|-|}(4.96,-0.5)(7.04,-0.5)
\rput*(6,-0.5){\psscalebox{0.8}{$\Omega_A$}}
\psline{|-|}(7.06,-0.5)(8.44,-0.5)
\rput*(7.775,-0.5){\psscalebox{0.8}{$\Omega_I$}}
\psline{|-|}(8.46,-0.5)(12,-0.5)
\rput*(10.225,-0.5){\psscalebox{0.8}{$\Omega_T$}}
\rput(1.775,1){\psscalebox{0.8}{\psframebox*[framearc=0.3]{N\'eel tails}}}
\rput(4.275,1){\psscalebox{0.8}{\psframebox*[framearc=0.3]{Interp.}}}
\rput(6,1){\psscalebox{0.8}{\psframebox*[framearc=0.3]{Asym. wall}}}
\rput(7.775,1){\psscalebox{0.8}{\psframebox*[framearc=0.3]{Interp.}}}
\rput(10.225,1){\psscalebox{0.8}{\psframebox*[framearc=0.3]{N\'eel tails}}}
\rput{45}(7.6,2.6){\psscalebox{0.65}{$m_\eta\approx\left(\begin{smallmatrix}\cos\theta_m\\\sin\theta_m\\0\end{smallmatrix}\right)$}}
\rput{45}(9.05,2.6){\psscalebox{0.65}{$m_\eta=\left(\begin{smallmatrix}\cos\theta_m\\\sin\theta_m\\0\end{smallmatrix}\right)$}}
\end{pspicture*}
\caption{Construction of recovery sequence}
\label{fig:domaindec}
\end{figure}
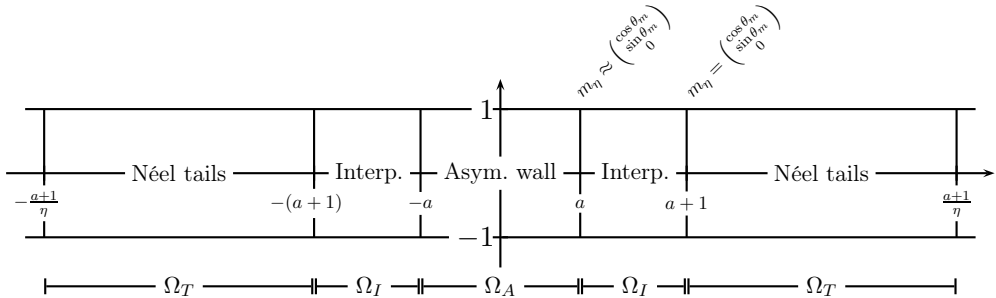
We consider the regions
\begin{align*}
\Omega_A &\defas [-a,a] \times (-1,1),\\
\Omega_I &\defas \Bigl([-(a+1),-a] \cup [a,a+1] \Bigr)\times(-1,1),\\
\Omega_T &\defas \Bigl([-\tfrac{a+1}{\eta},-(a+1)] \cup [a+1,\tfrac{a+1}{\eta}] \Bigr)\times(-1,1),
\end{align*}
where $a$ is a parameter of order $\ln^\frac{3}{2}\! \frac 1 \eta\gg 1$ (to be chosen explicitly at Step 1 below).

\begin{itemize}
  \item The (core) region $\Omega_A$ stands for the asymmetric part of the transition layer $m_\eta$: Here, $m'_\eta$ is of vanishing divergence (so, a stray-field free configuration) with an asymptotic angle transition from $-\theta_m$ to $\theta_m$ (as $\eta \tod 0$) 
so that the leading-order term is driven by the exchange energy. 

  \item The (tail) region $\Omega_T$ corresponds to the symmetric part of the transition layer $m_\eta$ that mimics the tails of a symmetric 
N\'eel wall. Here, the leading order term of the energy is driven by the stray field, the transition angle covering the range $[-\alpha, -\theta_m]$ (at the left) and $[\theta_m, \alpha]$ (at the right), respectively. 

  \item The (intermediate) region $\Omega_I$ is necessary for the transition between the core of $m_\eta$ and the tails of a N\'eel wall. 
This is because the asymmetric core of $m_\eta$ and the symmetric tails will not fit together perfectly (on $\Omega_A$, the angle transition is
$2\theta_m+o(1)$ and not exactly $2\theta_m$), however, this region only adds energy of order $\so(1)$.
\end{itemize}

So, let $m\in X_0$. Since $\theta_m \in [0,\pi]$, we need to treat three different cases:

\textbf{Case 1:} \textit{$\theta_m \not\in \{0, \alpha, \pi\}$.} We proceed in several steps:

\noindent {\bf Step 1:} \textit{Choice of $a$.}
We consider the $L^1(\R)$ positive function $E\colon\R\to \R_+$ defined by:
\begin{align*}
E(x_1) \defas \int_{-1}^1 \lvert \nabla m(x_1, x_3) \rvert^2 dx_3 \quad \textrm{ for a.e. } x_1\in \R.
\end{align*}

For $\eta \ll 1$, let $b = b(\eta) \defas \ln^\frac{3}{2}\tfrac{1}{\eta}$ (in fact, any choice $\ln^\gamma\! \tfrac{1}{\eta}$ with $\gamma \in (1,2)$ would work). We choose
$$a=a(\eta) \in [\tfrac{b}{2},b]$$
such that $a$ and $-a$ are Lebesgue points of $E$ and
\begin{align}\label{eq:choiceofa}
E(a)+E(-a) \leq \dashint_{\tfrac{b}{2}}^{b} \! E(x_1)+E(-x_1)\, dx_1 \leq \frac{2}{b} \int_\Omega \lvert \nabla m \rvert^2 dx = \frac{C_0}{b},
\end{align}
with $C_0 = 2 \int_\Omega \lvert \nabla m \rvert^2 dx < \infty$. In particular, the $\dot{H}^\frac{1}{2}$-trace of $m$ on the vertical lines $\{\pm a\}\times (-1,1)$ actually belongs to $H^1$. Since $\bar{m}_2 \mp \sin\theta_m \in {H}^1(\R_\pm)$ (due to $m_2(\pm\infty,\cdot)=\pm\sin\theta_m$), we also have
\begin{align*}
\bar{m}_2(\pm a)= \pm \sin\theta_m+o(1) \quad \textrm{as } \eta \tod 0.
\end{align*}
Recall that Sobolev's embedding theorem yields existence of $C>0$ such that
$$\|u\|_{L^\infty}\leq C \|\dds u\|_{L^2}, \quad \textrm{ for  every } u\in H_0^1\big((-1, 1)\big),$$ 
together with
$$\|u-\bar u\|_{L^\infty}\leq C \|\dds u\|_{L^2}, \quad \textrm{ for  every } u\in H^1\big((-1, 1)\big).$$ 
Therefore,
\begin{align}
\nonumber \sqrt{\frac{C_0}{b}} &\geq \biggl( \int_{-1}^1 \sum_{\sigma\in \{\pm1\}}\lvert \nabla m(\sigma a,x_3) \rvert^2 dx_3 \biggr)^\frac{1}{2}\\
\nonumber
&\geq \biggl(\sum_{\sigma\in \{\pm1\}} \int_{-1}^1 \lvert \partial_{x_3} m_1(\sigma a,x_3) \rvert^2 + \lvert \partial_{x_3} m_2(\sigma a,x_3) \rvert^2 + \lvert \partial_{x_3} m_3(\sigma a,x_3) \rvert^2 dx_3 \biggr)^\frac{1}{2}\\
\label{eq:valofm1m2m3}
&\geq C \sum_{\sigma\in \{\pm 1\}} \Bigl( \| m_1(\sigma a, \cdot) - \underbrace{\bar{m}_1(\sigma a)}_{=\cos\theta_m} \|_{L^\infty}+  
\|m_2(\sigma a,\cdot) - \bar{m}_2(\sigma a) \|_{L^\infty}+ \|m_3(\sigma a,\cdot)\|_{L^\infty} \Bigr).
\end{align}
It follows that $m_1(\pm a, x_3)=\cos\theta_m+o(1)$, $m_3(\pm a, x_3)=o(1)$ and 
$m_2(\pm a, x_3)= \bar{m}_2(\pm a)+o(1)= \pm\sin\theta_m+o(1)$ uniformly in $x_3\in (-1, 1)$ as $\eta\tod 0$ (since $b\to \infty$).

\noindent {\bf Step 2:} \textit{Definition of $m_\eta$.}

\begin{itemize}
  \item On $\Omega_A$, we choose that $m_\eta(x)=m(x)$ for every $x\in \Omega_A$.

  \item On the tail region $\{|x_1|\geq a+1\}$, we choose $m_\eta$ to be the $\mathbb{S}^1$-valued approximation of a N\'eel wall with a transition angle that
goes from $-\alpha$ to $-\theta_m$ (on the left) and from $\theta_m$ to $\alpha$ (on the right). More precisely, as in \cite{ignat09},  
let $m_\eta \colon \Omega\setminus (\Omega_A \cup \Omega_I) \to \mathbb{S}^1$ depend only on the $x_1$-direction and be given by
\begin{align*}
&\quad m_{1,\eta}(x_1,x_3) \defas
\begin{cases}
\cos\alpha + \frac{\cos\theta_m -\cos\alpha}{\ln \frac{1}{\eta}} \ln\bigl(\frac{a+1}{\eta  |x_1|} \bigr), & a+1\leq | x_1 | \leq \frac{a+1}{\eta},\,x_3\in(-1,1)\\
\cos\alpha, & \frac{a+1}{\eta}\leq |x_1|,\,x_3\in(-1,1),
\end{cases}\\
& \left\{\begin{aligned} m_{2,\eta}(x_1,x_3) &\defas \sgn(x_1) \sqrt{1-m^2_{1,\eta}(x_1,x_3)},\\
m_{3,\eta}(x_1,x_3) &\defas 0,
\end{aligned}\right\} \text{ on } \{a+1 \leq |x_1|\}\times (-1,1).
\end{align*}

  \item On the intermediate region $\Omega_I$,  i.e. for $a < \lvert x_1 \rvert <a+1$, we define $m_\eta$ by linear interpolation in $m_{3, \eta}$ and the phase $\phi_\eta$ of $(m_{1, \eta}, m_{2, \eta})$ (interpreted as complex number) between $\Omega_T$ and $\Omega_A$. For this, we choose $\eta$ sufficiently small, such that
\begin{align}\label{eq:niceeta}
\pm m_2(\pm a,\cdot)>\frac{\sin\theta_m}{2} > 0 \quad  \textrm{ in } (-1,1).
\end{align}

Therefore, there exists a unique phase
$\phi_\eta(\pm a,x_3)\in(0,\pi)$ of $(m_1,m_2)(\pm a,x_3)\in \R[2]\simeq \C$ such that 
\begin{align*}
(m_1 + i m_2)(\pm a,x_3) = \sqrt{1-m^2_3(\pm a,x_3)} \, \, e^{\pm i\phi_\eta(\pm a,x_3)} \quad  \textrm{ for every } x_3\in (-1,1).
\end{align*}
Observe that the function $\phi_\eta(\pm a,\cdot)$ depends on $\eta$ only through $a$. Recall that $m_{3,\eta}(\pm(a+1), \cdot)=0$ and $(m_{1,\eta} +i m_{2,\eta})(\pm(a+1), \cdot)=e^{\pm i\theta_{m}}$ so that we fix $\phi_\eta(\pm(a+1), \cdot)\defas \theta_m$ on 
$(-1, 1)$.
By linear interpolation, we then define $m_{\eta}\colon\Omega_I\to \mathbb{S}^2$ and $\phi_\eta\colon\Omega_I\to (0, \pi)$ by
\begin{align}\label{eq:defm3}
m_{3,\eta}(x) &\defas \bigl(1+a - \lvert x_1 \rvert \bigr) m_3(\pm a,x_3),\\
\nonumber
\phi_\eta(x)& \defas \bigl(1+a-\lvert x_1 \rvert\bigr) \phi_\eta(\pm a,x_3) + \bigl(\lvert x_1 \rvert - a \bigr) \theta_m,\\
\label{eq:defm1m2}
(m_{1,\eta} +i m_{2,\eta})(x) &\defas \sqrt{1-m^2_{3,\eta}(x)} \, \, e^{\pm i\phi_\eta(x)},\\\nonumber
\text{whenever } &a<\pm x_1<a+1, \, x_3\in (-1, 1).
\end{align}
\end{itemize}

Note that $m_\eta(\pm \infty, \cdot)=m^\pm_\alpha$ (in the sense of \eqref{convent}) since $m_\eta \neq m^\pm_\alpha$ only on the bounded set 
$\Omega_A\cup \Omega_I\cup \Omega_T$. We will show that $m_\eta$ has $H^1$ regularity on $\Omega_A$, $\Omega_I$ and $\Omega_T$. Moreover, the $H^\frac{1}{2}$-traces of $m_\eta$ on the vertical lines $\{\pm a\} \times (-1,1)$ and $\{\pm (a+1)\} \times (-1,1)$ do agree, so that finally 
$m_\eta\in \dot{H}^1(\Omega)$, i.e., $m_\eta\in X^\alpha$.

\textbf{Step 3:} \textit{Exchange energy estimate.}
We prove
\begin{align}\label{eq:exchangeenergy}
\int_\Omega \lvert \nabla m_\eta \rvert^2 dx \leq \int_\Omega \lvert \nabla m \rvert^2 dx + o(1).
\end{align}
Indeed,

\begin{itemize}
  \item on $\Omega_A$, we have that $m_\eta\equiv m$ so that
\begin{align}\label{eq:exchangeasym}
  \int_{\Omega_A} \lvert \nabla m_\eta \rvert^2 dx \leq \int_\Omega \lvert \nabla m \rvert^2 dx;
\end{align}

  \item on $\Omega_T$, since $\lvert m_{1,\eta} \rvert \leq \max \bigl( \lvert\cos\theta_m\rvert, \cos\alpha \bigr) =: \mu < 1$, we deduce
\begin{align}\label{eq:exchangeneel}
  \int_{-1}^1 \int_{a+1}^{\frac{a+1}{\eta}} \lvert \nabla m_\eta \rvert^2 dx &= 2 \int_{a+1}^{\frac{a+1}{\eta}} \frac{\lvert \tfrac{d}{dx_1}m_{1,\eta}(x_1) \rvert^2}{1-m_{1, \eta}^2(x_1)} dx_1\notag\\
&\leq 2 \bigl(\tfrac{\cos\theta_m-\cos\alpha}{\ln\frac{1}{\eta}}\bigr)^2 \int_{a+1}^{\frac{a+1}{\eta}} \frac{x_1^{-2}}{1-\mu^2} dx_1\notag\\
&\leq  \frac{C(\theta_m)}{b \, \ln^2\tfrac{1}{\eta}}=\so(1) \quad \textrm{as } \eta\tod 0,
\end{align}
where we used $a+1 \geq \frac{b}{2}$ in the last inequality;

  \item on the intermediate region $\Omega_I$, we use the following lemma:
\medskip
\begin{lem}
For $a\leq \pm x_1\leq a+1$ and $x_3\in (-1, 1)$, we have
\begin{align}\label{eq:partialsofm1m2m3}
\begin{split}
 (i)\quad  \left\lvert \partial_{x_1} m_{3,\eta}(x) \right\rvert^2 &\leq m^2_3(\pm a,x_3),\\
 (ii)\quad  \left\lvert \partial_{x_3} m_{3,\eta}(x) \right\rvert^2 &\leq \left\lvert \partial_{x_3} m_3(\pm a,x_3) \right\rvert^2,\\
  (iii)\quad \left\lvert \partial_{x_1}\left(\begin{smallmatrix}m_{1,\eta}\\m_{2,\eta}\end{smallmatrix}\right)(x) \right\rvert^2 &\leq m^2_3(\pm a,x_3) + \left\lvert \phi_\eta(\pm a,x_3) - \theta_m\right\rvert^2,\\
   (iv)\quad  \left\lvert \partial_{x_3}\left(\begin{smallmatrix}m_{1,\eta}\\m_{2,\eta}\end{smallmatrix}\right)(x) \right\rvert^2 &\leq 2 \left\lvert \partial_{x_3}\left(\begin{smallmatrix}m_1\\m_2\end{smallmatrix}\right)(\pm a,x_3) \right\rvert^2.
\end{split}
\end{align}
\end{lem}
\begin{proof}
  Inequalities $(i)$ and $(ii)$ immediately follow from the definition \eqref{eq:defm3} of $m_{3,\eta}$ on $\Omega_I$ .

To prove the remaining inequalities we use the identity $\lvert \partial_{x_i} \left( \rho(x) e^{i\varphi(x)} \right) \rvert^2 = \lvert \partial_{x_i} \rho(x) \rvert^2 + \lvert \rho(x) \partial_{x_i} \varphi(x) \rvert^2$ for real-valued functions $\rho$ and $\varphi$.
Therefore, for (iii), using that $\lvert m_{3,\eta}(x) \rvert \leq \lvert m_3(\pm a,x_3) \rvert \leq \frac{1}{2}$ for $\eta$ sufficiently small (see \eqref{eq:valofm1m2m3}), we deduce for $x\in \Omega_I$:
\begin{align*}
  \MoveEqLeft \left\lvert \; \partial_{x_1} (m_{1,\eta}+im_{2,\eta})(x) \right\rvert^2\\
  & = \underbrace{\tfrac{m_{3,\eta}^2(x)}{1-m_{3,\eta}^2(x)}}_{\leq 1}  \left\lvert \partial_{x_1} m_{3,\eta}(x) \right\rvert^2 + \underbrace{\left(1 - m_{3,\eta}^2(x) \right)}_{\leq 1} \left\lvert \partial_{x_1} \phi_\eta(x) \right\rvert^2\\
& \stackrel{(i)}{\leq}  m^2_3(\pm a,x_3) + \left\lvert \phi_\eta(\pm a,x_3) - \theta_m\right\rvert^2.
\end{align*}

Similarly, for $(iv)$, using that $t\mapsto \frac{t}{1-t}$ is increasing on $(0, 1)$, we obtain for the $x_3$-derivative of $m_{1,\eta}$ and $m_{2,\eta}$ and $x\in \Omega_I$:
\begin{align*}
  \MoveEqLeft \left\lvert \partial_{x_3}(m_{1,\eta}+im_{2,\eta})(x) \right\rvert^2\\
  &\leq \tfrac{m_{3,\eta}^2(x)}{1-m_{3,\eta}^2(x)} \left\lvert \partial_{x_3} m_{3,\eta}(x) \right\rvert^2 + \underbrace{\left(1 - m_{3,\eta}^2(x) \right)}_{\leq 1} \left\lvert \partial_{x_3} \phi_\eta(x) \right\rvert^2\\
&\stackrel{\eqref{eq:defm3}}{\leq} \tfrac{m_3^2(\pm a,x_3)}{1-m_3^2(\pm a,x_3)} \left\lvert \partial_{x_3} m_3(\pm a,x_3) \right\rvert^2 + \underbrace{2\left(1 - m_3^2(\pm a,x_3) \right)}_{\geq 1} \left\lvert \partial_{x_3} \phi_\eta(\pm a,x_3) \right\rvert^2\\
&\leq 2 \left( \tfrac{m_3^2(\pm a,x_3)}{1-m_3^2(\pm a,x_3)} \left\lvert \partial_{x_3}m_3(\pm a,x_3) \right\rvert^2 + \left(1 - m_3^2(\pm a,x_3) \right) \left\lvert \partial_{x_3} \phi_\eta(\pm a,x_3) \right\rvert^2 \right)\\
&= 2 \left\lvert \partial_{x_3}(m_1 +i m_2)(\pm a,x_3) \right\rvert^2.\qedhere
\end{align*}
\end{proof}

Note that for sufficiently small $\eta$ the function $\phi_\eta(\pm a,\cdot)\in(0,\pi)$ is bounded away from $0$ and $\pi$, such that by Lipschitz continuity of $\arccos$ and \eqref{eq:valofm1m2m3} we have
\begin{align}\label{eq:diffanglessmall}
\lvert \phi_\eta(\pm a,\cdot) - \theta_m \rvert^2 \leq \frac{C(\theta_m)}{b} \quad \textrm{ on } (-1, 1).
\end{align}

Therefore, after integrating \eqref{eq:partialsofm1m2m3} over $\Omega_I$, \eqref{eq:diffanglessmall}, \eqref{eq:choiceofa} and \eqref{eq:valofm1m2m3} show that
\begin{align}\label{eq:decayexchangetrans}
  \int_{\Omega_I} \lvert \nabla m_\eta \rvert^2 dx = \int_{\{a\leq |x_1|\leq a+1\}} \int_{-1}^1 \lvert \nabla m_\eta \rvert^2 dx_3\,dx_1 \leq \frac{C(\theta_m)}{b} = \so(1)
\textrm{ as } \eta\tod 0, \end{align}
which together with \eqref{eq:exchangeasym} and \eqref{eq:exchangeneel} implies \eqref{eq:exchangeenergy}.
\end{itemize}

\textbf{Step 4:} \textit{Stray-field energy estimate.}
We will prove that
\begin{align}\label{eq:strayfieldenergy}
\lambda \ln\tfrac{1}{\eta} \int_{\R[2]} \lvert h(m_\eta) \rvert^2 dx \leq 2\pi \,\lambda \, \bigl(\cos\theta_m - \cos\alpha \bigr)^2 + \so(1).
\end{align}
Indeed, we follow the arguments in \cite[Proof of Thm. 2(ii)]{ignat09} (see also \cite{ik10,io11}). First of all, recall that $m_{3, \eta}=0$ on $\partial \Omega$ and that $\nabla\cdot m'_\eta$ is supported in the compact set $\overline{\Omega_A\cup\Omega_I\cup\Omega_T}$. Therefore,
the stray field $h_\eta = -\nabla u_\eta$ with $u_\eta\in \dot{H}^1(\R[2])$  satisfies 
\begin{align*}
\int_{\R[2]} \! \nabla u_\eta  \cdot \nabla v \, dx = -\int_\Omega \! \nabla \cdot m_\eta' \, v \, dx \quad \forall v \in \dot{H}^1(\R[2]),
\end{align*}
so that by choosing $v := u_\eta$, we have:
\begin{align}
\nonumber
\MoveEqLeft \int_{\R[2]} \lvert \nabla u_\eta \rvert^2 dx = -\int_\Omega \nabla \cdot m_\eta' \, u_\eta \, dx\\
\label{str1}&= -\int_{\Omega_A} \underbrace{\nabla \cdot m'}_{=0} \, u_\eta \, dx - \int_{\Omega_I} \nabla \cdot m_\eta' \, u_\eta \, dx - \int_{\Omega_T} \nabla \cdot m_\eta' \, u_\eta \, dx.
\end{align}

\begin{itemize}
  \item On $\Omega_I$, since $m_{3,\eta} = 0$ on $\partial \Omega$ and $\bar{m}_{1,\eta}(\pm a) = \cos\theta_m$ as well as $m_{1,\eta}(\pm (a+1),\cdot) =\cos\theta_m$, we have
\begin{align*}
  \int_{(\Omega_I)_+} \! \nabla \cdot m_\eta' \, dx = \int_{(\partial \Omega_I)_+} \! m_\eta' \cdot \nu \, d\hm^1(x) = \bar{m}_{1,\eta}(a+1)-\bar{m}_{1,\eta}(a) = 0
\end{align*}
and similarly,
$\int_{(\Omega_I)_-} \! \nabla \cdot m_\eta' \, dx = 0$,
where $\nu$ denotes the outer unit normal to $\Omega_I$ and we use the notation $(\Omega_I)_+=\Omega_I\cap\{x_1\geq 0\}$, $(\Omega_I)_-=\Omega_I\cap\{x_1\leq 0\}$.
Therefore we may subtract the averages $$\textrm{$\bar{u}_\eta^+ = \dashint_{(\Omega_I)_+}  \! u_\eta \, dx\quad $ and $\quad \bar{u}_\eta^- = \dashint_{(\Omega_I)_-}\!u_\eta\,dx$}$$ of $u_\eta$ over the left and right parts of $\Omega_I$ such that by \eqref{eq:decayexchangetrans}, Cauchy-Schwarz and Poincar\'e-Wirtinger's inequality, we deduce:
\begin{align}
  \nonumber \int_{\Omega_I} \! \nabla \cdot m_\eta' \, u_\eta \, dx &= \int_{(\Omega_I)_-} \! \nabla \cdot m_\eta' \, \left(u_\eta - \bar{u}_\eta^- \right) dx + \int_{(\Omega_I)_+} \! \nabla \cdot m_\eta' \, \left(u_\eta - \bar{u}_\eta^+ \right) dx\\
 \nonumber &\leq \left( \int_{\Omega_I} \lvert \nabla \cdot m_\eta' \rvert^2 dx\right)^\frac{1}{2} \; \left( \int_{(\Omega_I)_-} \lvert u_\eta - \bar{u}_\eta^- \rvert^2 dx \right)^\frac{1}{2}\\
 \nonumber &+ \left( \int_{\Omega_I} \lvert \nabla\cdot m_\eta' \rvert^2 dx\right)^\frac{1}{2} \; \left( \int_{(\Omega_I)_+} \lvert u_\eta - \bar{u}_\eta^+ \rvert^2 dx \right)^\frac{1}{2}\\
 \nonumber &\stackrel{\eqref{eq:decayexchangetrans}}{\leq} \left( \frac{C}{b} \right)^\frac{1}{2} \; \left( \int_{\R[2]} \lvert \nabla u_\eta \rvert^2 dx \right)^\frac{1}{2}\\
 \label{str2} &= \Bigl( C \ln^{-\frac{3}{2}}\tfrac{1}{\eta}\Bigr)^\frac{1}{2} \; \left( \int_{\R[2]} \lvert \nabla u_\eta \rvert^2 dx \right)^\frac{1}{2}.
\end{align}

  \item On $\Omega_T$, the stray-field energy can be estimated using the trace characterization \eqref{h1trac} and the Cauchy-Schwarz inequality in Fourier space. Indeed, defining $m_{1,\eta}^\text{tails} \colon\Omega\to \R$ by setting $m_{1,\eta}^\text{tails}= m_{1,\eta}$ on $\Omega\setminus (\Omega_A\cup \Omega_I)$ and
$m_{1,\eta}^\text{tails} =\cos \theta_m$ on $\Omega_A\cup \Omega_I$, we have  
\begin{align}
\nonumber
\Bigl\lvert\int_{\Omega_T} \! \nabla \cdot m_\eta'\, u_\eta\, dx\Bigr\rvert &=\Bigl\lvert \int_{\Omega_T} \! \tfrac{d}{dx_1} m_{1, \eta} \, u_\eta \, dx\Bigr\rvert\\
\nonumber&= \int_{-1}^1 \int_{\R} \! \tfrac{d}{dx_1} m_{1,\eta}^\text{tails} \, u_\eta \, dx_1 \, dx_3\\
\nonumber&\leq \int_{-1}^1 \int_{\R} \! \lvert k_1 \rvert \lvert \ft_{x_1} m_{1,\eta}^\text{tails}(k_1) \rvert \, \lvert \ft_{x_1} u_\eta(k_1,x_3) \rvert \, dk_1 \, dx_3\\
\nonumber&\leq \int_{-1}^1 \Biggl( \int_{\R} \! \lvert k_1 \rvert \lvert \ft_{x_1} m_{1,\eta}^\text{tails}(k_1) \rvert^2 \, dk_1 \int_{\R} \lvert k_1 \rvert \lvert \ft_{x_1} u_\eta(k_1,x_3) \rvert^2 \, dk_1 \Biggr)^\frac{1}{2} dx_3\\
\label{str3}
&\stackrel{\eqref{h1trac}}{\leq} \left( \int_{\R[2]} \lvert \nabla u_\eta \rvert^2 dx \right)^\frac{1}{2} \left( 2 \int_{\R} \left\lvert \lvert \tfrac{d}{dx_1} \rvert^\frac{1}{2} m_{1,\eta}^\text{tails} \right\rvert^2 dx_1 \right)^\frac{1}{2}.
\end{align}

By considering the radial extension $M_{1,\eta}(x) = m_{1,\eta}^\text{tails}(\lvert x \rvert)$ of $m_{1,\eta}^\text{tails}$ on $\R\times\R_+$, which is possible since $m_{1,\eta}^\text{tails}$ is even, and using polar coordinates we then can estimate
\begin{align*}
  \int_{\R} \left\lvert \lvert \tfrac{d}{dx_1} \rvert^\frac{1}{2} m_{1,\eta}^\text{tails} \right\rvert^2 dx_1 &\stackrel{\eqref{h1trac}}{\leq} \int_{\R\times\R_+} \lvert \nabla M_{1,\eta} \rvert^2 dx\\
  &\leq \pi \int_{a+1}^{\frac{a+1}{\eta}} \lvert \tfrac{d}{dx_1} m_{1,\eta}^\text{tails} \rvert^2 \, x_1 dx_1\\
&= \pi \frac{\bigl( \cos\theta_m-\cos\alpha \bigr)^2}{\ln^2\frac{1}{\eta}} \underbrace{\int_{a+1}^{\frac{a+1}{\eta}} \! \tfrac{1}{x_1} \, dx_1}_{= \ln\frac{1}{\eta}}\\
&= \pi \frac{\bigl( \cos\theta_m - \cos\alpha \bigr)^2}{\ln\frac{1}{\eta}}.
\end{align*}
\end{itemize}

Collecting \eqref{str1}, \eqref{str2}, and \eqref{str3}, it follows that
\begin{align*}
\lambda \ln\tfrac{1}{\eta} \int_{\Omega} \lvert \nabla u_\eta \rvert^2 dx &\leq \lambda \ln\tfrac{1}{\eta} \biggl( C\ln^{-\frac{3}{4}}\tfrac{1}{\eta} + \Bigl( 2\pi \frac{( \cos\theta_m - \cos\alpha )^2}{\ln\frac{1}{\eta}} \Bigr)^\frac{1}{2} \biggr)^2\\
&\leq 2\pi\lambda \bigl( \cos\theta_m - \cos\alpha\bigr)^2 + C \ln^{-\frac{1}{4}} \tfrac{1}{\eta},
\end{align*}
i.e. \eqref{eq:strayfieldenergy}.\footnote{In fact, this motivates the choice $b=\ln^\gamma\frac{1}{\eta}$ with $\gamma>1$.}

\textbf{Step 5:} \textit{Anisotropy energy estimate.}
Finally, we prove
\begin{align}\label{eq:anisotropyenergy}
  \eta \int_\Omega \bigl( m_{1,\eta} - \cos\alpha \bigr)^2 + m_{3,\eta}^2\, dx = \so(1).
\end{align}

Indeed, since $a\sim b=\ln^\frac{3}{2}\tfrac{1}{\eta}$, on $\Omega_A\cup \Omega_I$ we have
\begin{align*}
  \eta \int_{-(a+1)}^{a+1} \int_{-1}^1 \bigl( m_{1,\eta} -\cos\alpha \bigr)^2 + m_{3,\eta}^2 \, dx_3\,dx_1 \leq C \eta \ln^\frac{3}{2}\tfrac{1}{\eta} = \so(1),
\end{align*}
and on $\Omega_T$, $m_{3, \eta}=0$ so that\footnote{Note that here it is important to have $b=\ln^\gamma\frac{1}{\eta}$ with $\gamma<2$.}
\begin{align*}
  \MoveEqLeft \eta \int_{\Omega_T} \! \bigl( m_{1,\eta} -\cos\alpha \bigr)^2 + m_{3,\eta}^2 \, dx\\
&= 4\eta \, \frac{\bigl( \cos\theta_m - \cos\alpha \bigr)^2}{\ln^2\tfrac{1}{\eta}} \int_{a+1}^{\frac{a+1}{\eta}} \! \ln^2(\tfrac{a+1}{\eta \, x_1})\, dx_1\\
&
\stackrel{y=\frac{\eta x_1}{a+1}}{\leq} C b \, \frac{\bigl( \cos\theta_m - \cos\alpha \bigr)^2}{\ln^2\tfrac{1}{\eta}} \underbrace{\int_\eta^{1} \! \ln^2 y\, dy}_{=\co(1)}\\
&\leq C \ln^{-\frac{1}{2}}\tfrac{1}{\eta} = \so(1).
\end{align*}

Moreover, on $\Omega\setminus (\Omega_A \cup \Omega_I \cup \Omega_T)$ we have $\bigl( m_{1,\eta} - \cos\alpha \bigr)^2 + m_{3,\eta}^2 = 0$ so that 
\eqref{eq:anisotropyenergy}
 holds.

\textbf{Step 6:} \textit{Conclusion.} Combining
\eqref{eq:exchangeenergy}, \eqref{eq:strayfieldenergy} and \eqref{eq:anisotropyenergy}, it follows that
\begin{align*}
E_\eta(m_\eta) &= \int_\Omega \lvert \nabla m_\eta \rvert^2 dx + \lambda \ln\tfrac{1}{\eta} \int_{\R[2]} \lvert h(m_\eta) \rvert^2 dx + \eta \int_\Omega 
\bigl(m_{1,\eta} - \cos\alpha \bigr)^2 + m_{3,\eta}^2\, dx\\
&\leq \int_\Omega \lvert \nabla m \rvert^2 dx + 2\pi\,\lambda\,\bigl( \cos\theta_m - \cos\alpha \bigr)^2 + \so(1) = E_0(m) + \so(1),
\end{align*}
which is \eqref{eq:upperbd}. Finally, let us prove that $m_\eta\to m$ in $\dot{H}^1(\Omega)$. First, observe that by construction, $m_\eta\equiv m$ on $\Omega_A$. Therefore, since $\bigcup_{\eta\tod 0}\Omega_A = \Omega$, $m_\eta\to m$ in $L^2_\text{loc}(\Omega)$. Moreover, \eqref{eq:exchangeenergy} implies that $\{m_\eta\}_{\eta\tod 0}$ is uniformly bounded in $\dot{H}^1(\Omega)$, so that $m_\eta\wto m$ in $\dot{H}^1(\Omega)$. By weak lower-semicontinuity of $\lVert \cdot \rVert_{L^2(\Omega)}$ and \eqref{eq:exchangeenergy}, one obtains
$\|\nabla m_\eta\|_{L^2(\Omega)}\to \|\nabla m\|_{L^2(\Omega)}$ and concludes that $m_\eta\to m$ in $\dot{H}^1(\Omega)$. 

\textbf{Case 2:} \textit{$\theta_m\in \{0,\pi\}$.} 
By Remark~\ref{rem_zero}, $m$ is constant, such that its exchange energy does not contribute to $E_0(m)$. Thus, we have to construct a sequence $m_\eta$ of asymptotically vanishing exchange energy, whose stray-field and anisotropy energy converge to $2\pi\,\lambda \, \bigl(\cos\theta_m - \cos\alpha\bigr)^2$. The function $m_\eta$ from Case~1 is a good candidate for the second property. However, if $\theta_m\in\{0,\pi\}$, it does not belong to $H^1(\Omega)$, since then $1-m^2_{1, \eta}$ behaves linearly w.r.t. the distance to the set $\{m^2_{1,\eta}=1\}$ and \eqref{eq:exchangeneel} fails. Therefore, we are obliged to construct a transition region between the two tails where this behavior is corrected.

With these considerations, we define $m_\eta \colon\Omega\to \mathbb{S}^1$ by
\begin{align*}
m_{1,\eta}(x_1,x_3) &\defas \begin{cases}
\cos\theta_m-\frac{1}{4} \bigl( \cos\theta_m - \cos\alpha \bigr) \frac{\ln 2}{\ln\frac{1}{\eta}} x_1^2, &\lvert x_1 \rvert \leq 2,\\
\cos\alpha + \bigl( \cos\theta_m - \cos\alpha \bigr) \frac{\ln(\frac{1}{\eta x_1})}{\ln{\frac{1}{\eta}}}, &2\leq\lvert x_1 \rvert \leq \frac{1}{\eta},\\
\cos\alpha, & \frac 1 \eta \leq |x_1|,
\end{cases}
\end{align*}
and again
\begin{align*}
m_{2,\eta}(x_1,x_3) &\defas \sgn(x_1) \sqrt{1-m_{1,\eta}^2(x_1,x_3)},\\
m_{3,\eta}(x_1,x_3) &\defas 0.
\end{align*}

Admissibility in $X^\alpha$ is obvious and one can then show, using the methods given above, that
\begin{align*}
\int_\Omega \lvert \nabla m_\eta \rvert^2 dx &\leq C \ln^{-1}\tfrac{1}{\eta},\\
\lambda \ln\tfrac{1}{\eta} \int_{\R[2]} \lvert \nabla u_\eta \rvert^2 dx &\leq 2\pi \, \lambda \, \bigl( \cos\theta_m - \cos\alpha \bigr)^2 + C \ln^{-\frac{1}{2}}\tfrac{1}{\eta},\\
\eta \int_\Omega \!\bigl(m_{1,\eta} - \cos\alpha\bigr)^2 + m_{3,\eta}^2\, dx &\leq C \ln^{-2}\tfrac{1}{\eta}.
\end{align*}
The strong convergence $m_\eta\to m$ in $\dot{H}^1(\Omega)$ also follows as in Step 6 of Case 1 by noting that the constructed transition layer $m_\eta$ has the property $m_\eta\to m=\bigl(\cos\theta_m,\pm\sin\theta_m, 0 \bigr)$ a.e. in $\Omega$. 

\textbf{Case 3:} \textit{$\theta_m=\alpha$.}
Since $m$ already has the correct boundary values, we can simply choose:
\begin{align*}
m_\eta \defas m.
\end{align*}

Admissibility of $m_\eta$ in $X^\alpha$ is clear and we can estimate:
\begin{align*}
\int_\Omega \lvert \nabla m_\eta \rvert^2 dx &= \int_\Omega \lvert\nabla m \rvert^2 dx,\\
\lambda \ln\tfrac{1}{\eta} \int_{\R[2]} \lvert h(m_\eta) \rvert^2 dx = 0 &=2\pi\, \lambda \, \bigl( \cos\alpha - \cos\alpha \bigr)^2,\\
\eta \int_\Omega \bigl(m_{1,\eta} - \cos\alpha \bigr)^2 + m_{3,\eta}^2 \, dx & = \so(1),
\end{align*}
since $m(\pm \infty,\cdot) = \bigl(\cos\alpha,\pm\sin\alpha, 0 \bigr)$.\hfill $\qquad\Box$

\bigskip

\begin{proof}[Proof of Corollary \ref{cor1}] The first equality in \eqref{eq:reducedmodel} is a direct consequence of the concept of $\Gamma$-convergence. Indeed, we know by Theorem \ref{thm:existence} that there exists a minimizer $m_\eta\in X^\alpha$ of $E_\eta$ for every $0<\eta< 1$. By Proposition \ref{prop:compactness1}, 
up to a subsequence and translation in $x_1$-direction, we have that $m_\eta \wto m$ in $\dot{H}^1(\Omega)$ for some $m\in X_0$ so that Theorem \ref{thm:gammalimlb} implies
$$\liminf_{\eta \downarrow 0} \min_{X^\alpha} E_\eta=\liminf_{\eta \downarrow 0} E_\eta(m_\eta)\geq E_0(m)\geq \min_{X_0} E_0.$$ On the other hand, Theorem \ref{thm:existence} also implies existence of a minimizer $m \in X_0$ of $E_0$. By Theorem \ref{thm:gammlimub}, there exists a family $\{\tilde m_\eta\}_{\eta\tod 0}\subset X^\alpha$ such that
$$\min_{X_0} E_0=E_0(m)\geq \limsup_{\eta \downarrow 0} E_\eta(\tilde m_\eta)\geq \limsup_{\eta \downarrow 0} \min_{X^\alpha} E_\eta.$$
Therefore, $\min_{X^\alpha} E_\eta\to \min_{X_0} E_0$ as $\eta\tod 0$. For the second equality in \eqref{eq:reducedmodel}, note that by Theorem \ref{thm:existence} one has
$$\min_{X_0} E_0=\min_{\theta\in[0,\pi]} \Bigl( \EA(\theta) + \lambda \, \ES(\alpha-\theta) \Bigr).$$ 
By Lemma~\ref{lem:lsceasym}, the minimum of the RHS is indeed attained. It remains to show that it is achieved for angles $\theta\in [0, \frac \pi 2]$. Indeed, let $\theta \in [0, \pi]$ be the minimizer of the above RHS and $m\in X_0$ with $\theta_m=\theta$ be the minimizer of $\EA(\theta)$. If $\theta \in (\frac \pi 2, \pi]$, then one considers $\tilde m\in X_0$ given by $\tilde m'\equiv -m'$ and $\tilde m_2\equiv m_2$, so that
$\theta_{\tilde m}=\pi-\theta_m\in [0, \frac \pi 2)$. Then $\tilde m$ and $m$ have the same exchange energy (i.e., $\EA(\theta_{\tilde m})=\EA(\theta_m)$)
and $\ES(\alpha-\theta_{\tilde m})\leq \ES(\alpha-\theta_{m})$ which proves \eqref{eq:reducedmodel}. Observe that the last inequality is strict whenever $\alpha\in (0, \frac \pi 2)$, so that for such angles $\alpha$ the minimal value of $E_0$ is achieved only for angles $\theta\in[0, \frac \pi 2]$.

Let us now prove the relative compactness in the strong $\dot{H}^1$-topology of minimizing families $\{m_\eta\}_{\eta\tod 0}\subset X^\alpha$ of $E_\eta$, i.e., which satisfy $E_\eta(m_\eta)\to \min_{X_0} E_0$. By Proposition~\ref{prop:compactness1}, 
up to a subsequence and translations in $x_1$-direction, we may assume that $m_\eta \wto m$ in $\dot{H}^1(\Omega)$ for some $m\in X_0$ so that Theorem \ref{thm:gammalimlb} implies \footnote{We use that $\limsup_n (a_n+b_n)\geq \limsup_n a_n+\liminf_n b_n$ for two bounded sequences $(a_n)$ and $(b_n)$.  }
\begin{align*}
\min_{X_0} E_0&=\lim_{\eta \tod 0} E_\eta(m_\eta)\\
&\geq \limsup_{\eta \tod 0} \int_\Omega \lvert \nabla m_\eta \rvert^2 dx\\
&\quad +\liminf_{\eta \tod 0} \left( \lambda \ln\tfrac{1}{\eta} \int_{\R[2]} \lvert h(m_\eta) \rvert^2 dx + \eta \int_\Omega \bigl( m_{1,\eta}- \cos\alpha \bigr)^2 dx \right)\\
&\stackrel{\mathclap{\eqref{eq:exchangeestimate}, \eqref{eq:sfestimate}}}{\geq}\quad \EA(\theta_m) + \lambda \ES(\alpha-\theta_m)\geq \min_{X_0} E_0.
\end{align*} 
Therefore, all above inequalities become equalities, in particular, $\lim_{\eta \tod 0} \int_\Omega \lvert \nabla m_\eta \rvert^2 dx=\int_\Omega \lvert \nabla m\rvert^2 dx$. Hence, one has $m_\eta \to m$ in $\dot{H}^1(\Omega)$, i.e. up to the subsequence taken in Proposition~\ref{prop:compactness1} and translations the entire family
$m_\eta$ converges strongly to $m$ in $\dot{H}^1(\Omega)$ where $m$ is a minimizer of $E_0$.
\end{proof}

%% file: construction.tex
\section*{Appendix}
\section{Construction of an asymmetric-Bloch type wall of arbitrary wall angle $\theta\in(0,\frac{\pi}{2}]$}\label{sec:const}
In this section we construct a stray-field free domain wall for any given angle $\theta \in (0,\frac{\pi}{2}]$. In particular, this shows that the set $X_0\cap X^\theta$ is non-empty, and we may apply the direct method in the calculus of variations to deduce existence of minimizers of $\EA$ (cf. Theorem~\ref{thm:existence}).

The construction we present here is of asymmetric Bloch-wall type in the following sense: The trace of $m\in \dot{H}^1(\Omega, \mathbb{S}^2)$ on the boundary 
$
\overline{Bdry}:=\partial\Omega\cup\bigg(\{\pm\infty\}\times[-1,1]\bigg) \cong \mathbb{S}^1$ (see \eqref{not_bdry})
has a non-zero topological degree. In fact, due to $m_3=0$ on $\partial\Omega$ as well as $m_3(\pm\infty,\cdot)=0$ 
(so, $(m_1, m_2):\overline{Bdry}\to \mathbb{S}^1$), one obtains (by the homeomorphism \eqref{not_bdry}) a map $\tilde{m} \in H^\frac{1}{2}(\mathbb{S}^1, \mathbb{S}^1)$ to which a topological degree can be associated (see, e.g., \cite{brezisnirenberg95}).

\medskip
\begin{rem} \label{rem5}
  \begin{enumerate}
    \item Asymmetric Bloch walls as well as the configuration we are about to construct do have a non-zero topological degree (e.g. $\pm1$) on $\overline{Bdry}$, whereas asymmetric N\'eel walls have degree $0$ on $\partial \Omega$. We make the following observation (see Lemma~\ref{lem-topo}): the non-vanishing topological degree of $(m_1, m_2):\overline{Bdry}\to \mathbb{S}^1$ nucleates at least one vortex singularity of $(m_1, m_3)$ (carrying a non-zero topological degree) as illustrated in Figure \ref{asym}.

    \item Note that for the angle $\theta=0$, by Remark \ref{rem_zero}, one has that $m\in X_0$ if and only if $m\in \{\pm {\bf e}_1\}$, so that $m$ has degree zero on $\partial \Omega$; thus, no asymmetric-Bloch type wall exists in this case.
    \item  An asymmetric-N\'eel type configuration $\tilde{m}\in X_0\cap X^\theta$, i.e. with $\deg \tilde{m} = 0$, can be obtained from any $m\in X_0\cap X^\theta$ using even reflection in $(m_1,m_2)$ and odd reflection in $m_3$ across one of the components of $\partial \Omega$ together with a rescaling in $x$ so that $\tilde m$ is defined on $\Omega$. However, starting with $m\in \mathcal{L}^\theta$ (introduced at Section \ref{outlook}), the reflected configuration has at least two vortices in $(m_1,m_3)$, so that it cannot have minimal energy. In \cite{dioasymwalls12}, we construct an asymptotically energy minimizing configuration of asymmetric-N\'eel type for small angles.
    \end{enumerate}
\end{rem}

The degree argument shows that we cannot expect a homotopy between asymmetric N\'eel and Bloch wall in the class of stray-field free walls. Hence, it is unclear how the nevertheless expected transition from asymmetric N\'eel to Bloch wall actually takes place.

\medskip
\begin{prop}\label{prop:constr}
  Given $\theta\in(0,\frac{\pi}{2}]$, there exists a map $m\colon\Omega \to \mathbb{S}^2$ with the following properties:
  \begin{itemize}
    \item $m \in \dot{H}^1(\Omega)$,
    \item $m(x_1,\cdot) = m^\pm_\theta$ for all $\lvert x_1 \rvert$ sufficiently large,
    \item $\nabla \cdot m' = 0$ in $\Omega$ and $m_3 = 0$ on $\partial\Omega$,
    \item $\deg(m\big\vert_{\partial\Omega})=-1$.
  \end{itemize}
\end{prop}

\begin{proof}
  To construct $m$, we will search for a stream function $\psi \colon \R[2] \to \R$ with the following properties:
  \begin{enumerate}
    \item $\psi \in C^3(\R[2])$ with $\lvert \nabla \psi \rvert \leq 1$ in $\R[2]$, \footnote{In fact, we will construct a smooth function $\psi$.}\label{eq:psic3n1}
    \item $\psi(x) = -(x_3+1)\cos\theta$ for all $\lvert x \rvert$ sufficiently large,\label{eq:psibcx1}
    \item $\psi(\cdot,-1)=0$ and $\psi(\cdot,1)=-2\cos\theta$ in $\R$,\label{eq:psibcx3}
    \item there exists a continuous curve $\gamma$, connecting the upper and lower components $\R\times\{+1\}$ and $\R\times\{-1\}$ of $\partial\Omega$, on which $\lvert \nabla \psi\rvert = 1$.\label{eq:psigamma}
  \end{enumerate}
  We then define $m$ according to
  \beq \label{def_magul}
    m'\defas\nabla^\perp\psi,\quad m_2(x) \defas
    \begin{cases}
      -\sqrt{1-\lvert \nabla \psi(x) \rvert^2},&\text{if $x$ is to the left of }\gamma,\\
      \sqrt{1-\lvert \nabla \psi(x) \rvert^2},&\text{if $x$ is to the right of }\gamma.
    \end{cases}
  \eeq
  Note that by Lemma~\ref{lem:m2h1} below, $m_2$ is Lipschitz continuous. Indeed, we remark that $D^2 (\lvert\nabla\psi\rvert^2)$ is globally bounded since $\lvert \nabla \psi \rvert = \cos\theta$ outside of a compact set and apply Lemma~\ref{lem:m2h1} to $f=1-\lvert \nabla \psi\rvert^2\geq 0$.

  Figure~\ref{fig:abwc} shows the level lines of $\psi$, in Figure~\ref{fig:abwv} the region around the vortex is enlarged.
  \begin{figure}[t]
    \centering
    \begin{pspicture}(-5,-1)(5,1)
      \psline(-5,1)(5,1)
      \psline(-5,-1)(5,-1)
      \psline(-4,-1)(-4,1)
      \psline(4,-1)(4,1)
      \psline(-1,-1)(-1,1)
      \psline(1,-1)(1,1)
      \psline(-1,-0.25)(1,-0.25)
      \psline(-1,-0.75)(1,-0.75)
      \psset{linestyle=dotted,dotsep=0.5pt}
      \psline(-4.9,-0.8)(-4,-0.8)
      \psline(-4.9,-0.6)(-4,-0.6)
      \psline(-4.9,-0.4)(-4,-0.4)
      \psline(-4.9,-0.2)(-4,-0.2)
      \psline(-4.9,0.0)(-4,0.0)
      \psline(-4.9,0.2)(-4,0.2)
      \psline(-4.9,0.4)(-4,0.4)
      \psline(-4.9,0.6)(-4,0.6)
      \psline(-4.9,0.8)(-4,0.8)
      \psline(4.9,-0.8)(4,-0.8)
      \psline(4.9,-0.6)(4,-0.6)
      \psline(4.9,-0.4)(4,-0.4)
      \psline(4.9,-0.2)(4,-0.2)
      \psline(4.9,0.0)(4,0.0)
      \psline(4.9,0.2)(4,0.2)
      \psline(4.9,0.4)(4,0.4)
      \psline(4.9,0.6)(4,0.6)
      \psline(4.9,0.8)(4,0.8)
      \psline(-1,0.9)(1,0.9)
      \psline(-1,0.8)(1,0.8)
      \psline(-1,0.7)(1,0.7)
      \psline(-1,0.6)(1,0.6)
      \psline(-1,0.5)(1,0.5)
      \psline(-1,0.4)(1,0.4)
      \psline(-1,0.3)(1,0.3)
      \psline(-1,0.2)(1,0.2)
      \psline(-1,0.1)(1,0.1)
      \psline(-1,0.0)(1,0.0)
      \psline(-1,-0.1)(1,-0.1)
      \psline(-1,-0.2)(1,-0.2)
      \psline(-1,-0.8)(1,-0.8)
      \psline(-1,-0.9)(1,-0.9)
      \psbezier(-4,0.8)(-3,0.8)(-2,0.9)(-1,0.9)
      \psbezier(-4,0.6)(-3,0.6)(-2,0.8)(-1,0.8)
      \psbezier(-4,0.4)(-3,0.4)(-2,0.7)(-1,0.7)
      \psbezier(-4,0.2)(-3,0.2)(-2,0.6)(-1,0.6)
      \psbezier(-4,0.0)(-3,0.0)(-2,0.5)(-1,0.5)
      \psbezier(-4,-0.2)(-3,-0.2)(-2,0.4)(-1,0.4)
      \psbezier(-4,-0.4)(-3,-0.4)(-2,0.3)(-1,0.3)
      \psbezier(-4,-0.6)(-3,-0.6)(-2,0.2)(-1,0.2)
      \psbezier(-4,-0.8)(-3,-0.8)(-2,0.1)(-1,0.1)
      \psbezier(-2.8,-1)(-2.8,-0.6)(-1.5,0.0)(-1,0.0)
      \psbezier(-1,-0.1)(-1.5,-0.1)(-2.9,-0.9)(-1,-0.9)
      \psbezier(-1,-0.2)(-1.4,-0.2)(-1.8,-0.8)(-1,-0.8)
      \psbezier(4,0.8)(3,0.8)(2,0.9)(1,0.9)
      \psbezier(4,0.6)(3,0.6)(2,0.8)(1,0.8)
      \psbezier(4,0.4)(3,0.4)(2,0.7)(1,0.7)
      \psbezier(4,0.2)(3,0.2)(2,0.6)(1,0.6)
      \psbezier(4,0.0)(3,0.0)(2,0.5)(1,0.5)
      \psbezier(4,-0.2)(3,-0.2)(2,0.4)(1,0.4)
      \psbezier(4,-0.4)(3,-0.4)(2,0.3)(1,0.3)
      \psbezier(4,-0.6)(3,-0.6)(2,0.2)(1,0.2)
      \psbezier(4,-0.8)(3,-0.8)(2,0.1)(1,0.1)
      \psbezier(2.8,-1)(2.8,-0.6)(1.5,0.0)(1,0.0)
      \psbezier(1,-0.1)(1.5,-0.1)(2.9,-0.9)(1,-0.9)
      \psbezier(1,-0.2)(1.4,-0.2)(1.8,-0.8)(1,-0.8)
      \psbezier(-1,-0.4)(-1.2,-0.4)(-1.3,-0.6)(-1,-0.6)
      \psbezier(1,-0.4)(1.2,-0.4)(1.3,-0.6)(1,-0.6)
      \pscircle[linestyle=solid,linewidth=0.2pt](0,-0.5){0.125}
      \pscircle(0,-0.5){0.1}
      \psbezier(0.05,-0.3)(-0.5,-0.3)(-0.5,-0.4)(-1,-0.4)
      \psbezier(0.05,-0.7)(-0.5,-0.7)(-0.5,-0.6)(-1,-0.6)
      \psbezier(-0.05,-0.3)(0.5,-0.3)(0.5,-0.4)(1,-0.4)
      \psbezier(-0.05,-0.7)(0.5,-0.7)(0.5,-0.6)(1,-0.6)
      \psset{linestyle=solid,linewidth=0.5pt}
      \psline(0,1)(0,-0.375)
      \psline(0,-0.625)(0,-1)
      \psbezier(0,-0.375)(0,-0.4)(-0.1,-0.4)(-0.1,-0.5)
      \psbezier(0,-0.625)(0,-0.6)(-0.1,-0.6)(-0.1,-0.5)
      \psline[linewidth=0.3pt](0.3,1.15)(0.05,0.55)
      \rput[b](0.3,1.2){\psscalebox{0.7}{$\gamma$}}
      \rput[b](-1,1.05){\psscalebox{0.4}{$-2\bigl(1-\cos\theta\bigr)$}}
      \rput[b](1,1.05){\psscalebox{0.4}{$2\bigl(1-\cos\theta\bigr)$}}
      \rput[b](4,1.05){\psscalebox{0.7}{$L_\theta\gg 1$}}
      \rput[b](-4,1.05){\psscalebox{0.7}{$-L_\theta$}}
      \rput[r](-5,-1){\psscalebox{0.7}{$-1$}}
      \rput[r](-5,1){\psscalebox{0.7}{$1$}}
    \end{pspicture}
    \caption{Sketch of the level lines of a stream function of an asymmetric domain wall of Bloch type.}
    \label{fig:abwc}
  \end{figure}
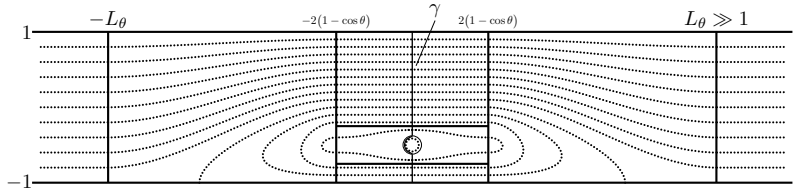

  \begin{figure}[t]
    \centering
    \begin{pspicture}(-4,-1)(4,1)
      \psframe(-4,-1)(4,1)
      \psline(0,1)(0,0.5)
      \psline(0,-1)(0,-0.5)
      \psbezier(0,0.5)(0,0.3)(-0.4,0.3)(-0.4,0)
      \psbezier(0,-0.5)(0,-0.3)(-0.4,-0.3)(-0.4,0)
      \pscircle[linewidth=0.5pt](0,0){0.5}
      \pscircle[linewidth=0.5pt](0,0){0.3}
      \rput[b](0.5,1.2){\psscalebox{0.7}{$\gamma$}}
      \psline[linewidth=0.3pt](0.5,1.15)(0.05,0.55)
      \psset{linestyle=dotted,dotsep=0.5pt}
      \pscircle(0,0){0.05}
      \pscircle(0,0){0.2}
      \pscircle(0,0){0.4}
      \pscircle(0,0){0.6}
      \psbezier(0,0.7)(-2.5,0.7)(-2.5,-0.7)(0,-0.7)
      \psbezier(0,0.7)(2.5,0.7)(2.5,-0.7)(0,-0.7)
      \psbezier(0,0.8)(-3.5,0.8)(-3,-0.7)(-4,-0.7)
      \psbezier(0,0.8)(3.5,0.8)(3,-0.7)(4,-0.7)
      \psbezier(0,-0.8)(-3.5,-0.8)(-3,0.7)(-4,0.7)
      \psbezier(0,-0.8)(3.5,-0.8)(3,0.7)(4,0.7)
      \psbezier(-4,0.85)(-3,0.85)(-3.5,0.9)(0.01,0.9)
      \psbezier(-4,-0.85)(-3,-0.85)(-3.5,-0.9)(0.01,-0.9)
      \psbezier(4,0.85)(3,0.85)(3.5,0.9)(-0.01,0.9)
      \psbezier(4,-0.85)(3,-0.85)(3.5,-0.9)(-0.01,-0.9)
      \rput[b](-4,1.1){\psscalebox{0.7}{$-2\bigl(1-\cos\theta\bigr)$}}
      \rput[b](4,1.1){\psscalebox{0.7}{$2\bigl(1-\cos\theta\bigr)$}}
      \rput[r](-4.1,0.95){\psscalebox{0.7}{$\tfrac{1-\cos\theta}{2}$}}
      \rput[r](-4.1,-0.95){\psscalebox{0.7}{$-\tfrac{1-\cos\theta}{2}$}}
    \end{pspicture}
    \caption{Enlargement of the area around the vortex, cf. Figure~\ref{fig:abwc}.}
    \label{fig:abwv}
  \end{figure}

  \begin{figure}[h]
    \centering
    \begin{pspicture}(-3,-3)(3,3)
      \psset{unit=3cm}
      \pscircle(0,0){0.25}
      \pscircle(0,0){0.125}
      \psset{linestyle=dotted,dotsep=0.3pt}
      \pscircle(0,0){0.2}
      \pscircle(0,0){0.15}
      \pscircle(0,0){0.09}
      \pscircle(0,0){0.04}
      \pscircle(0,0){0.01}
      \pscircle(0,0){0.001}
      \psellipse(0,0)(0.4,0.30)
      \psellipse(0,0)(0.8,0.35)
      \psellipse(0,0)(1.8,0.40)
      \psellipse(0,0)(3.2,0.45)
      \psline(-1,-0.55)(1,-0.55)
      \psline(-1,-0.60)(1,-0.60)
      \psline(-1,-0.65)(1,-0.65)
      \psline(-1,-0.70)(1,-0.70)
      \psline(-1,-0.75)(1,-0.75)
      \psline(-1,-0.80)(1,-0.80)
      \psline(-1,-0.85)(1,-0.85)
      \psline(-1,-0.90)(1,-0.90)
      \psline(-1,-0.95)(1,-0.95)
      \psline(-1,0.55)(1,0.55)
      \psline(-1,0.60)(1,0.60)
      \psline(-1,0.65)(1,0.65)
      \psline(-1,0.70)(1,0.70)
      \psline(-1,0.75)(1,0.75)
      \psline(-1,0.80)(1,0.80)
      \psline(-1,0.85)(1,0.85)
      \psline(-1,0.90)(1,0.90)
      \psline(-1,0.95)(1,0.95)
      \psset{linestyle=solid}
      \psframe*[linecolor=white](-4,-0.5)(-1,0.5)
      \psframe*[linecolor=white](1,-0.5)(4,0.5)
      \psline(-1,-1)(-1,1)
      \psline(1,-1)(1,1)
      \psline(-1,0.5)(1,0.5)
      \psline(-1,-0.5)(1,-0.5)
      \rput[t](-1,-1.05){\psscalebox{0.6}{$-1$}}
      \rput[t](1,-1.05){\psscalebox{0.6}{$1$}}
      \rput[r](-1.05,-0.5){\psscalebox{0.6}{$-\frac{1}{2}$}}
      \rput[r](-1.05,-0.25){\psscalebox{0.6}{$-\frac{1}{4}$}}
      \rput[r](-1.05,0.25){\psscalebox{0.6}{$\frac{1}{4}$}}
      \rput[r](-1.05,0.5){\psscalebox{0.6}{$\frac{1}{2}$}}
      \psline(0,1)(0,0.25)
      \psbezier(0,0.25)(0,0.15)(-0.2,0.15)(-0.2,0)
      \psbezier(0,-0.25)(0,-0.15)(-0.2,-0.15)(-0.2,0)
      \psline(0,-1)(0,-0.25)
      \psset{linewidth=0.005}
      \psline(0.02,0.725)(1.1,0.8)
      \rput[l](1.12,0.8){\psscalebox{0.7}{$\hat{\gamma}$}}
    \end{pspicture}
    \caption{Sketch of vortex function $s$ with ellipsoid level sets in the inner part.}
    \label{fig:hatpsi}
  \end{figure}
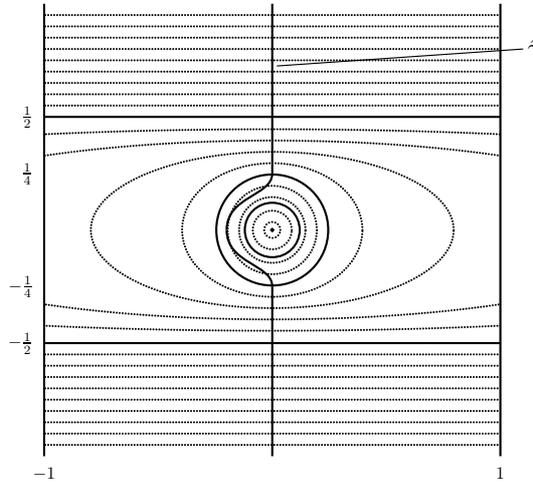

  \textbf{Step 1:} \textit{Construction in the inner part around the vortex.}
  As a first step in the construction of $\psi$ we implicitly define a function $s \in C^\infty(Q)$, $$Q=([-1,1]\times\R)\setminus \{(0,0)\},$$ by specifying its level sets (cf. Figure~\ref{fig:hatpsi}). Later, we will define $\psi$ by rescaling, shifting and smoothing the function $1-s$. Consider
  \begin{align*}
   f\colon Q\times (0, \infty)\to \R, \quad  f(\hat{x},s) \defas
    \begin{cases}
      \bigl(\tfrac{\hat{x}_3}{s}\bigr)^2 - 1, &\text{if } s\geq \tfrac{1}{2}, \hat{x}\in Q,\\
      \bigl(\tfrac{s}{t(s)}\bigr)^2 \bigl(\tfrac{\hat{x}_1}{s}\bigr)^2 + \bigl(\tfrac{\hat{x}_3}{s}\bigl)^2 - 1, &\text{if } 0<s< \tfrac{1}{2}, \hat{x}\in Q,\\
    \end{cases}
  \end{align*}
  where $t \colon (0,\tfrac{1}{2}) \to \R$ is a smooth function that satisfies the structural condition
  \begin{align}\label{eq:condt}
    t(s) = s \text{ if }s \in [0, \tfrac{1}{4}], \quad
    \tfrac{d^2}{ds^2} t\geq 0
    \text{ on }(0,\tfrac{1}{2}),
  \end{align}
  and such that $\tfrac{1}{t(s)}$ vanishes to infinite order at $s=\tfrac{1}{2}$, e.g.
  \begin{align*}
    t(s) = e^\frac{1}{\frac{1}{2}-s} \text{ if } s\in [\tfrac{3}{8}, \tfrac 1 2].
  \end{align*}
  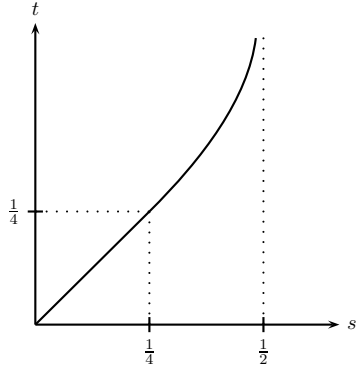
\begin{figure}
    \centering
    \begin{pspicture}(0,-0.5)(4,4)
      \psline{->}(0,0)(4,0)\rput[l](4.1,0){\psscalebox{0.7}{$s$}}
      \psline{->}(0,0)(0,4)\rput[b](0,4.1){\psscalebox{0.7}{$t$}}
      \psline(3,-0.1)(3,0.1)\rput[t](3.0,-0.2){\psscalebox{0.7}{$\tfrac{1}{2}$}}
      \psline(1.5,-0.1)(1.5,0.1)\rput[t](1.5,-0.2){\psscalebox{0.7}{$\tfrac{1}{4}$}}
      \psline(-0.1,1.5)(0.1,1.5)\rput[r](-0.2,1.5){\psscalebox{0.7}{$\tfrac{1}{4}$}}
      \psline[linestyle=dotted](1.5,0)(1.5,1.5)
      \psline[linestyle=dotted](0,1.5)(1.5,1.5)
      \psline[linestyle=dotted](3,0)(3,3.8)
      \psline(0,0)(1.5,1.5)
      \psbezier(1.5,1.5)(2.2,2.2)(2.8,3)(2.9,3.8)
    \end{pspicture}
    \caption{The semi-major axis $t\colon(0,\frac{1}{2})\to\R$ of the ellipses in Figure~\ref{fig:hatpsi}.}
  \end{figure}

  We note that the latter implies that $f$ is smooth across $s=\tfrac{1}{2}$ and thus in the whole domain $Q\times(0,\infty)$.

  {\it Claim: For every $\hat{x}\in Q$ there exists a unique solution $s=s(\hat{x})$ of $f(\hat{x},s)=0$.} 

  We first argue that the solution is unique: Indeed, because of \eqref{eq:condt} we have in particular $\dds t(s)\geq 1$ so that $\partial_s f(\hat{x},s)<0$ for all $(\hat{x},s)\in Q\times(0,\infty)$, provided we are not in the case of $s \geq \tfrac{1}{2}$ and $\hat{x}_3=0$. This case however is not relevant for uniqueness since then $f(\hat{x},s)\equiv -1$.

  It follows from the explicit form of $f$ that
  \begin{align*}
    s(\hat{x}) = \begin{cases}
      \lvert \hat{x} \rvert, &\text{for }\lvert \hat{x} \rvert \leq \tfrac{1}{4},\\
      \lvert \hat{x}_3 \rvert, &\text{for }\lvert \hat{x}_3 \rvert \geq \tfrac{1}{2},\\
    \end{cases}
  \end{align*}
  is a solution.

  Hence, it remains to show existence of a solution for $\lvert \hat{x}_3 \rvert < \tfrac{1}{2}$ but $\lvert \hat{x} \rvert > \tfrac{1}{4}$: Indeed, $\lvert \hat{x}_3 \rvert < \tfrac{1}{2}$ implies $f(\hat{x},\tfrac{1}{2})<0$ and $\lvert \hat{x} \rvert > \tfrac{1}{4}$ yields $f(\hat{x},\tfrac{1}{4})>0$. Thus, the existence of a solution $s=s(\hat{x})\in(\tfrac{1}{4},\tfrac{1}{2})$ of $f(\hat{x},s)=0$ follows from the intermediate value theorem.

  The implicit function theorem yields smoothness of $s \colon Q \to (0,\infty)$, with
  \begin{align}\label{eq:nabls}
    \hat{\nabla}s(\hat{x})=-\tfrac{\nabla_{\hat{x}} f(\hat{x},s(\hat{x}))}{\partial_s f(\hat{x},s(\hat{x}))} = \frac{(\frac{\hat{x}_1}{t^2(s)},\frac{\hat{x}_3}{s^2})}{\frac{\hat{x}_1^2}{t^2(s)}\frac{\frac{dt}{ds}}{t(s)} + \frac{\hat{x}_3^2}{s^2}\frac{1}{s}} \quad \text{if } |\hat{x}_3|\leq \tfrac 1 2 ,
  \end{align}
  and $\hat\nabla s(\hat{x})=\pm {\bf e}_3$ if $\pm \hat{x}_3\geq \tfrac 1 2$.
  Note that
  \begin{align*}
    \lvert \hat{\nabla} s(\hat{x})\rvert^2 = \frac{(\frac{\hat{x}_1}{t(s)})^2\frac{1}{t^2(s)} + (\frac{\hat{x}_3}{s})^2 \frac{1}{s^2}}{(\frac{\hat{x}_1^2}{t^2(s)}\frac{\frac{dt}{ds}(s)}{t(s)} + \frac{\hat{x}_3^2}{s^2}\frac{1}{s})^2} \leq \frac{\frac{1}{s^2}\bigl((\frac{\hat{x}_1}{t(s)})^2 + (\frac{\hat{x}_3}{s})^2 \bigr)}{\frac{1}{s^2} (\frac{\hat{x}_1^2}{t^2(s)} + \frac{\hat{x}_3^2}{s^2})^2} = 1 \quad \text{if }|\hat{x}_3|\leq \tfrac 1 2,
  \end{align*}
  since \eqref{eq:condt} yields $\tfrac{1}{t(s)} \leq \tfrac{1}{s}$ and $\tfrac{\frac{dt}{ds}(s)}{t(s)}\geq \tfrac{1}{s}$ whenever $s\in (0, \tfrac 1 2)$.

  Let us finally remark that the curve 
  \begin{align}\label{eq:defgamma}
    \hat{\gamma} \subset \{0\} \times \{\lvert x_3\rvert \geq \tfrac{1}{4} \} \cup \bigl( B(0,\tfrac{1}{4})\setminus B(0,\tfrac{1}{16}) \bigr)
  \end{align}
  which is indicated in Figure~\ref{fig:hatpsi}, has the property
  \begin{align*}
    \lvert \hat{\nabla} s \rvert = 1 \quad \text{on }\hat{\gamma}.
  \end{align*}
  
  \textbf{Step 2:} \textit{Regularization of the vortex at $\hat{x}=0$.} In this step, we define a function $\hat{\psi}_1$ on $[-1,1]\times \R$ that -- up to rescaling and recentering -- already coincides with the final $\psi$ close to $\{x_1=0\}$. The subsequent steps 3-6 modify $\hat{\psi}_1$ for large $\hat{x} \in \R[2]$ to achieve the boundary conditions for $\lvert x_1\rvert \to \infty$ and to make Lemma~\ref{lem:m2h1} applicable.
  
  In principle, we would like to set $\hat{\psi}_1=1-s$, but since $s$ is not smooth in $\hat{x}=0$ this would generate a vortex-type point-singularity at $\hat{x}=0$ for $\hat{\nabla}^\perp \hat{\psi}_1$. Instead, let $\rho \colon [0,\infty)\to \R$ be a smooth function that satisfies
  \begin{align*}
    \rho(s) = 1 - s \text{ if }s\geq \tfrac{1}{16},\quad -1\leq \tfrac{d\rho}{ds}(s)\leq 0 \text{ if }s\geq 0, \quad \tfrac{d^n\rho}{ds^n}(0)=0 \text{ for all integers } n>0.
  \end{align*}

  Then the function $$\hat{\psi}_1(\hat{x}) \defas \rho\bigl(s(\hat{x})\bigr), \quad \hat{x}\in Q,$$ is smooth, satisfies $\lvert \hat{\nabla} \hat{\psi}_1 \rvert = \lvert \dds\rho \rvert \, \lvert \hat{\nabla} s \rvert \leq 1$ and can be extended to a smooth function $\hat{\psi}_1$ on $[-1,1]\times\R$ by setting $\hat{\psi}_1(\hat{x}=0)\defas \rho(s=0)$. The regularity of $\hat{\psi}_1$ around $\hat{x}=0$ is due to $s(\hat{x})=|\hat{x}|$ for $\lvert\hat{x}\rvert \leq \tfrac{1}{4}$ and $\frac{d^n\rho}{ds^n}(0)=0$ for all $n>0$.
  
  Note that by definition of $\rho$ and $s$ we still have
  \begin{align*}
    \lvert \hat{\nabla} \hat{\psi}_1 \rvert = 1 \quad \text{on }\hat{\gamma},
  \end{align*}
  for $\hat{\gamma}$ as in \eqref{eq:defgamma}.

  \textbf{Step 3:} \textit{Extending $\hat{\psi}_1$ to $\R[2]$.} Here, we use $\hat{\psi}_1$ (defined on $[-1,1]\times\R$) to define a smooth function $\hat{\psi}_2$ on $\R[2]$ with the properties
  \begin{align}\label{eq:proppsi2}
    \partial_{\hat{x}_1}\hat{\psi}_2 = 0 \text{ if }\lvert \hat{x}_1 \rvert \geq 2,\quad \lvert \hat{\nabla} \hat{\psi}_2 \rvert \leq 1 \text{ on }\R[2],\quad \lvert \hat{\nabla} \hat{\psi}_2 \rvert = 1 \text{ on }\hat{\gamma}.
  \end{align}
  
  Let $\varphi\colon \R \to[-1,1]$ be a smooth odd non-linear change of variables (cf. Figure~\ref{fig:cov}) with
  \begin{gather*}
    \varphi(s)=s \text{ on } (0,\tfrac{1}{4}),\quad\varphi(s)=1 \text{ if } s\geq 2, \quad 0<\dds\varphi(s)\leq 1 \text{ on } (0,2).\quad 
  \end{gather*}

  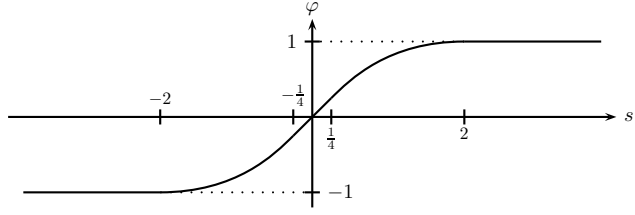
\begin{figure}
    \centering
    \begin{pspicture}(-4,-1.1)(4,1.1)
      \psline{->}(-4,0)(4,0)
      \psline{->}(0,-1.2)(0,1.3)
      \psline(-0.25,-0.1)(-0.25,0.1)
      \psline(0.25,-0.1)(0.25,0.1)
      \psline(-2,-0.1)(-2,0.1)
      \psline(2,-0.1)(2,0.1)
      \psline(-0.25,-0.25)(0.25,0.25)
      \psbezier(-2,-1)(-1,-1)(-0.5,-0.5)(-0.25,-0.25)
      \psbezier(2,1)(1,1)(0.5,0.5)(0.25,0.25)
      \psline(-3.8,-1)(-2,-1)
      \psline(3.8,1)(2,1)
      \rput[b](0,1.35){\psscalebox{0.7}{$\varphi$}}
      \rput[l](4.1,0){\psscalebox{0.7}{$s$}}
      \rput[b](-0.25,0.15){\psscalebox{0.6}{$-\tfrac{1}{4}$}}
      \rput[t](0.25,-0.15){\psscalebox{0.6}{$\tfrac{1}{4}$}}
      \rput[t](2,-0.15){\psscalebox{0.6}{$2$}}
      \rput[b](-2,0.15){\psscalebox{0.6}{$-2$}}
      \psline(-0.1,1)(0.1,1)\rput[r](-0.2,1){\psscalebox{0.7}{$1$}}
      \psline(-0.1,-1)(0.1,-1)\rput[l](0.2,-1){\psscalebox{0.7}{$-1$}}
      \psline[linestyle=dotted](0.1,1)(2,1)
      \psline[linestyle=dotted](-0.1,-1)(-2,-1)
    \end{pspicture}
    \caption{The non-linear change of variables $\varphi$.}
    \label{fig:cov}
  \end{figure}

  Then we let
  \begin{align*}
    \hat{\psi}_2(\hat{x}) \defas \hat{\psi}_1(\varphi(\hat{x}_1),\hat{x}_3) \quad\text{for }\hat{x}_1\in\R[2],
  \end{align*}
  such that the properties \eqref{eq:proppsi2} are easily verified.

  \textbf{Step 4:} \textit{Matching the boundary conditions on $\partial\Omega$.} In this step, we rescale and recenter $\hat{\psi}_2$ according to Figure~\ref{fig:abwc} to achieve the boundary conditions $0$ and $-2\cos\theta$ on the lower and upper components of $\partial\Omega$, i.e. \eqref{eq:psibcx3}.

  More precisely, we want to obtain \eqref{eq:psibdry} below. Since $\hat{\psi}_2(\hat{x})=1-\lvert \hat{x}_3 \rvert$ for $\lvert\hat{x}_3\rvert\geq \tfrac{1}{2}$, we place the center of the ``regularized vortex'' $\hat{x}=0$ of $\hat{\psi}_2$ at $x_\theta=\bigl(0,-\cos\theta\bigr)$, and thereby define the smooth function:
  \begin{align*}
    \psi_2(x) \defas \bigl(1-\cos\theta\bigr) \, \hat{\psi}_2(\hat{x}) \quad \text{for } x\in \R[2],  
  \end{align*}
  where $\hat{x}$ is related to $x$ via
  \begin{align}
    x = x_\theta + (1-\cos\theta) \hat{x}.\label{eq:cov}
  \end{align}
Then
\begin{align}\label{eq:psibdry}
  \psi_2(x)=1-\cos\theta -\lvert x_3 + \cos\theta \rvert\quad\text{on}\quad\bigl\{ \lvert x_3 + \cos\theta \rvert \geq \tfrac{1-\cos\theta}{2}\bigr\} \supset \partial\Omega,
\end{align}
such that the boundary conditions hold. Moreover, we have $\lvert \nabla \psi_2 \rvert \leq 1$ in $\R[2]$ as well as $\lvert \nabla \psi_2 \rvert = 1$ on the curve $\gamma$ that $\hat{\gamma}$ induces via the change of variables \eqref{eq:cov}, cf. Figure~\ref{fig:abwc}.

Note that $\psi_2$ only depends on $x_3$ for $\lvert x_1 \rvert \geq 2(1-\cos\theta)$. However, \eqref{eq:psibcx1} does not yet hold.

\textbf{Step 5:} \textit{Controlling the behavior for $\lvert x_3 \rvert \gg 1$.} 
To allow for an application of Lemma~\ref{lem:m2h1} we want to obtain \eqref{eq:psibcx1}, in particular bounded second derivatives of $f=1-\lvert \nabla \psi \rvert^2$. To this end, we will first interpolate $\psi_2$ in $x_3$ with the boundary data 
  \begin{align*}
    \psi_\text{out}\defas -(x_3+1)\cos\theta
  \end{align*}
  for $\lvert x_3 \rvert \gg 1$. In Step 6, we will then interpolate with $\psi_\text{out}$ in $x_1$.

  We proceed in two steps: First, we employ a regularized $\max(\tilde{t},t)$-function to modify $\psi_2$ outside of $\Omega$ to make sure that the slope of $\psi_2$ agrees with that of $\psi_\text{out}$ for large $\lvert x_3 \rvert$. Then, since $\lvert \partial_{x_3} \psi_\text{out} \rvert = \cos\theta < 1$, we can use interpolation with a slowly varying cut-off function to define a new function $\psi_\text{in}$ that coincides with $\psi_\text{out}$ for $\lvert x_3 \rvert \gg 1$ and still satisfies $\lvert \nabla \psi_\text{in} \rvert \leq 1$, cf. Figure~\ref{fig:interppsi}.

  \begin{figure}
    \centering
    \begin{pspicture}(-7,-4)(7,1)
      \psline{->}(-6.5,0)(6.5,0)\rput[l](6.6,0){\psscalebox{0.8}{$x_3$}}
      \psline{->}(0,-4)(0,1)\rput[b](0,1.1){\psscalebox{0.8}{$\psi$}}
      \psline(-1,-0.1)(-1,0.1)\rput[b](-1.05,0.15){\psscalebox{0.7}{$-1$}}
      \psline(1,-0.1)(1,0.1)\rput[b](1,0.15){\psscalebox{0.7}{$1$}}
      \psline(-0.1,-0.5)(0.1,-0.5)\rput[r](-0.15,-0.5){\psscalebox{0.7}{$-2\cos\theta$}}
      \rput(-2,1){\psscalebox{0.8}{$\psi_\text{out}$}}\psline[linewidth=0.01](-2.35,0.8)(-2.6,0.5)
      \rput(-1.5,-2){\psscalebox{0.8}{$\psi_2$}}\psline[linewidth=0.01](-1.65,-1.9)(-2.35,-1.45)
      \rput(5.2,-2.2){\psscalebox{0.8}{$\psi_\text{in}$}}\psline[linewidth=0.01](5,-2.1)(4.55,-1.95)
      \rput(5.4,-3.1){\psscalebox{0.8}{asymptote of $\psi_3$}}\psline[linewidth=0.01](5.25,-2.95)(4.55,-2.45)
      \psline[linestyle=dashed](-6,1.25)(6,-1.75)
      \psline[linestyle=dotted](-6,0.25)(6,-2.75)
      \psbezier(-1,0)(-0.25,0.75)(-0.25,0.75)(1,-0.5)
      \psline[linestyle=dashed](-4,-3)(-1,0)
      \psline[linestyle=dashed](1,-0.5)(4,-3.5)
      \psbezier(-1,0)(-1.8,-0.8)(-1.8,-0.8)(-3,-0.5)
      \psbezier(-3,-0.5)(-4,-0.25)(-5,1)(-6,1.25)
      \psbezier(1,-0.5)(2.33,-1.833)(2.33,-1.833)(3,-2)
      \psbezier(3,-2)(4,-2.25)(5,-1.5)(6,-1.75)
    \end{pspicture}
    \caption{Sketch of the functions $\psi_2(x_1,\cdot)$ (for $\lvert x_1 \rvert$ large, fixed) and $\psi_\text{out}$, as well as their interpolant $\psi_\text{in}$.}
    \label{fig:interppsi}
  \end{figure}
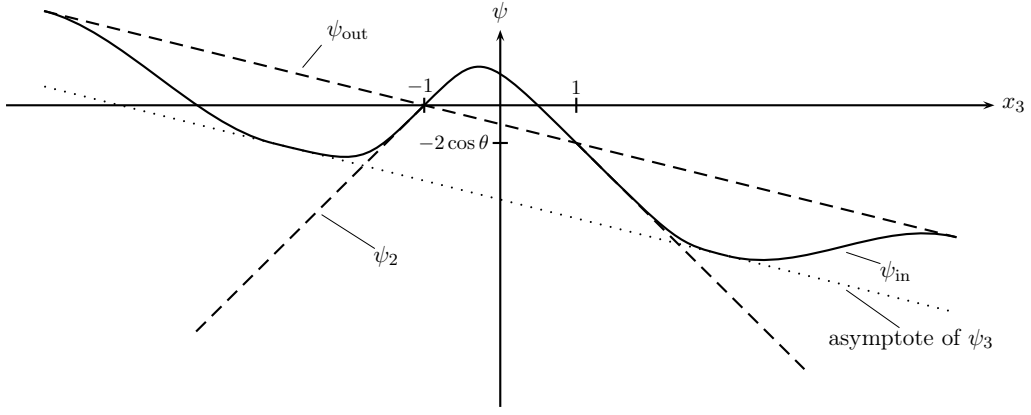

  Let $\eta \colon \R \to [0,1]$ be a smooth, increasing cut-off function with $\eta\equiv 0$ on $\R_-$, $\eta\equiv 1$ on $[1,\infty)$, and $\lVert \dds \eta \rVert_\infty < \infty$. To regularize 
  \begin{align*}
    \max(\tilde{t},t)=\tilde{t}+\max(0,t-\tilde{t})=\tilde{t}+\int_0^{t-\tilde{t}} \mathbf{1}_{[0,\infty)}(s) \,ds
  \end{align*}
  we replace $\mathbf{1}_{[0,\infty)}$ by $\eta$, i.e. we define a smooth $h\colon \R\times\R \to \R$ via
  \begin{gather}\label{eq:defh}
    h(\tilde{t},t) \defas \tilde{t} + \int_0^{t-\tilde{t}} \! \eta(s) \, ds.
  \end{gather}
  Observe that 
  \begin{align}\label{eq:proph}
    h(\tilde{t},t)=\tilde{t} \text{ if }\tilde{t}\geq t,\quad \text{and}\quad h(\tilde{t},t)=\int_0^1 \eta(s) \, ds - 1 + t,\text{ if }t\geq \tilde{t}+1.
  \end{align}
    Moreover, $\partial_{\tilde{t}} h(\tilde{t},t) = 1-\eta(t-\tilde{t})$ and $\partial_t h(\tilde{t},t)=\eta(t-\tilde{t})$.

  Hence, the function $\psi_3 \colon \R[2]\to\R$ given by
  \begin{gather*}
    \psi_3(x) \defas \begin{cases}
      \psi_2(x), &\text{for } x\in\Omega,\\
      h\bigl(\psi_2(x),\psi_\text{out}(x)-1\bigr), &\text{otherwise},
    \end{cases}
  \end{gather*}
  is smooth and satisfies
  \begin{gather}
    \psi_3 \stackrel{\eqref{eq:proph}}{\equiv} \psi_2 \text{ on }\Omega, \quad \psi_3-\psi_\text{out} \stackrel{\eqref{eq:proph}}{=} \int_0^1\eta ds - 2\text{ if }\lvert x_3 \rvert \stackrel{\eqref{eq:psibdry}}{\geq} 1+\tfrac{2}{1-\cos\theta} \asdef M_\theta,\label{eq:proppsi3}\\
    \lvert \nabla \psi_3 \rvert \leq 1 \text{ on }\Omega,\quad\partial_{x_3}\psi_3 = 0 \text{ for }\lvert x_3 \rvert \geq 2,\notag\\
    \lvert \nabla \psi_3 \rvert \leq \bigl( 1-\eta(\psi_\text{out}-1 -\psi_2) \bigr) \lvert \nabla \psi_2 \rvert + \eta(\psi_\text{out}-1-\psi_2) \lvert \nabla \psi_\text{out} \rvert \leq 1 \text{ on }\R[2]\setminus\Omega.\notag
  \end{gather}

  It remains to interpolate $\psi_3$ and $\psi_\text{out}$: For $L\geq M_\theta+1$, consider the slowly varying cut-off function $\eta_L \colon \R_+ \to [0,1]$ given by
  \begin{gather*}
    \eta_L(t)\defas \eta^2(\tfrac{t-M_\theta}{L-M_\theta}).
  \end{gather*}
  Then
  \begin{gather}
    \eta_L(t)=0 \text{ if } t\leq 2, \quad \eta_L(t)=1 \text{ if } t\geq L, \quad (\tfrac{d}{dt} \eta_L )^2 \leq \tfrac{C(\theta)}{L^2} \eta_L \leq \tfrac{C(\theta)}{L^2} \text{ on }\R_+,\label{eq:propcutoff}
  \end{gather}

  and we define the smooth function $\psi_\text{in} \colon\R[2]\to\R$ by
  \begin{align*}
    \psi_\text{in}(x) \defas \eta_L(\lvert x_3 \rvert) \psi_\text{out}(x) + \bigl(1-\eta_L(\lvert x_3 \rvert)\bigr) \psi_3(x),
  \end{align*}
  There exists $L(\theta)$ such that for any $L\geq L(\theta)$ we have $\lvert \nabla \psi_\text{in} \rvert \leq 1$ on $\R[2]$: Note that $\lvert \nabla \psi_\text{in} \rvert = \lvert \nabla \psi_3 \rvert \leq 1$ on $\{\lvert x_3 \rvert \leq M_\theta\}$, while on $\{\lvert x_3 \rvert \geq M_\theta \}$ we have $\psi_\text{in}\stackrel{\eqref{eq:proppsi3}}{=} \psi_\text{out} - (2-\int_0^1\eta ds)(1-\eta_L(\lvert x_3 \rvert ))$, such that due to $\nabla \psi_\text{out} = (0, -\cos\theta)$
  \begin{align*}
    \lvert \nabla \psi_\text{in} \rvert \leq \cos\theta + C\tfrac{d}{dt} \eta_L \stackrel{\eqref{eq:propcutoff}}{\leq} \cos\theta + \tfrac{C}{L} \leq 1 \text{ for $L$ sufficiently large}.
  \end{align*}
  Moreover, $\partial_{x_1} \psi_\text{in} = 0$ for $\lvert x_1 \rvert \geq 2(1-\cos\theta)$.

  \textbf{Step 6:} \textit{Interpolation with the boundary conditions at $x_1=\pm\infty$.}
  In order to obtain \eqref{eq:psibcx1} it now remains to interpolate $\psi_\text{in}$ with the boundary data $\psi_\text{out}$ for $\lvert x_1 \rvert \gg 1$.

  For this, we again consider the cut-off function $\eta_L \colon \R_+ \to [0,1]$ with properties \eqref{eq:propcutoff} and define the desired smooth $\psi\colon\R[2]\to\R$ by
  \begin{align*}
    \psi(x) \defas \eta_L(\lvert x_1 \rvert) \psi_\text{out}(x) + \bigl(1-\eta_L(\lvert x_1 \rvert)\bigr) \psi_\text{in}(x).
  \end{align*}

  Clearly, $\psi$ satisfies the boundary conditions on $\partial\Omega$ as well as $\psi = \psi_\text{out}$ for $\lvert x \rvert \gg 1$. In the core region $\{\lvert x_1 \rvert \leq M_\theta\}\cap \Omega$ we have $\psi=\psi_\text{in}=\psi_3=\psi_2$ and therefore $\lvert \nabla \psi \rvert = 1$ on the curve $\gamma$ defined in Step 4. Moreover, $\lvert \nabla \psi \rvert \leq 1$ on $\{\lvert x_1 \rvert \leq M_\theta\} \cup \{\lvert x_1 \rvert \geq L\}$.
  
  For sufficiently large $L \geq L(\theta)$ we can also assert $\lvert \nabla \psi \rvert \leq 1$ on $\{M_\theta \leq \lvert x_1 \rvert\leq L \}$: In fact, we have 
  $$\nabla \psi = \left( \sgn(x_1) (\psi_\text{out}-\psi_\text{in})\tfrac{d}{dt} \eta_L , -\eta_L \cos\theta + (1-\eta_L) \partial_{x_3} \psi_\text{in} \right),$$
  where we used that $\partial_{x_1}\psi_\text{in}(x)=0$ on $\left\{\lvert x_1 \rvert \geq M_\theta \right\}$.
  Hence, by convexity of $z\mapsto z^2$, $\eta_L\in[0,1]$, and $\lvert\partial_{x_3}\psi_\text{in}\rvert\leq 1$:
  \begin{align*}
    \lvert \nabla \psi \rvert^2 &\leq (\tfrac{d}{dt} \eta_L)^2 \underbrace{\sup \lvert \psi_\text{out} - \psi_\text{in} \rvert^2}_{\mathclap{\leq 4\text{ by def. of $\psi_\text{in}$ and \eqref{eq:defh}}}} + \bigl((1-\eta_L)\partial_{x_3} \psi_\text{in} + \eta_L (-\cos\theta) \bigr)^2\\
    &\leq  C (\tfrac{d}{dt} \eta_L)^2 + (1-\eta_L) +\eta_L (\cos \theta)^2 \\
    &\stackrel{\mathclap{\eqref{eq:propcutoff}}}{\leq} 1 - \bigl(\sin^2 \theta - \tfrac{CC(\theta)}{L^2} \bigr)\eta_L\\
    &\leq 1\quad \text{if}\quad L\geq L(\theta) \text{ is sufficiently large.}
  \end{align*}
  \textbf{Step 7:} \textit{The degree of $(m_1,m_2)$.} Using $\eta_L(\lvert x_3 \rvert) = 0$ and $\psi_3=\psi_2\stackrel{\eqref{eq:psibdry}}{=}(1-\cos\theta) - \lvert x_3 + \cos\theta \rvert$ in a neighbourhood of $\partial \Omega$, we compute
  \begin{align*}
    \nabla \psi_\text{in}(x) &=  \partial_{x_3} \psi_3(x) \mathbf{e}_3 = - \sgn(x_3)\mathbf{e}_3 \quad \text{on }\partial\Omega,
  \end{align*}
  and therefore, due to $\psi_\text{out}=\psi_\text{in}$ on $\partial \Omega$
  \begin{align*}
    \nabla \psi(x) &= \eta_L(\lvert x_1\rvert) \nabla \psi_\text{out}(x) + \bigl(1-\eta_L(\lvert x_1 \rvert)\bigr) \nabla \psi_\text{in}(x)\\
    &= - \bigl(\eta_L(\lvert x_1 \rvert)\cos\theta + (1-\eta_L(\lvert x_1 \rvert)) \sgn(x_3) \bigr) \mathbf{e}_3 \quad \text{on }\partial\Omega,
  \end{align*}
  or in view of the definition of $m$ in \eqref{def_magul}:
  \begin{align*}
    m(x) = \bigl(\eta_L(\lvert x_1 \rvert)\cos\theta + (1-\eta_L(\lvert x_1 \rvert)) \sgn(x_3), \,\sgn(x_1) \sqrt{1-m_1^2(x)},\,0\bigr) \quad \text{on }\partial\Omega.
\end{align*}
Hence, $(m_1,m_2)$, as a map $\mathbb{S}^1 \to \mathbb{S}^1$, has degree~$-1$ on $\partial\Omega$, cf. Figure~\ref{fig:degofm}.
\begin{figure}
  \centering
  \begin{pspicture}(-4.5,-3)(4.5,2)
    \psline[linestyle=dashed](-3.5,-1.5)(-3.5,1.5)
    \psline[linestyle=dashed](3.5,-1.5)(3.5,1.5)
    \psline(-3.5,1.5)(3.5,1.5)
    \psline(-3.5,-1.5)(3.5,-1.5)
    \rput{-30}(-4.1,0){\pscircle(0,0){0.5}\psline[linewidth=0.07]{->}(0,0)(0.5,0)}
    \rput{-120}(-2,-2.1){\pscircle(0,0){0.5}\psline[linewidth=0.07]{->}(0,0)(0.5,0)}
    \rput{180}(0,-2.1){\pscircle(0,0){0.5}\psline[linewidth=0.07]{->}(0,0)(0.5,0)}
    \rput{120}(2,-2.1){\pscircle(0,0){0.5}\psline[linewidth=0.07]{->}(0,0)(0.5,0)}
    \rput{30}(4.1,0){\pscircle(0,0){0.5}\psline[linewidth=0.07]{->}(0,0)(0.5,0)}
    \rput{-15}(-2,2.1){\pscircle(0,0){0.5}\psline[linewidth=0.07]{->}(0,0)(0.5,0)}
    \rput{0}(0,2.1){\pscircle(0,0){0.5}\psline[linewidth=0.07]{->}(0,0)(0.5,0)}
    \rput{15}(2,2.1){\pscircle(0,0){0.5}\psline[linewidth=0.07]{->}(0,0)(0.5,0)}
    \rput(0,0){$\Omega=\R\times[-1,1]$}
    \psline(2.5,2.3)(3.4,2.9)
    \rput[l](3.5,3.1){$(m_1,m_2)\in\mathbb{S}^1$}
    \psline{->}(4.0,2.2)(4.0,2.5)\rput[b](4.0,2.55){\psscalebox{0.7}{$m_2$}}
    \psline{->}(4.0,2.2)(4.3,2.2)\rput[l](4.35,2.2){\psscalebox{0.7}{$m_1$}}
  \end{pspicture}
  \caption{$(m_1,m_2)$ on $\partial\Omega$.}
  \label{fig:degofm}
\end{figure}
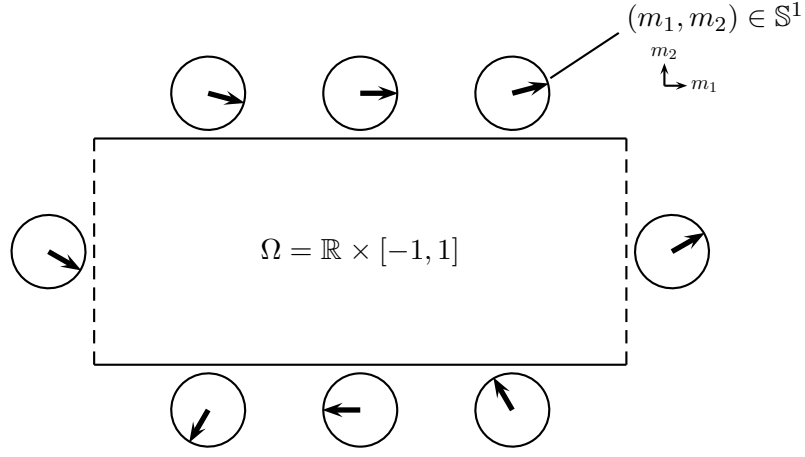
\end{proof}

In order to prove that the magnetization $m$ defined at \eqref{def_magul} belongs to $\dot{H}^1(\Omega)$ it is enough to check that
$f=1-|\nabla \psi|^2$ has the property that $\sqrt{f}$ is Lipschitz in $\R[2]$ where $\psi$ is the stream function constructed above:

\bigskip

\begin{lem}
\label{lem:m2h1}
Let $f\in C^2(\R[N],\R_+)$ be a non-negative function with $D^2 f\in L^\infty(\R[N])$. Then $\nabla\sqrt{f} \in L^\infty(\R[N])$ and we have
\begin{align}\label{eq:sqrtlip}
  \lVert\nabla \sqrt{f}\rVert_\infty^2 \leq \tfrac{1}{2} \lVert D^2 f\rVert_\infty.
\end{align}
\end{lem}

\begin{proof} We distinguish two cases:

  \textbf{Case 1:} \textit{$D^2f \equiv 0$ on $\R[N]$, i.e., $f$ is an affine function.} Since by assumption $f\geq 0$ in $\R[N]$, one has $f \equiv \text{const}$. Thus, the assertion of Lemma~\ref{lem:m2h1} becomes trivial.

  \textbf{Case 2:} \textit{$\lVert D^2f \rVert_\infty > 0$.} Let $x,x_0\in\R[N]$. Taylor's expansion yields for some intermediate $\tilde{x}\in\R[N]$:
  \begin{align}
    0 \leq f(x) &= f(x_0) + \nabla f(x_0) \cdot (x-x_0) + \tfrac{1}{2} (x-x_0) \cdot D^2f(\tilde{x}) (x-x_0)\notag\\
    &\leq \lvert f(x_0) \rvert + \nabla f(x_0) \cdot (x-x_0) + \tfrac{1}{2} \lVert D^2 f \rVert_\infty \lvert x-x_0 \rvert^2.\label{eq:taylor}
  \end{align}
  Hence, choosing $x\in\R[N]$ such that $x-x_0 = -\tfrac{\nabla f(x_0)}{\lVert D^2f \rVert_\infty}$, we obtain
  \begin{gather*}
    \tfrac{\lvert \nabla f(x_0) \rvert^2}{\lVert D^2f \rVert_\infty} \leq 2 \lvert f(x_0) \rvert,
  \end{gather*}
  i.e.
  \begin{gather}
    \lvert \nabla \sqrt{f}(x_0) \rvert \leq \tfrac{1}{\sqrt{2}}\lVert D^2 f \rVert_\infty^\frac{1}{2}\quad\text{if }f(x_0)\neq 0\label{eq:sqrtlipfn0}.
  \end{gather}

  If there exist points at which $f$ vanishes, we apply \eqref{eq:sqrtlipfn0} to $f+\varepsilon$ instead of $f$ (with $\varepsilon >0$), and deduce for $x,y\in\R[N]$
  \begin{align*}
    \lvert \sqrt{f(x)+\varepsilon}-\sqrt{f(y)+\varepsilon} \rvert &\leq \int_0^1 \lvert (\nabla \sqrt{f+\varepsilon})(tx+(1-t)y) \rvert\, \lvert x-y\rvert\, dt\\
    &\stackrel{\eqref{eq:sqrtlipfn0}}{\leq} \tfrac{1}{\sqrt{2}} \lVert D^2f\rVert_\infty^\frac{1}{2} \lvert x-y\rvert.
  \end{align*}

  Letting $\varepsilon \tod 0$ we obtain $\lvert \sqrt{f(x)}-\sqrt{f(y)}\rvert \leq \tfrac{1}{\sqrt{2}} \lVert D^2 f \rVert^\frac{1}{2}_\infty \lvert x - y \rvert$ for all $x,y\in\R[N]$, such that \eqref{eq:sqrtlip} follows.
\end{proof}

\section{Proof of Lemma~\ref{lem-topo}}
The relation between the topological degree of $(m_1,m_2)$ on $\partial\Omega$ and the vortex singularity of $(m_1,m_3)$ observed in the previous construction is studied in Lemma~\ref{lem-topo} that we prove in the following:

\begin{proof}[Proof of Lemma \ref{lem-topo}] 
  Due to $\nabla\cdot m'=0$ in $\mathcal{D}'(\R[2])$ we may represent $m'=\nabla^\perp \psi$ for a stream function $\psi \colon \R[2] \to \R$ with $\psi(x_1,-1)=0$, $\psi(x_1,1)=-2\cos\theta$ for all $x_1 \in \R$. Under the hypothesis $m\in \dot{H}^1(\Omega, \mathbb{S}^2)$, one gets that $\nabla \psi\in \dot{H}^{1}\cap L^{\infty}(\Omega)$. Since $m$ has non-zero topological degree on $\overline{Bdry}$, the set $\left\{ x\in\partial\Omega \with m_1(x) < 0 \right\}$ is non-empty (recall that $m\in \dot{H}^{1/2}(\partial \Omega, \mathbb{S}^2)$). We assume that it has non-empty intersection with $\R\times \{-1\}$. (The other case is similar.) Since $m_1 = -\partial_{x_3} \psi<0$ and $\psi=0$ on that subset of $\R\times\{-1\}$, one sees that the set $\left\{ x \in \Omega \with \psi > 0 \right\}$ is non-empty. In particular, there exists a connected component $C$ of $\left\{ x \in \Omega \with \psi > 0 \right\}$ whose boundary intersects $\R\times\{-1\}$ on a set containing an interval (see Figure \ref{fig_curb}).
  
Since $\nabla \psi \in \dot{H}^1(\Omega)$, and taking into account the boundary conditions at $x_1=\pm\infty$, the Sobolev embedding theorem on sets $\{a\leq\lvert x_1 \rvert \leq a+1\}$ yields 
$\psi(x_1,x_3)\to-(x_3+1)\cos\theta$ uniformly in $x_3$ as $\lvert x_1 \rvert \tou \infty$.
Hence, any level set $\{\psi=\varepsilon\}$ for $\varepsilon>0$ is bounded, and $\psi$ attains a maximum $x_0$ in the interior of $C\subset\Omega$. Let $\beta_0=\psi(x_0)>0$.

\begin{figure}[h]
  \centering
  \includegraphics[width=12cm]{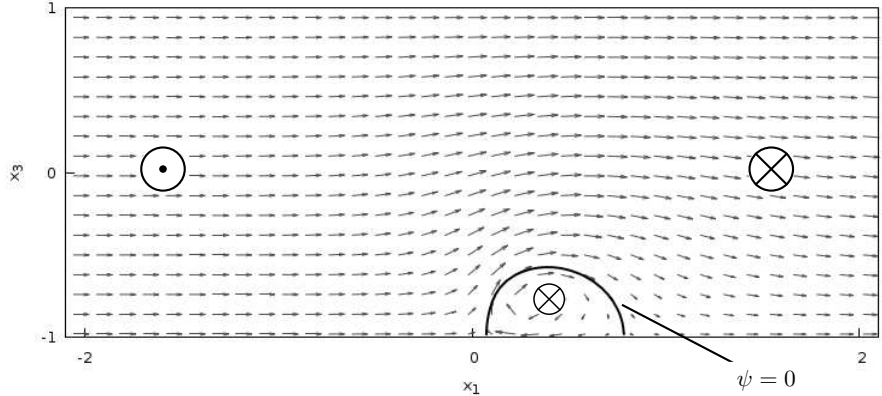}
  \rput(-2,3.3){
    \pscircle*[linecolor=white](0,0){0.31cm}
    \pscircle(0,0){0.3cm}
    \rput{45}(0,0){\psline(-0.3,0)(0.3,0)}
    \rput{-45}(0,0){\psline(-0.3,0)(0.3,0)}
  }
  \rput(-4.9,1.575){\psscalebox{0.7}{
    \pscircle*[linecolor=white](0,0){0.31cm}
    \pscircle(0,0){0.3cm}
    \rput{45}(0,0){\psline(-0.3,0)(0.3,0)}
    \rput{-45}(0,0){\psline(-0.3,0)(0.3,0)}
  }}
  \rput(-10,3.3){
    \pscircle*[linecolor=white](0,0){0.31cm}
    \pscircle(0,0){0.3cm}
    \pscircle*(0,0){0.05cm}
  }
  \psline(-3.9,1.5)(-2,0.5)\rput*(-2,0.5){\psscalebox{0.8}{$\psi=0$}}
  \caption{The zero level set of $\psi$ for an asymmetric Bloch wall of angle $\theta=0.7$ (cf. footnote in Fig.~\ref{asym}).}
\label{fig_curb}
\end{figure}

In the case of a vector field $m\in C^1(\Omega)$, so $\psi\in C^2(\Omega\subset \R[2], \R)$, by Sard's theorem, there exists a regular value $\beta\in (0, \beta_0)$ of $\psi$. In particular, there exists a smooth cycle ${\cal \gamma}\subset C$ such that $\psi\equiv \beta$ and $|\nabla \psi|>0$ on ${\cal \gamma}$. Therefore, $\nu:=\frac{\nabla \psi}{|\nabla \psi|}:{\cal \gamma}\to \mathbb{S}^1$ is a normal vector field at ${\cal \gamma}$ so that it carries a topological degree equal to $1$. Hence, $(m_1, m_3)=\nabla^\perp \psi$ presents a vortex singularity inside $\Omega$ carrying a non-zero winding number.
This argument is still valid for the general case $m\in \dot{H}^1(\Omega)$ (where $\nabla \psi \in \dot{H}^{1}(\Omega)$); in fact, Sard's theorem is valid also for $\psi\in W_{loc}^{2,1}(C)$ (see e.g., Bourgain-Korobkov-Kristensen \cite{BBK}) where we recall that $\psi>0$ on $C$ and $\psi=0$ on $\partial C$ so that almost all $\beta\in (0, \beta_0)$ is a regular value, i.e., the pre-image $\psi^{-1}(\beta)$ is a finite disjoint family of $C^1$-cycles and the normal vector field $\nu$ on each cycle is absolutely continuous, in particular, it carries a winding number $1$.
\end{proof}